\documentclass[11pt]{amsart}
\usepackage{amsmath}
\usepackage{amsthm}
\usepackage{amsbsy}
\usepackage{amsopn}
\usepackage{amsmath,amscd}
\usepackage{amssymb}
\usepackage{amsfonts}
\usepackage{multirow}
\usepackage{mathrsfs} 
\usepackage[official ]{eurosym}
\usepackage{accents}

\usepackage{graphicx}
\graphicspath{
    {../Pictures/}
}

\newcommand{\da}[2]{{\color{green!80!black} #1
}{\color{blue!90!black}#2 
}}

\newcommand\cI{\mathcal{I}}

\def\bH{\mathbf{H}}

\def\bN{\mathbf{N}}

\def\bP{\mathbf{P}}
\def\bQ{\mathbf{Q}}
\def\bR{\mathbf{R}}

\newcommand\frP{\mathfrak{P}}

\newcommand\sX{\mathscr{X}}
\newcommand\sY{\mathscr{Y}}
\newcommand\sA{\mathscr{A}}

\newcommand\sN{\mathscr{N}}
\newcommand\sM{\mathscr{M}}
\newcommand\sV{\mathscr{V}}

\newcommand\sH{\mathscr{H}}

\newcommand\sP{\mathscr{P}}
\newcommand\sQ{\mathscr{Q}}

\newcommand\sT{\mathscr{T}}
\newcommand\sR{\mathscr{R}}

\newcommand\sU{\mathscr{U}}

 \newcommand\rmint{\mathrm {int}}

\newcommand{\risom}{\stackrel{\sim}{\to}}

 \usepackage[all]{xy}
                \usepackage{amsthm,amsfonts,amsmath,amscd,amssymb,epsfig,verbatim,
}

\usepackage{graphicx}
\usepackage{epstopdf}
\date{\today} 
 
\usepackage{graphicx}

\usepackage[dvipsnames,svgnames,x11names,hyperref]{xcolor}
\usepackage{url,graphicx,verbatim,amssymb,enumerate,stmaryrd}
\usepackage[pagebackref,colorlinks,citecolor=Bittersweet,linkcolor=Bittersweet,urlcolor=Bittersweet,filecolor=Bittersweet]{hyperref}
\usepackage{tikz} 
\usetikzlibrary{arrows,decorations.pathmorphing,backgrounds,fit,positioning,shapes.symbols,chains,calc}
\usepackage[all]{xy}
\xyoption{matrix}
\xyoption{arrow}
\usepackage[hmargin=3cm,vmargin=3cm]{geometry}
\usepackage{setspace,kantlipsum}
\tikzset{help lines/.style={step=#1cm,very thin, color=gray},
help lines/.default=.5} 

\usepackage{booktabs}

\theoremstyle{plain}
\newtheorem{theorem}{Theorem}[section]
\newtheorem{thm}[theorem]{Theorem}
\newtheorem{lemma}[theorem]{Lemma}
\newtheorem{proposition}[theorem]{Proposition}
\newtheorem{prop}[theorem]{Proposition}

\newtheorem{cor}[theorem]{Corollary}

\theoremstyle{definition}
\newtheorem{definition}[theorem]{Definition}
\newtheorem{inductive step}[theorem]{Inductive step}

\newtheorem{inductive lemma}[theorem]{Inductive Lemma}
\theoremstyle{remark}
\newtheorem{example}[theorem]{Example}
\newtheorem{remark}[theorem]{Remark}
\newtheorem*{remark*}{Remark}

\newtheorem*{example*}{Example}

\newcommand{\wt}{\widetilde}
\newcommand{\wh}{\widehat}
\newcommand{\ol}{\overline}

\newcommand{\p}{\partial}

\newcommand{\om}{\omega}
\newcommand{\Om}{\Omega}
\newcommand{\eps}{\varepsilon}

\newcommand{\Skel}{\mathrm {Skel}}
\newcommand{\Cone}{\mathrm {Cone}}

\newcommand{\R}{{\mathbf{R}}}
\newcommand{\C}{{\mathbf{C}}}

\renewcommand{\P}{{\mathbf{P}}}
\newcommand{\Int}{{\rm Int\,}} 

\newcommand{\dist}{{\rm dist}}

\newcommand{\Id}{\mathrm {Id}}

\newcommand{\inv}{\mathrm{inv}}

\newcommand{\attract}{\mathrm{attract}}

\newcommand{\sym}{\mathrm{sym}}

\newcommand{\RR}{\mathbb{R}}

\newcommand{\II}{\mathcal{I}}

\newcommand{\cL}{\mathcal{L}}

\def\Op{{\mathcal O}{\it p}\,}

\newcommand{\fM}{\mathfrak{M}}

\renewcommand{\t}{{\mathfrak t}}

\renewcommand{\u}{{\mathfrak u}}

\numberwithin{figure}{section}

\title{Weinstein manifolds as cotangent buildings}

\author{Daniel  \'Alvarez-Gavela}
\address{Department of Mathematics \\ Brandeis University \\ Waltham, MA, 02139}
\email{dgavela@brandeis.edu}
\thanks{DA was partially supported by NSF grant DMS-1638352 and the Simons Foundation}

\author{Yakov Eliashberg }
\address{Department of Mathematics\\Stanford University \\ Stanford, CA 94305}
\email{eliash@stanford.edu}
\thanks{YE was partially supported by NSF grant DMS-21044773. }

\author{David Nadler}
\address{Department of Mathematics\\University of California, Berkeley\\Berkeley, CA  94720-3840}
\email{nadler@math.berkeley.edu}
\thanks{DN was partially supported by NSF grant DMS-1802373.}

\begin{document}
\begin{abstract}
We introduce the framework of cotangent buildings to complement and refine 
that of Weinstein handlebodies. While 
Weinstein handlebodies are suitable for a ``bottom-up" analysis of the Weinstein structure, 
cotangent buildings also enable a ``top down" analysis. Further, 
cotangent buildings include a precise control over the interaction of any subcollection of the various building blocks, each of which is modeled on the cotangent bundle of a manifold with corners. 
Our main result is
that any Weinstein manifold is Weinstein homotopic to one admitting the structure of a cotangent building. 
 \end{abstract}

\maketitle

 \onehalfspacing
 \tableofcontents
\begin{spacing}{1.4}
 
 \section{Introduction}\label{sec:intro}
       
       \subsection{Main results}
   Let $X$ be a smooth manifold.   The notion of a handlebody presentation 
      $$X= H_k  \cup (\cdots \cup ( H_2 \cup  (H_1\cup H_0))\cdots )
     $$
     as introduced by S.~Smale in \cite{Sm61},  is a basic tool in differential topology.
   In particular it  allows one to describe how  the   topology    of the manifold with boundary
      $X_{j}:= \bigcup_{ i \leq j} H_i
 $ 
   changes when attaching the handle $H_{j+1}$. The data of a handlebody presentation is given in the form of framed embeddings of spheres $S^{i-1}$ in the boundary of $X_j$, where $i$ is the index of the handle $H_{j+1}$.
   
    \begin{figure}[h]
\includegraphics[scale=0.3]{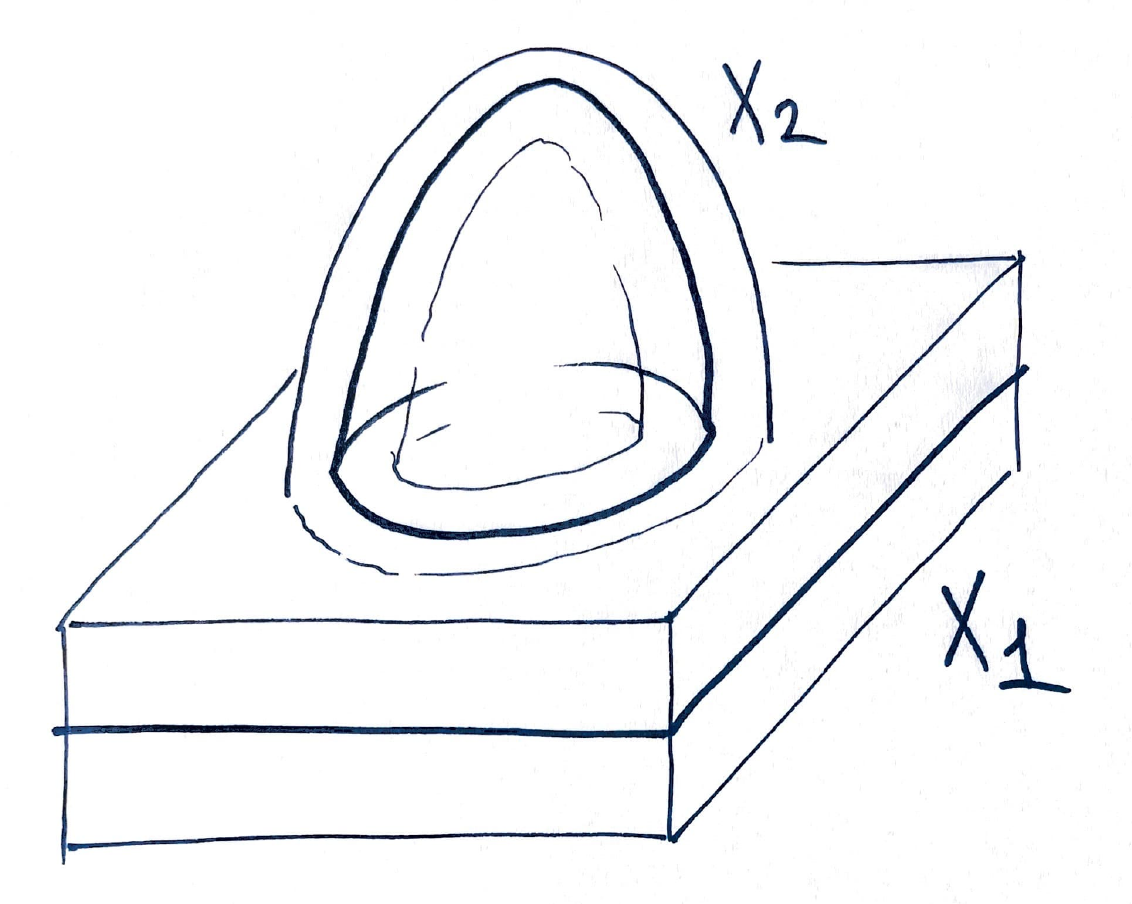}
\caption{A cotangent building with two blocks is similar to a Weinstein handlebody with two handles.}
\label{2-block}
\end{figure}

    \begin{figure}[h]
\includegraphics[scale=0.3]{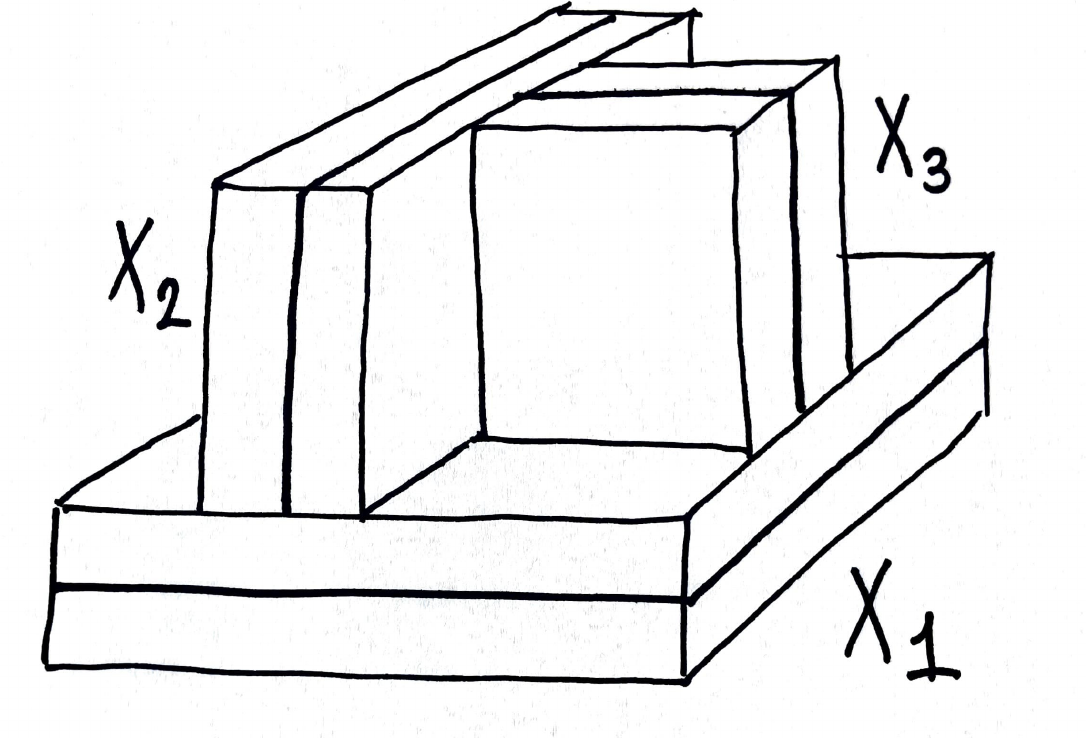}
\caption{A cotangent building with three blocks has more structure than a Weinstein handlebody with three handles, concretely there is greater control in the region where the three blocks meet.}
\label{3-block}
\end{figure}

      When $(X, \lambda)$ is a Weinstein manifold (a Liouville manifold of Weinstein type), a Weinstein handlebody presentation, as      introduced in  \cite{W91,EG91}, allows one to perform  a similar inductive analysis of the    symplectic  topology of $(X, \lambda)$.    A Weinstein handlebody presentation is given in the form of isotropic embeddings of spheres $S^{i-1}$ in the ideal contact boundary of $X_j$, equipped with a framing of their symplectic normal bundles. 
      
       Any Weinstein handle has index $i \leq n$ where $\dim X = 2n$, and in the critical dimension $i=n$, when  
       $H_{j+1} \simeq T^*D^{n}$,
       the attachment is uniquely determined by the Legendrian sphere $S^{n-1} = \p D^n$ inside  the  ideal contact boundary of $X_j$.  Note that the attachment of the handle $H_{j+1}$, and in particular the Legendrian sphere $S^{n-1} = \p D^n$ inside 
        the  ideal contact boundary of $X_j$,  is agnostic to the prior  handlebody presentation of $X_j$. Thus the procedure is  poorly tuned to  understanding the interaction of multiple handles and   also to an inductive analysis of $X$ from the ``top down", i.e.~ a description of  
              $X_{\geq j }:= \bigcup_{i \geq j} H_i$
in terms of $X_{\geq j+1}$  and  the handle $H_{j}$.      

   In this article 
 we present a refinement   of the handlebody  presentation of a Weinstein manifold  $(X, \lambda)$ by  what we call a {\em cotangent  building structure}, denoted  
     $$X=(X_k \to X_{k-1} \to \cdots \to X_1).
     $$
A cotangent building is built out of {\em cotangent blocks}, which are essentially cotangent bundles   $X_j = T^*M_j$ 
     of manifolds with corners $M_j$ with some boundary faces stripped (so for example allowing for subcritical blocks). The precise definition of a cotangent  block is given in Section~\ref{sec:erasing-face} and that of cotangent building in Section \ref{sec:def-W-bldg}.

The key additional structure present in the notion of a cotangent  building  versus a handlebody  presentation is the control we have over the interaction of any subcollection of the cotangent blocks. A cotangent  building with two cotangent  blocks is a modest generalization of a handlebody presentation with two handles (since the zero-sections of cotangent blocks may have  corners). 
But already a cotangent  building with three cotangent  blocks contains significantly more structure than a handlebody presentation with three handles, specifically in the region where the three blocks overlap. More precisely, in a 3-story building
 $(X_3=T^*M_3)\to (X_2=T^*M_2)\to (X_1=T^*M_1)$  
 the block $X_2$ 
 is attached to the contact boundary of $X_1$ along a neighborhood $U$
of  the Legendrian $\p M_2$ which can be identified with $\{p^2+z^2\leq  \eps\}\subset (J^1(\p M_2), dz-pdq).$
The boundary $S=\p U=\{p^2+z^2= \eps\}$ is contained in the contact boundary of the $2$-story building $(X_2\to X_1)$. The third block $X_3$ is attached to $ (X_2\to X_1)$ along a Legendrian $\p M_3\subset \p (X_2\to X_1)$. While in the handlebody picture it could be any Legendrian, the cotangent building condition requires  that $\p M_3$ intersects $S$ along a Legendrian in the contact manifold $\{z=0\}\subset S$, which is the dividing set of the convex hypersurface $S\subset \p (X_2\to X_1)$. When this occurs we say that the Legendrian $\partial M_3 \subset \partial (X_2 \to X_1)$ is {\em adapted} to the building structure $X_2 \to X_1$, see Figure \ref{adapted-legendrian}. 

  \begin{figure}[h!]
 \centering
  \includegraphics[scale=0.3]{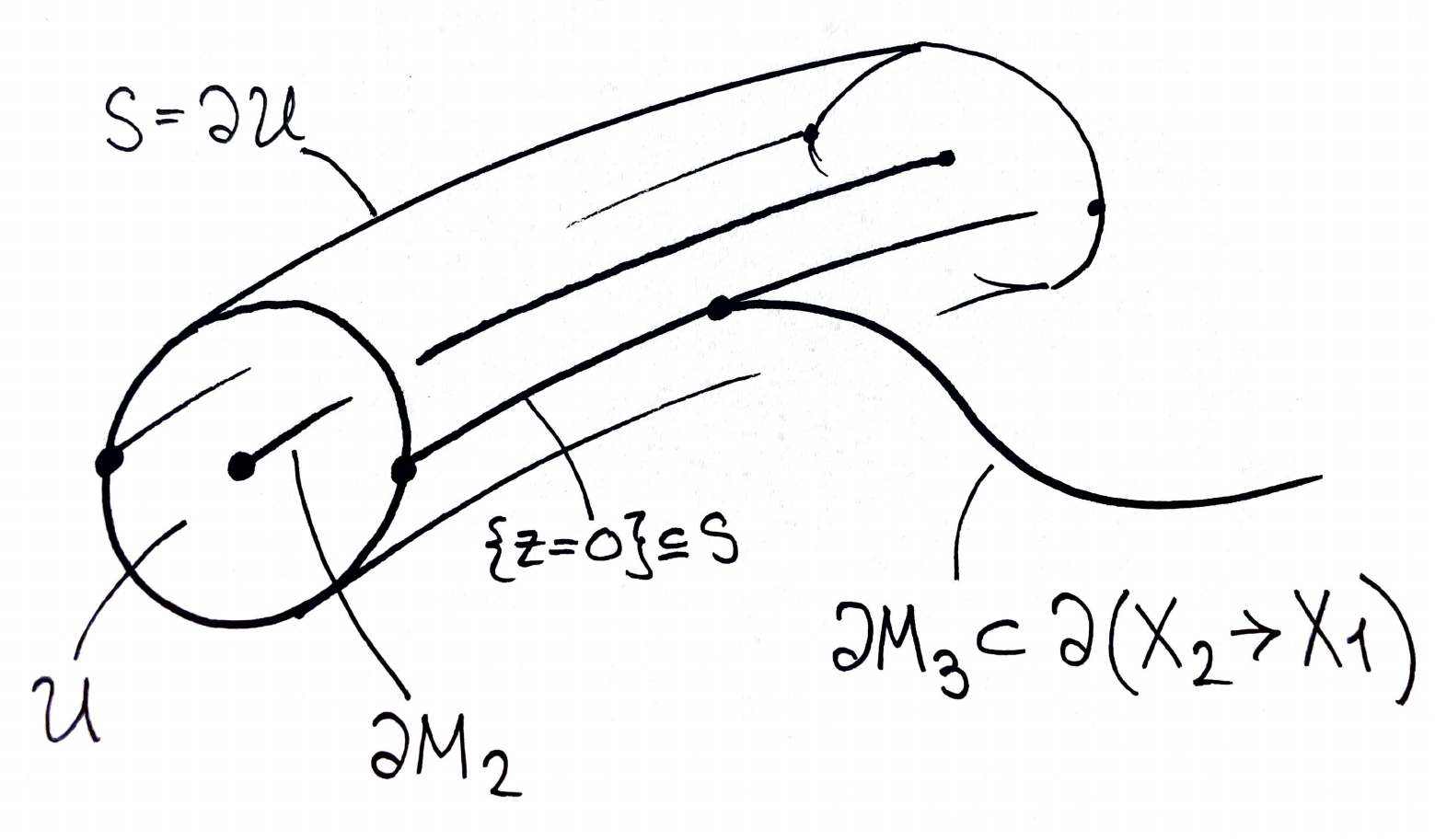}
 \caption{The geometry of an adapted Legendrian.}
 \label{adapted-legendrian}
 \end{figure}

     The main result of the article, Theorem  \ref{thm:WctoW}, establishes existence of a cotangent building structure on any Weinstein manifold.
 
 \begin{thm}
 \label{thm:intro WctoW}
  Any Weinstein manifold $(X, \lambda)$  
   is Weinstein homotopic to a Weinstein manifold
admitting a cotangent building structure 
   $$X=(X_k \to X_{k-1} \to \cdots \to X_1)
     $$
     where each  cotangent block $X_j$ is the cotangent bundle $T^*M_j$ of a disk with corners $M_j$ with some boundary faces stripped.
  \end{thm}
%
  
 


  As cotangent blocks are not Weinstein manifolds in the traditional sense (as they could be cotangent bundles of manifolds with corners), we work in a slightly more general context of {\em Weinstein blocks}, aka {\em  W-blocks}, and prove the above result in this more general setup.

  We originally arrived at cotangent buildings as a convenient formalism for our work on the arborealization program (simplification of singularities of  skeleta  of Weinstein manifolds).   For more on this application, we refer the reader to~\cite{AGEN21} where  cotangent buildings play an essential role in controlling the interaction between three or more strata of the skeleton.
  However, we believe  that  they may be useful for other applications and so have provided a  presentation in this paper independent of this specific application.

    \subsection{Summary of sections}
    In Section \ref{sec:L-man}, we review     the  notions  of Liouville and Weinstein manifolds. 
    
    In Section   \ref{sec:Wbc}, we introduce  Weinstein blocks, aka {\em W-blocks}, a variation on  Weinstein manifolds that provides useful flexibility  in two ways: W-blocks may have     boundaries and corners, and from the start,  W-blocks are defined to be germs at their skeleta.   
    
    In Section \ref{sec:operations}, we  discuss various operations on W-blocks (e.g. horizontal gluing and splitting, vertical gluing, and erasing faces and smoothing corners). We also discuss Legendrian submanifolds and  W-hypersurfaces  in  ideal boundaries of W-blocks.

In Section \ref{sec:W-bldg}, we introduce and analyze the notion of a  W-building, and in particular that of a cotangent building,  and    prove our main results concerning  existence  of cotangent building structures.


%

\subsection{Acknowledgements}

The first author thanks the Institute for Advanced Study, Princeton University, MIT and CRM Montreal  for providing a great research environment, and was supported by  NSF DMS grant \#1638352 and the Simons Foundation. The second author  thanks RIMS Kyoto and ITS ETH Zurich for their hospitality. The third author thanks SLMath for its hospitality  and acknowledges the support of NSF DMS grant \#2401178.

Finally, we are very grateful for the support of the American Institute of Mathematics, which hosted a workshop on the arborealization program in 2018 from which this project has greatly benefited.

\section{Liouville  and Weinstein manifolds}\label{sec:L-man} 
We first review  the  notions  of Liouville manifolds, domains and germs, then introduce Weinstein manifolds following  the modified formalism for potentials recently proposed by Kai Cieliebak in \cite{Ci23}. 
  
 
 \subsection{Liouville manifolds, domains  and  germs} 
 
 An {\em  exact symplectic manifold} is a symplectic manifold $(X,\om)$ with exact symplectic form together with the choice of a  primitive $\lambda$ for $\om$, i.e. a 1-form $\lambda$ such that $\om=d\lambda$. Such an $X$ cannot be closed (compact and boundaryless), and in general one needs to impose some kind of convexity or finiteness condition for $\lambda$ at infinity in order to develop a meaningful theory for $(X,\lambda)$. 
 
  \begin{definition}\label{def: Liouv man}  An exact symplectic manifold $(X,\lambda)$ is a {\em finite type Liouville manifold}, or a {\em Liouville manifold with cylindrical end}, if   
  \begin{enumerate}
  \item   the  field $Z$ which is $\omega$-dual to $\lambda$ is complete for $t\to+\infty$,
  \item   there exists a compact domain $W\subset X$ such that $Z$ is outward  transverse  to the boundary $\p W$,
  \item  the union of forward trajectories of $Z$ starting at $\p W$ is equal to $X\setminus\Int W$.
  \end{enumerate} 
  \end{definition}
  
  \begin{remark} The above condition automatically  implies completeness of $Z$ as $t\to-\infty$. \end{remark}

The $1$-form $\lambda$ is called a {\em Liouville form}, and the $\om$-dual vector field $Z$, i.e.~determined by $\iota_Z\om=\lambda$, is called the {\em Liouville field}. An embedding between exact symplectic manifolds which relates the chosen primitives by pullback is called a {\em Liouville embedding}. 
 It will be convenient for us to abuse  terminology and use the adjective {\em Liouville} to modify the nouns {\em form, field} and {\em embedding} even when $(X,\lambda)$ is not a Liouville manifold, but merely an exact symplectic manifold. All Liouville manifolds considered in this paper will be of finite type, and hence we will usually abbreviate them to just {\em Liouville manifolds}. 

\begin{definition}\label{def: Liouv dom}
An exact symplectic compact manifold with boundary $(W,\lambda)$ is a {\em Liouville domain} if the field $Z$ which is $\omega$-dual to $\lambda$ is transverse to $\partial W$ and points outside of $W$. 
\end{definition}

\begin{remark} Equivalently, we may ask that the Liouville form $\lambda$ restricts on $\partial W$ to a contact form which is compatible with the orientation $\partial W$ inherits as the boundary of $W$, which is oriented by the top exterior power of $d \lambda$. \end{remark}

\begin{definition}\label{def: dom} Let $(X,\lambda)$ be a Liouville manifold. Any domain $W\subset X$ as in Definition \ref{def: Liouv man} is called a {\em defining Liouville domain} for $(X,\lambda)$. The pair $(W,\lambda|_W)$ is a Liouville domain in the sense of Definition \ref{def: Liouv dom}.  
\end{definition}

Any Liouville domain $(W,\lambda)$ can be completed to a Liouville manifold $$\wh W = W \cup_{\partial W} (\partial W \times [0,\infty) )$$ where $\lambda$ is extended by $e^t \left(\lambda|_{\partial W}\right)$ on $\partial W \times [0,\infty)$, and the smooth gluing is achieved by the collar $\partial W \times (-\varepsilon, 0]$ for $\partial W$ in $W$ in which $Z=\partial / \partial t$, so for each $x \in \partial W$ the curve $t \mapsto (x,t)$ is a trajectory of the Liouville field. Note that by construction the Liouville domain $W$ we started with is a defining domain for the Liouville manifold $\wh W$. 

If a Liouville domain $W$ is a defining domain for a Liouville manifold $X$, there is a canonical Liouville isomorphism $\wh W \to X$ which is the identity on $W$, by sending Liouville trajectories to Liouville trajectories.

The subset  $\bigcap_{t>0}Z^{-t}(W) \subset X$ is independent of a choice of the  defining Liouville domain, and it is called the  {\em skeleton} of  the Liouville manifold $X$ (and  of the Liouville domain $W$). We will denote it by  $\Skel(X,\lambda)$ (or $\Skel(W,\lambda)$).
  Equivalenty, $\Skel(X,\lambda)$ is the attractor of the 
positive flow of the contracting field $-Z$ for any defining   Liouville domain $W$. 
For instance, for the standard Liouville structure $ pdq $ on the cotangent bundle $T^*M$,  the skeleton $\Skel(T^*M,pdq)$ is the $0$-section $M\subset T^*M$.

When $W_1,W_2 $ are two defining domains for a Liouville manifold $(X,\lambda)$ and $W_1 \subset W_2$, the flow of $Z$ gives a trivialization of the collar $W_2 \setminus \Int W_1$ and also a canonical Liouville isomorphism $\wh W_1 \to \wh W_2$. 
In this case, we also say $W_1$ is a defining domain for the Liouville domain $(W_2, \lambda|_{W_2})$. 
Note that given any two defining domains $W_1$ and $W_2$ for $(X,\lambda)$, there exists a defining domain $W_3$ for $(X,\lambda)$ such that $W_3 \subset W_1\cap W_2$.  Hence we get canonical Liouville isomorphisms, fixed on $W_3$, fitting in a strictly commutative diagram
\[ \xymatrix{ \wh W_3 \ar[d] \ar[r] & \wh W_1 \ar[d] \\  \wh W_2 \ar[r] & X} \]


\begin{definition}
Fix a smooth manifold $X$. A {\em Liouville germ} $(\sX,\lambda)$ in $X$ consists of an equivalence class of Liouville domains $(W,\lambda)$ with $W \subset X$  where $(W_1,\lambda_1) \sim (W_2,\lambda_2)$ if there exists a defining domain $W \subset W_1 \cap W_2$ for both $(W_1,\lambda_1)$ and $(W_2,\lambda_2)$ such that $\lambda_1|_{W}=\lambda_2|_{W}$. Any representative $(W,\lambda)$ of the equivalence class will be called a defining domain for the Liouville germ $(\sX,\lambda)$. 
\end{definition} 

From the above discussion it follows that the equivalence relation is generated by the identification of any two Liouville domains $(W_1,\lambda_1)$ and $(W_2,\lambda_2)$ in $X$ such that $W_1 \subset W_2$ and $\lambda_2 |_{W_1} = \lambda_1$.

An isomorphism between Liouville germs $(\sX_1,\lambda_1) \to ( \sX_2, \lambda_2)$ consists of the equivalence class of a Liouville isomorphism $f:(W_1,\lambda_1) \to ( W_2,\lambda_2)$ between defining domains $W_i$ for $(\sX_i,\lambda_i)$, $i=1,2$, where two such maps are equivalent if they agree when restricted to some smaller defining domain of $\sX_1$. Isomorphisms between Liouville germs will always be understood as equivalence classes up to this relation.

Any Liouville manifold $(X,\lambda)$ determines a Liouville germ $(\sX,\lambda)$ consisting of the equivalence class of any defining domain $W \subset X$. Similarly, any Liouville domain $(W,\lambda)$ determines a Liouville germ. Note that any two defining domains for the same Liouville manifold have canonically isomorphic Liouville germs and the germ of a defining domain $W$ for a Liouville manifold $X$ is canonically isomorphic to the Liouville germ of the Liouville manifold $X$ itself. We will henceforth gloss over such technical distinctions and will implicitly identify canonically isomorphic germs.

We will typically indicate the passage from manifold to germ by the change in font $X\rightsquigarrow \sX$ and will omit the background manifold from the notation altogether. In particular, the  germ of a cotangent bundle $(T^*M, pdq)$ along its $0$-section will be denoted by   $(\sT^*M, pdq)$. 

Note that any two defining domains for a Liouville germ $(\sX,\lambda)$ share the same skeleton, which we may therefore denote by $\Skel(\sX,\lambda)$. We may think of a Liouville germ $(\sX,\lambda)$ as the germ of a Liouville manifold at its skeleton, and
 two Liouville germs $(\sX_1,\lambda_1)$ and $(\sX_2,\lambda_2)$ are isomorphic if there exists the germ of a diffeomorphism $\Op(K_1)   \to \Op(K_2)$ of the background manifolds $X_i$ at their skeleta $K_i = \Skel(\sX_i,\lambda_i) \subset X_i$
which sends one equivalence class to the other. 
Constructions which are carried out at the level of Liouville germs are agnostic about what happens outside of an arbitrarily small neighborhood of the skeleton. This additional flexibility allows for some statements and conclusions to be formulated more efficiently. 

  \begin{figure}[h!]
 \centering
  \includegraphics[scale=0.3]{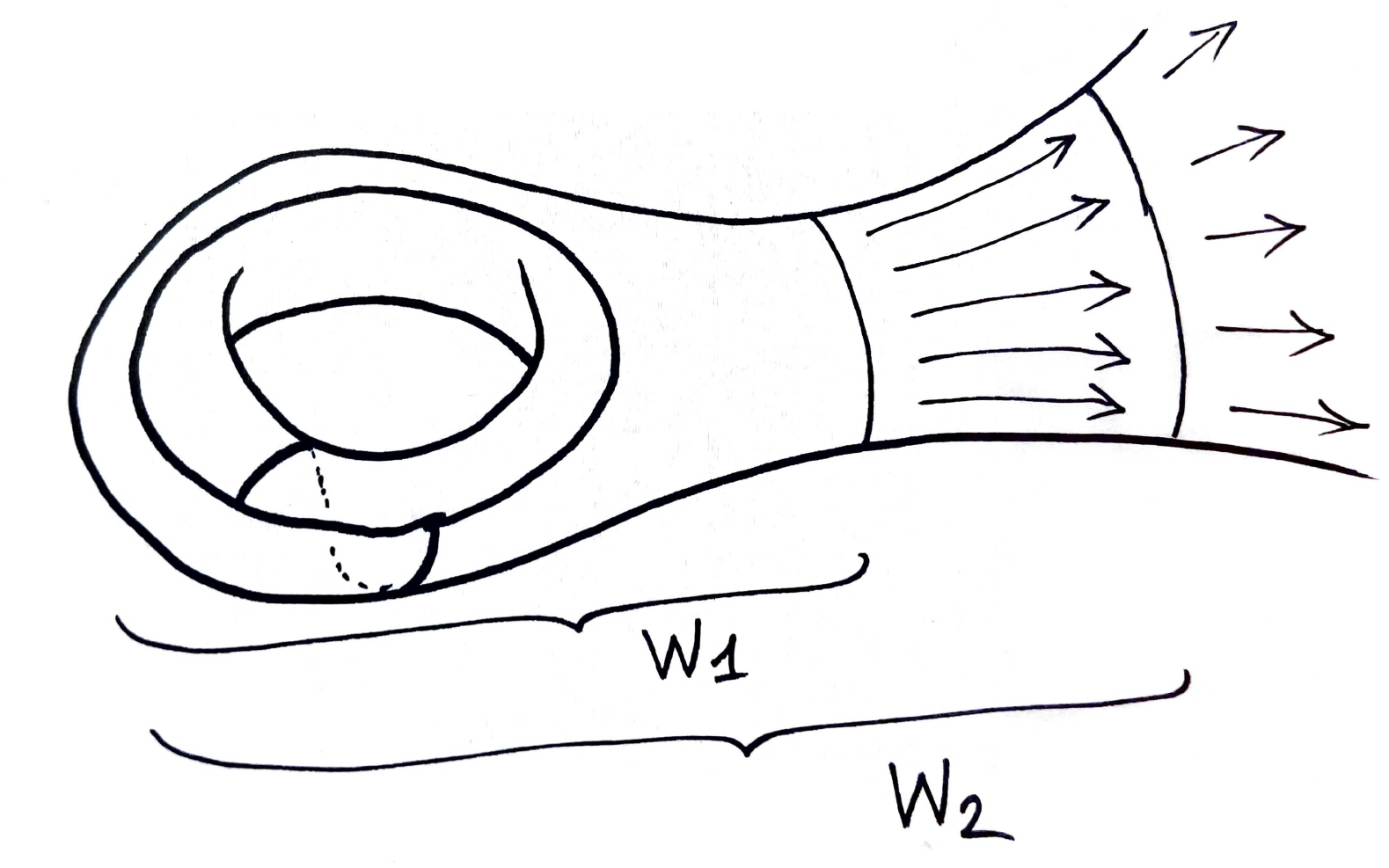}
 \caption{Two defining domains for the same Liouville manifold.}
 \label{defining-domains}
 \end{figure}

\begin{definition} The {\em ideal boundary}  of a finite type Liouville manifold $(X,\lambda)$ is defined to be $\partial_{\infty}X = \big(X \setminus \Skel(X, \lambda) \big) / \R$, where the $\R$-action is given by the  flow of $Z$.   
\end{definition}

We denote by $\pi_\infty : X\setminus \Skel(X,\lambda)\to\p_\infty X$  the projection along the trajectories of  the vector field $Z$. Since $\lambda(Z)=0$, the form $\lambda$ descends to  $\partial_\infty X$  as  a co-oriented contact structure $\xi_\infty$, while 
$(X \setminus \Skel(X, \lambda),\lambda)$ is canonically isomorphic to the {\em symplectization} $S(\p_\infty X,\xi_\infty)$. We recall that the symplectization  $S(\p_\infty X,\xi_\infty)$ is defined as the positive conormal   $N_+( \xi_\infty)\subset T^*(\p_\infty X)$ of the co-oriented hyperplane  sub-bundle $\xi_\infty\subset T(\p_\infty X)$, endowed with the Liouville form $pdq|_{N_+(\p_\infty X)}$. A choice of a defining domain for $X$ is equivalent to the choice of a contact form  $\alpha$ for $\xi_\infty$ and hence yields an identification of $(X \setminus \Skel(X, \lambda)$ with $(\p_\infty X\times\R, e^s\alpha)$. 
   
 In the other direction, the {\em contactization} of a Liouville  manifold $(N, \lambda)$ is the manifold $M \times \R$ equipped with the contact form $\lambda-du$, where $u \in \R$. For example, the contactization of the cotangent bundle $T^*M$ is the 1-jet bundle $J^1M$ with its canonical contact structure. 
 
  Note that the ideal boundary is determined by the germ $(\sX,\lambda)$; for any defining domain $W$ we have a canonical identification between the space of Liouville trajectories in $W \setminus \Skel (W,\lambda) $ and $\p_\infty X$, as well as a canonical contactomorphism $(\partial W ,\ker(\lambda|_{\p W} ) ) \to (\p_\infty X, \xi_\infty)$. We will therefore sometimes write
 $\p_\infty\sX$ instead of $\p_\infty X$.  Note that the contact form $\alpha$ on $\p_\infty X$ corresponding to $\lambda|_{\partial W}$ is not well-defined at the level of germs, since it depends on the representative $W$.

  \begin{figure}[h!]
 \centering
  \includegraphics[scale=0.3]{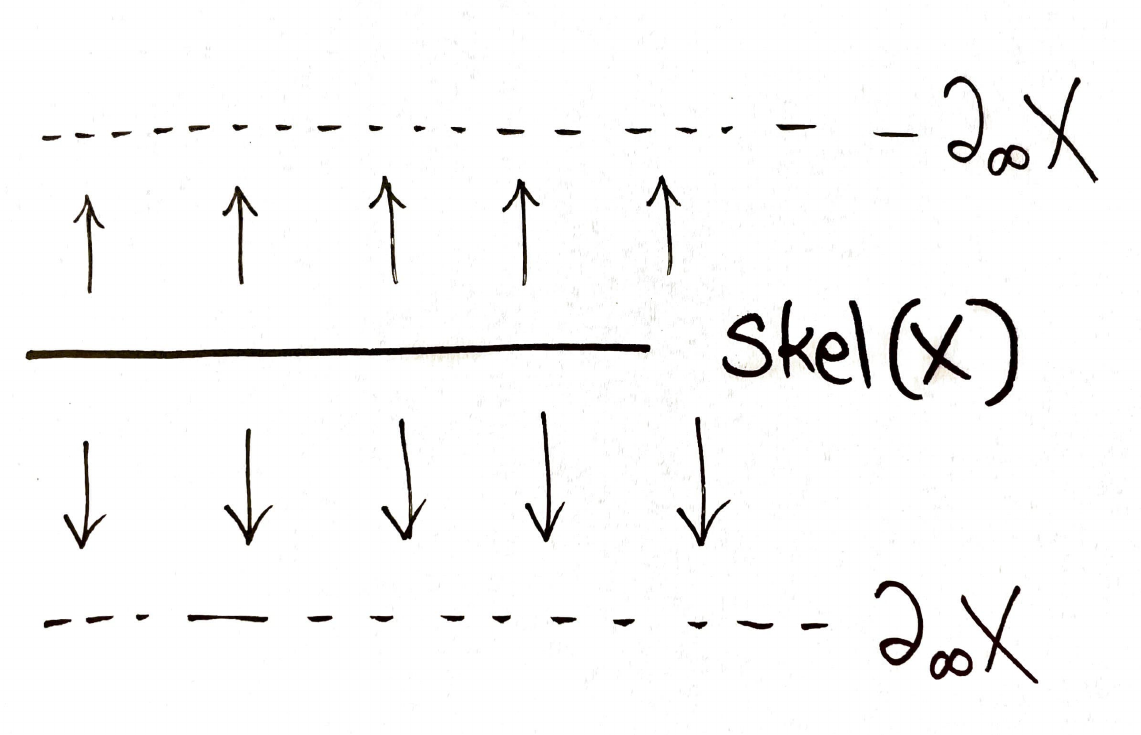}
 \caption{Ideal boundary.}
 \label{ideal-boundary}
 \end{figure}


\subsection{Weinstein manifolds after Cieliebak}\label{sec:W-man}


The notion of a {\em Weinstein manifold} was first  introduced in \cite{EG91},  building on \cite{E90} and \cite{W91}. 

Traditionally, a Weinstein manifold is a Liouville manifold $(X,\lambda)$ for which the Liouville field $Z$ admits a Lyapunov function $\phi:X\to\R$.
Recall that a function $\phi$ is called {\em Lyapunov} for a vector field $Z$ if   
 $d\phi(Z)\geq \delta (\|Z\|^2+ \|d \phi \|^2)$ for  some positive function $\delta >0$ and a choice of a  Riemannian metric on $X$. A problem with this definition as stated is that  it is unknown whether a  Liouville field admitting a Lyapunov function can be perturbed to  a generic Liouville field,   still admitting a Lyapunov function.  Hence, one  typically also  assumes  that the  zeroes of $Z$  are suitably generic  (e.g. they are non-degenerate, or birth-death, or Morse-Bott type, etc., see \cite{CE12}). 
 
 In this paper, we adopt  a modified definition, which was recently proposed by  K. Cieliebak in \cite{Ci23}, that does not encounter  this difficulty.

\subsubsection{Near-metrics, near-gradients, and potentials}\label{sec:near-grad}
For a bilinear  form $g$, let us denote by $g^{\sym}$ its symmetrization $g^{\sym}(u,v):=\frac{1}{2}(g(u,v)+g(v,u)).$ We call $g$ positive if $g^{\sym}$ is positive definite. 

 \begin{definition}  Let $X$ be a smooth manifold.
\begin{enumerate} 
\item  A field of positive  bilinear forms  $g$ on $X$ is  called a {\em near-metric}.
 \item Given a   near-metric  $g$ on $X$, a vector field $Z$ is called a {\em near-gradient} of a function $\phi:X\to\R$ if we have $d\phi=\iota_Z g$.
\item  Conversely, a function $\phi$ is called a    {\em potential} of a vector field $Z$ if there exists a near-metric $g$ such that  $Z$ is the near-gradient of $\phi$ with respect to $g$.
\end{enumerate}
\end{definition} 
\begin{remark}\label{rem:pot-inv} 
Let $g$ be a near metric on $X$, $\phi:X \to \bR$ a function and $f:\wt X\to  X$ a diffeomorphism. If a vector field $Z$ on $X$ is a near-gradient of $\phi$ with respect to $g$, then the vector field 
$f^*Z=df^{-1}(Z)$ on $\wt X$ is a near-gradient of the function $\phi\circ f$ with respect to $f^*g$.
\end{remark}

  \begin{lemma}\label{lm:space-of-pot}  The  space of potentials of a vector field $Z$ is either    empty or contractible.
\end{lemma}
\begin{proof}
Indeed,   if $Z$ is the near $g_0$-gradient of  $\phi_0$ and the near $g_1$-gradient of  $\phi_1$  then it is the near $g_t$-gradient of $\phi_t$, $t\in[0,1]$, where  $g_t:=(1-t)g_0+tg_1$   and 
$\phi_t=(1-t)\phi_0+t\phi_1$.
  \end{proof}

\begin{lemma} \label{lm:pot-Lyap} A potential of a vector-field $Z$ is a Lyapunov function of $Z$.
\end{lemma}
\begin{proof} Let us choose $g^{\sym}$ as the background Riemannian metric.
We have $|d\phi(Z)|=g(Z,Z)=\|Z\|^2$. Let $I_g, I_{g^{\sym}}$ be  isomorphisms $TX\to T^*X$ defined by the formulas $V \mapsto\iota_V g$,  $V\mapsto\iota_V g^{\sym} $ respectively.
Then $\|d\phi\|=\|I_{g^{\sym}}^{-1}(d\phi)\|=\| I_{g^{\sym}}^{-1}(I_g(Z))\|\leq C\|Z\|$ for some $C>0$, and hence, $|d\phi(Z)|\geq \delta(\|Z\|^2+\|d\phi\|^2)$.
\end{proof}

 Away from  a neighborhood of   the zeros of $Z$,  any Lyapunov function $\phi$  is a potential, and in fact, even the gradient for a genuine metric  $g$  such that  $Z$ is $g$-orthogonal to the level sets $\{\phi=a\}$ and $\|Z\|^2=d\phi(Z)$. Hence, a potential can be equivalently defined as a function $\phi:X\to\R$ which is a potential on a neighborhood of the set of zeros of $Z$ and which satisfies the condition $d\phi(Z)>0$ elsewhere.

The  converse  to the statement  of Lemma \ref{lm:pot-Lyap}   holds
  for  the most important Lyapunov functions for our concerns.

Recall that a function $\phi$ on a manifold of dimension $m$ is called: 
\begin{enumerate}
\item
   {\em Morse} if all its critical points are non-degenerate, 
   so near each critical point the function $\phi$ is equivalent to $$\sum\limits_{j=1}^{m} \pm x_j^2;$$ 
  
   \item
  {\em Morse-Bott}  if near each critical point the function $\phi$ is equivalent to $$\sum\limits_{j=1}^{k} \pm x_j^2$$ for 
   some $1\leq k\leq m$;
   \item
   {\em generalized Morse} or  {\em Igusa}   if all  its  critical points are either non-degenerate or embryos,  so near each critical point the function $\phi$ is equivalent to $$\sum\limits_{j=1}^{m} \pm x_j^2 \quad \text{or} \quad \sum\limits_{j = 1}^{m-1} \pm x_j^2+x_m^3.$$
   
 \end{enumerate}
 
 Note that  the  Morse condition for a Lyapunov function $\phi$ is equivalent to the condition that  each zero $p$ of $Z$ is isolated, non-degenerate and hyperbolic (i.e.  $d_pZ$ has no pure imaginary eigenvalues).   The embryo condition at a point $p$ implies that  non-zero eigenvalues of $d_pZ$  have non-vanishing  real part,  while $\dim(K:= \ker d_pZ)=1$  and the second differential  $d^2_pZ|_K$ does not vanish.
 We note that  $d^2_pZ|_K$ is defined up to a non-zero factor.

\begin{lemma}[Cieliebak, \cite{Ci23}] \label{lm:lm:Lyap-pot}
If  $\phi$ is  a Lyapunov function for $Z$  and $\phi$ is Morse, Morse-Bott,   or  Igusa, then there exists a near-metric $g$ such that $Z$ is the near-gradient for  $\phi$.
\end{lemma}
 This statement
 is proved in \cite{Ci23} for Morse and Igusa. The Morse-Bott case is similar, adding parameters.

\begin{example}\label{ex:cot-bundle-potential}
Let $X=T^*M$ be a cotangent bundle, $Z=p\frac{\p}{\p p}$, and $Q:T^*M\to\R$ be any Riemannian metric on the bundle $T^*M\to M$. Then $Q$ is a potential for $Z$.
\end{example}

\subsubsection{Weinstein manifolds}

\begin{definition}
We will say a    (finite type) Liouville manifold  $(X,\lambda)$    is  {\em Weinstein} if 
  the corresponding Liouville field
$Z$ admits a potential function $\phi$, in other words, there exists a near-metric $g$ such that  $\iota_Z g = d \phi$. \end{definition}


\begin{lemma}[Cieliebak, \cite{Ci23}]\label{lm:Ciel-perturb}The projection $(\lambda,g,\phi)\mapsto\phi$ is a microfibration. Concretely, suppose we are given any family
$(\lambda_s,g_s,\phi_s)$,  $s\in K$, parameterized by a   compact set $K$, such that $\phi_s$ is the $g_s$-potential for the Lyapunov field $Z_s$ corresponding to $\lambda_s$.
Then given a homotopy $\phi_{s,t}, \; s\in K,\; t\in[0,1]$, $\phi_{s,0}=\phi_s$, there exists $\eps>0$ and a covering homotopy $(\lambda_{s,t},g_{s,t})$, $\;s\in K,\;t\in[0,\eps]$,  such that $\phi_{s,t}$ is the $g_{s,t}$-potential for the Lyapunov field $Z_{s,t}$ corresponding to  $\lambda_{s,t}$,  $s\in K,\,t
\in[0,\eps]$.
\end{lemma}
 \begin{proof}
Denote $\om_s=d\lambda_s$,  and define $I^s: TV\to T^*V$  by the formula $I^s(v)=\iota_{v}\om_s$. Define a map $A_s:TX\to TX$ by the formula 
$\om_s(v,A_s Y)=g_s(v,Y),\; v,Y\in TV$. Note that $A_s$ is invertible.

Then  for the Liouville field $Z_s$ of $\lambda_s$, we have 
$$d\phi_s(Y)=g_s(Z_s,Y)=\om_s(Z_s,AY)=\lambda(A_sY),\; Y\in TV, $$
i.e. $\lambda_s=d\phi_s\circ A_s^{-1}$.
Given a homotopy $\phi_{s,t}$ of $\phi_s=\phi_{s,0}$, $s\in K$, $t\in[0,1]$, define $\lambda_{s,t}=d\phi_{s,t}\circ A_s^{-1}$. Then $\lambda_{s,t}$ is Liouville for small $t$, and if $Z_{s,t}$  denotes the corresponding Liouville field, 
i.e. $\iota_{Z_{s,t}} d\lambda_{s,t}=\lambda_{s,t}$, then  $$d\phi_{s,t}(Y)=\lambda_{s,t}(A_sY)=\om_{s,t}(Z_{s,t},A_sY) =g_{s,t}(Z_{s,t},Y),$$ where $g_{s,t}(X,Y):=d\lambda_{s,t}(X,A_sY)$. In other words, $Z_{s,t}$ is  the near-gradient for $\phi_{s,t}$ with respect to the near-metric $g_{s,t}$.
 \end{proof}
 
 \begin{cor}[Cieliebak, \cite{Ci23}]\label{cor:W-Morse} Any    Weinstein structure   can be  deformed  by an arbitrary $C^\infty$-small homotopy to a Weinstein structure admitting a Morse potential.  \end{cor}
 
 As was pointed out above, this corollary is the main reason why  in the definition of a Weinstein manifold we use   potentials rather than  a possibly larger class of Lyapunov functions.
 
 \begin{remark}\label{rem:perturb-fixing-symplectic}
 The Morsification of a  Weinstein structure  provided  by  Lemma \ref{lm:Ciel-perturb} and its Corollary \ref{cor:W-Morse}   may change the background symplectic form $\omega$ to a different though diffeomorphic symplectic form $\wt\om=\phi^*\om$ for a compactly supported   diffeomorphism $\phi:X\to X$. In most applications this does not matter. However,   if it is desirable to keep the  background symplectic structure unchanged  this can be achieved  using Remark \ref{rem:pot-inv}. 
 \end{remark}
 
\begin{example}\label{ex:complex-geom} Suppose  that  we are given a  finite type Stein complex manifold $(X,J)$, i.e. a complex manifold which admits an exhausting  $J$-convex (i.e. strictly plurisubharmonic) function $\phi$ without critical points at infinity. Here the  $J$-convexity condition means that
the  symmetric form $g_\phi(U,V):=dd^\C\phi(U,JV)$,  where $U,V\in TX$ and  $d^\C\phi(X)=d\phi(JX)$, is positive definite.
Then, as shown in  \cite{CE12}, see also  \cite{EG91}, $(X,d^\C\phi)$ is a Weinstein manifold, with the Liouville field $Z_\phi$ equal to the $g_\phi$-gradient of $\phi$.
\end{example}

 The Weinstein condition implies that $\Skel( X,\lambda)$ is the union of the $Z$-stable manifolds of   critical points of $\phi$, i.e. points converging to $\text{Crit}(\phi)$ in forward time.  
 If $\phi$ is Morse or generalized Morse it was shown   in \cite{CE12}, see also \cite{EG91, E95}, that
   $\Skel(X, \lambda)$ is the union of submanifolds which are isotropic for $\lambda$, and hence for $\om$.

\begin{remark}
Note that $\lambda$-isotropic is equivalent to $\omega$-isotropic and $Z$-conic.
\end{remark} 
 
If  $(X, \lambda)$ is a Weinstein manifold, then any defining  Liouville domain $(W, \lambda|_W)$    can be made a sub-level set $\{\phi\leq c\}$ for some choice of a potential $\phi$ which has no critical values $\geq c$.  Hence, we will refer to Liouville defining domains  for  Weinstein manifolds as Weinstein defining domains, and   {\em Liouville germs}  of Weinstein manifolds  as  {\em Weinstein germs}.   
   
\begin{example}\label{ex:cotangent}
The standard Liouville structure $\lambda=pdq$ on the cotangent bundle $T^*M$ of a closed manifold $M$ is  Weinstein. The existence of a potential for the corresponding Liouville field $Z=p\frac{\p}{\p p}$ can be verified directly, see Example \ref{ex:cot-bundle-potential}. Or it can be seen from the complex geometric construction in Example \ref{ex:complex-geom},   as $Z$  can be identified with $Z_\phi$, where $\phi$ is the squared-distance   function from $M$ in the complexification of a compatible real analytic structure,  in any  Hermitian  metric, see \cite{CE12} for the details.
 \end{example}

\subsection{Homotopies of Liouville  and Weinstein structures}
 
\begin{definition}  A {\it homotopy of (finite type) Liouville structures} $(X,\lambda_t)$ is a family $\lambda_t$, $t\in[0,1]$, of finite type Liouville structures on a manifold $X$  admitting a   smooth  family $W_t \subset X$ of defining Liouville domains.  \end{definition}

Given such a homotopy  $(X,  \lambda_t)$, there exists an isotopy $\psi_t:\ X\to  X$ such that $\psi_t^*\om_t=\om_0$, where $\omega_t=d\lambda_t$. Moreover, one can always arrange that $\psi_t^*\lambda_t=\lambda_0+dH_t$ for uniformly compactly supported  functions $H_t$, see \cite{CE12}, Sections 11.1 and 11.2. In particular, it is always sufficient to consider homotopies fixing the symplectic form, and moreover, changing the Liouville form by adding a compactly supported exact form. 

Note that for any fixed symplectic form $\omega$ the space of Liouville structures on $(X,\omega)$ which agree with a given Liouville structure at infinity is convex, hence any two such Liouville structures are canonically homotopic. 
Note also that if a symplectomorphism $f:(X_1,d\lambda_1)\to (X_2,d\lambda_2)$  satisfies $f^*\lambda_2=h\lambda_1$ for a  positive function $h$, then the Liouville structures $f^*\lambda_2$ and $\lambda_1$ on $X_1$ are Liouville homotopic, see \cite{CE12}.

\begin{definition}  A {\it homotopy of (finite type) Weinstein structures} $(X,\lambda_t)$ is a homotopy of (finite type) Liouville structures admitting
  a smooth family of potentials $\phi_t$.  In other words, there exists a smooth family of near-metrics $g_t$ such that  $\iota_{Z_t} g_t = d \phi_t$. \end{definition}

\begin{remark} It is not known whether two Weinstein type Liouville forms which are Liouville homotopic are Weinstein homotopic, i.e. whether  there exists a homotopy $\wt Z_t$ connecting $Z_0$ and $Z_1$ which  admits a family of potentials. \end{remark}

We say that a homotopy of Liouville or Weinstein structures is {\em strong} if the skeleton is constant throughout the homotopy.  We say that two Liouville or Weinstein manifolds $(X_1,\lambda_1)$ and $(X_2,\lambda_2)$ are {\em isomorphic} if there exists a Liouville isomorphism $f:X_1 \to X_2$, $f^*\lambda_2=\lambda_1$ between them, and {\em deformation equivalent} (resp. {\em strongly deformation equivalent}) if there exists a diffeomorphism $f:X_1 \to X_2$ such that $(X_1,\lambda_1)$ and $(X_1,f^*\lambda_2)$ are homotopic (resp. strongly homotopic). 

One may formulate similar notions at the level of Weinstein (or Liouville) germs. We will do so in Section \ref{sec:block-homotopy} below for the more general notion of W-blocks.


 \section{W-blocks}\label{sec:Wbc}

\subsection{Manifolds with corners}\label{ssec:Mbc} 

\subsubsection{Faces and facets}
Recall that an  $n$-dimensional smooth manifold $M$ with corners is (Hausdorff, paracompact, and) locally smoothly modeled on open subsets of $\cI^n \subset \R^n$, 
where $\cI$ denotes the  interval $[0,\eps)$ with an unspecified small $\eps>0$.
When we discuss smooth objects on $M$ such as functions or differential forms, we always mean objects that locally extend from  an open in $\cI^n $ to an object on an open in $\R^n$ which is smooth in the traditional sense.

Given an $n$-dimensional smooth manifold $M$, and a fixed $k\leq n$, 
a point $x \in M$  is called a corner of order $k$ if there is a neighborhood of $x$ in $M$ diffeomorphic to a neighborhood of the origin in $\cI^k \times \R^{n-k}$ such that $x$ corresponds to the origin.  
  We denote the locus of order $k$ corners by $\partial_k M$, and note that it is an $(n-k)$-dimensional smooth manifold.

Each connected component $V$ of  $\p_kM$ is naturally the interior of an $(n-k)$-dimensional manifold with corners $P$ such that the inclusion $j:V \hookrightarrow M$ extends to an immersion $\ol j:P\to M$ with image $\ol j(P)$  the closure of $V$. 

\begin{definition} We call each such $P$ a {\em   $k$-face} of $M$, and will also call a 1-face of $M$ a {\em facet}. \end{definition}
 
 \subsubsection{Self-biting}
It is convenient to place a  modest  global condition on the closure of each component of $\partial_kM$ within $M$. We say $M$  is {\em self-biting} when for some $k$-face $V$ the immersion  $\ol j:P\to M$ fails to be an embedding. If $M$ is not self-biting, we identify the $k$-face $P$ with its image  
 $\ol j(P)$.
 For a manifold with corners $M$ which is not self-biting, each $k$-face $P$ is a manifold with corners which is also not self-biting.
 
In what follows, {\em we will assume all manifolds  with corners are not self-biting, unless explicitly stated otherwise}.

 \begin{figure}[h]
\includegraphics[scale=0.12]{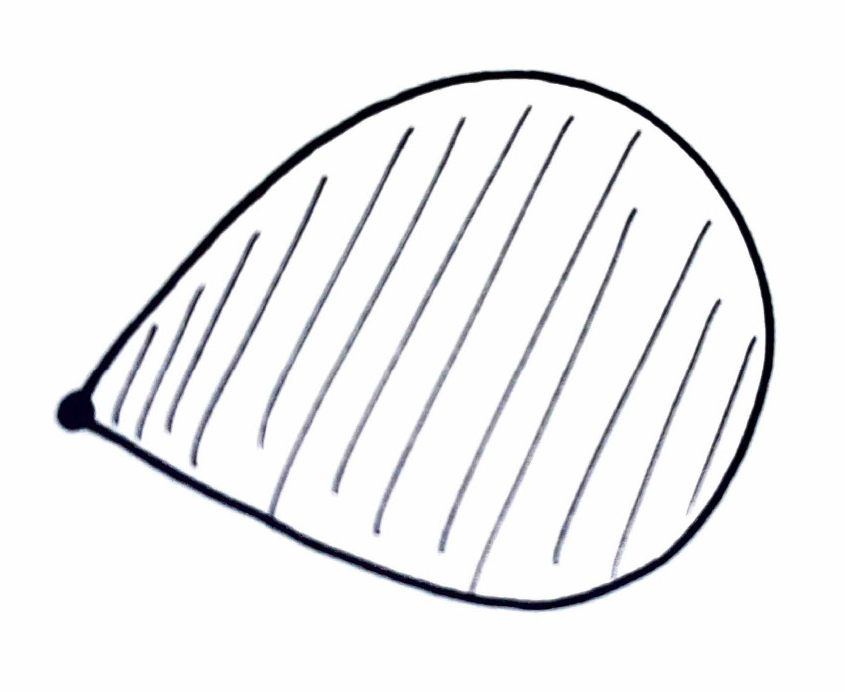}
\caption{A 2-dimensional manifold with corners $M$ with a self-biting face. Indeed, there's a single $1$-face $P$, which is diffeomorphic to a closed interval $[0,1]$, and the immersion $P \to M$ fails to be an embedding because both boundary points of the interval get sent to the same point, namely the unique order 2 corner of the boundary.}
\label{fig:collarstr}
\end{figure}

%
%


\subsubsection{Corner structure}

We will  assume our manifolds  with corners $M$  come equipped with a {\it corner structure},
 in the  sense of a compatible system of collars as follows. 
 
 First, 
  we require that
each  $k$-face  $P \subset M$  comes equipped with   the germ of an embedded split collar neighborhood $U_P = P \times \cI^k \subset M$. Equivalently,  near each point $x \in P$ we have canonical  collar coordinates $x=(y,t)$, where $y \in P$ and $t=(t_1, \ldots , t_k) \in \cI^k$. Note $P$ is cut out by $t_1 = \cdots = t_k=0$, and more generally, for $j\leq k$, each $j$-face $Q\subset M$ intersecting $U_P$  is cut out by setting exactly $j$ of the coordinates $t_i$ equal to zero. 
We require the  compatibility that, up to the reordering, these $j$ coordinates $t_i$ agree with the collar structure for $Q\subset M$. We only care about the germ of the collar $U_P$ at $P$, in other words $\cI=[0,\varepsilon)$ for $\varepsilon>0$ unspecified and arbitrarily small. We will sometimes refer to the coordinate $t_j$ on the collar $U_P$ of a $k$-face $P$ as a {\em distance} to the facet $\{t_j=0\}$.

\subsubsection{Smoothing corners} \label{sec:smooth-corners}

Let $M$ be a smooth manifold with corners,  equipped   with a corner structure. A list  (i.e.~ordered set) $\bP=[P_1,\dots,P_k]$ of facets of $M$   is called {\em admissible} if
  any two facets in the list which share a common $k$-face $Q$  also   share a common 2-face adjacent to  $Q$.
 Given an admissible list $\bP$
  we can form a new manifold with corners $M^{\bP, \frown} \subset M$, equipped with corner structure, by the following inductive procedure. Begin with faces $P_1$ and $P_2$.  Note that  the non-self-biting condition for faces of $M$ implies that    $Q:=P_1\cap P_2$ is a  union of disjoint  2-faces, and hence, can be viewed as one, possibly disconnected 2-face.   Let $t_1,t_2$ be collar coordinates on $\Op Q$\footnote{We use in this paper Gromov's notation $\Op$ for an unspecified  neighborhood of a set, in line with our emphasis on germs of various  sets} such that $P_j\cap\Op Q=\{t_j=0\}$, and $M\cap\Op Q=\{0\leq t_1,t_2<\sigma\}$ for a sufficiently small $\sigma>0$. Choose a    positive $C^\infty$-smoothing $\theta(s)$ of the function $|s|$ on $[-\frac\sigma2,\frac\sigma2]$ which is equal to $|s|$ for $|s|>\frac\sigma4$ and replace  $M\cap\Op Q$ with  
 $t_1+t_2\geq \theta(t_1-t_2),  t_1,t_2<\sigma$.  As a result  we get a new manifold with corners $M^{\frown;\{P_1,P_2\}}$ where we smoothed the corner along the 2-face $Q$ and replaced facets $P_1, P_2$ by a facet $P^\frown_{12}$ which is a result of gluing  $P_1$ and $P_2$ along $Q$. The new collar coordinate is $\frac{1}{2}(t_1+t_2-\theta(t_1-t_2))$ which away from $\Op(Q)$ is either $t_1$ or $t_2$. We emphasize that the non-self-biting condition still holds for $M^{\frown;\{P_1,P_2\}}$. 
 
 We repeat that procedure again, smoothing the corner of $M^{\frown;\{P_1,P_2\}}$ along the 2-face $P^\frown_{12}\cap P_3$, etc. until all facets from the list $\bP=[P_1,\dots,P_k]$  are exhausted.
 
  The resulting manifold with corners     $M^{\frown;\bP}$, has a  new smooth (possibly disconnected) facet $P^\frown_{1\dots k}$ replacing  the facets $P_1,\dots, P_k$.  Note that up to an isotopy of $M$ supported near its boundary, $M^{\frown;\bP}$ only depends on the underlying set $\bP$, and not its ordering. 
In particular, if  $\bP$ is a list of all facets of $M$, then  $M^\frown:=M^{ \frown,\bP}$ is  a smooth manifold with  smooth boundary, and  $P_{1\dots k}^\frown=\p M^\frown$.

  \begin{figure}[h!]
 \centering
  \includegraphics[scale=0.3]{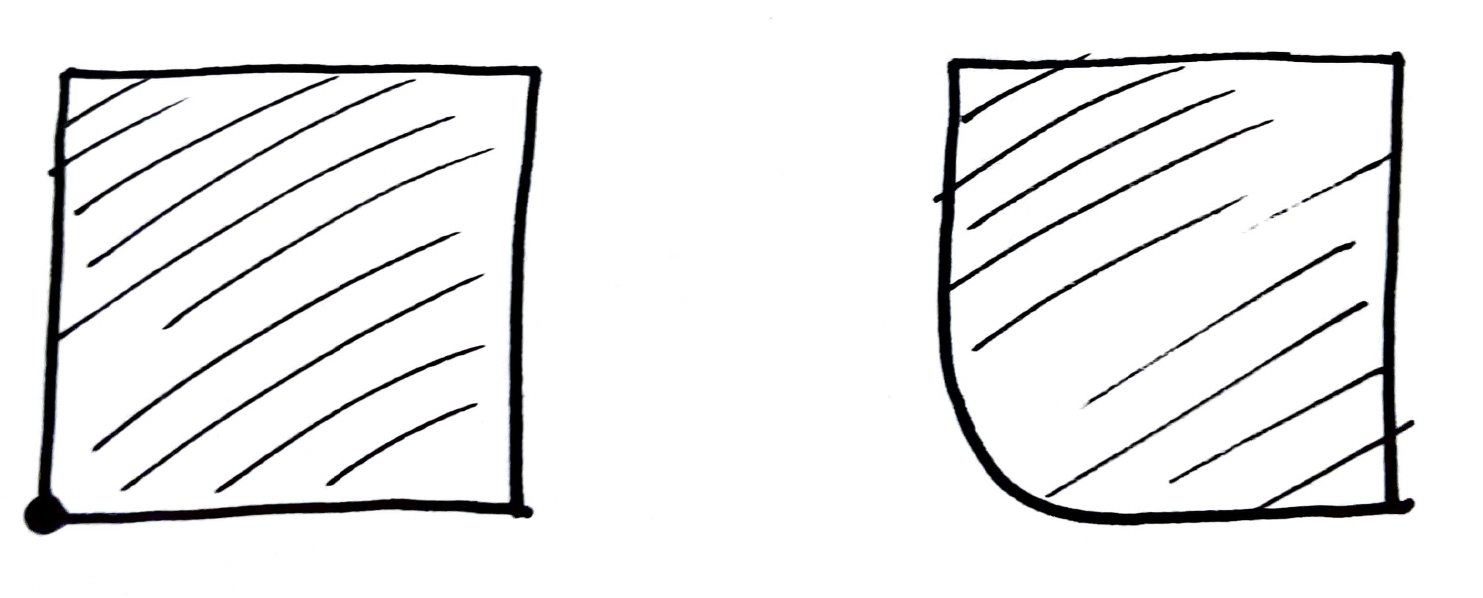}
 \caption{Smoothing corners of a manifold with corners.}
 \label{corner-smoothing}
 \end{figure}


\subsubsection{Collar extension and restriction}\label{sec: collar ext/rest}
Given a manifold with corners $M$ equipped with a corner structure, it will be useful to introduce,
for $\delta\in (-\infty, \eps)$, closely related 
manifolds with corners $M^{\delta \llcorner}$  equipped with corner structures.

First, we set $M^{0 \llcorner} = M$. In general, we will have $M^{\delta_1 \llcorner} \supset M^{\delta_2 \llcorner}$,  for $\delta_1 <\delta_2$, and a natural identification $(M^{\delta_1 \llcorner})^{\delta_2 \llcorner}= M^{(\delta_1+\delta_2) \llcorner}$.

For $\delta \in (0, \eps)$, and any $k$-face $P \subset M$, set
 $$
 U_P^{[0,\delta)} = P \times [0, \delta)^k \subset P \times [0, \eps)^k = U_P.$$
 Note that if  $\ell>k$,  $Q$  is an $\ell$-face and $P$  is a $k$-face adjacent to $Q$ then   $U_P^{[0,\delta)} \supset U_Q^{[0,\delta)} $.
 Denote 
$$   U^{[0,\delta) } =  \bigcup \limits_{\mbox{$1$-faces $P$}} U_P^{[0,\delta)} =  \bigcup \limits_{\mbox{$k$-faces $P$ with $k>0$}} U_P^{[0,\delta)}$$ and  define $M^{\delta \llcorner} := M \setminus U^{[0,\delta) } $ with its natural structure of a smooth manifold with corners.

Note that  for any  $k$-face $P$ of $M$ we have  $$P^{\delta \llcorner} \times \{ \underbrace{(\delta,\dots,\delta) }_k \}$$ the corresponding  $k$-face of $M^{\delta \llcorner}$,  where  the notation $P^{\delta \llcorner} $ has a similar meaning applied to the manifold with corners $P$.  
%
We denote also
$$
U_P^{\delta \llcorner} 
=U_P \setminus\bigcup\limits_{Q \;\hbox{is a $1$-face of $P$}}U_Q^{[0,\delta)} .$$
 The collar coordinates $t-\delta$ provide a natural corner structure on $M^{\delta \llcorner}$.

\begin{definition} For $\delta \in (-\infty, 0)$,  
 and any $k$-face $P \subset M$, set
$$
U_P^{[\delta, 0)} = P \times [\delta,0)^k , \qquad U^{[\delta,0 )} =   \bigcup \limits_{\mbox{$k$-faces $P$ with $k>0$}} U_P^{[\delta, 0)} 
$$
Then we define $M^{\delta \llcorner} = M  \cup U^{[\delta,0)}$ with its natural structure as smooth manifold with corners.
\end{definition}
Note that the $k$-faces of $M^{\delta \llcorner}$ are given by the same construction $P^{\delta \llcorner}$ applied to the $k$-faces of $M$,
%
and have the natural collars
$$
U_P^{\delta \llcorner} =  P^{\delta \llcorner} \times [\delta, \eps)^k
$$
  with collar coordinates $t-\delta$ providing a natural corner structure on $M^{\delta \llcorner}$.

\begin{definition}
For $\delta=\pm\eps/2$, we will set $M^{\pm \llcorner}= M^{\delta \llcorner}$.  We will call $M^{- \llcorner}$ the {\em collar extension} of $M$, and  $M^{+ \llcorner}$ the {\em collar restriction} of $M$. 
The collar restriction of the collar extension will recover the original manifold $M = (M^{- \llcorner})^{+\llcorner}$, and vice versa $M=(M^{+\llcorner})^{-\llcorner}$.
\end{definition}

  \begin{figure}[h!]
 \centering
  \includegraphics[scale=0.3]{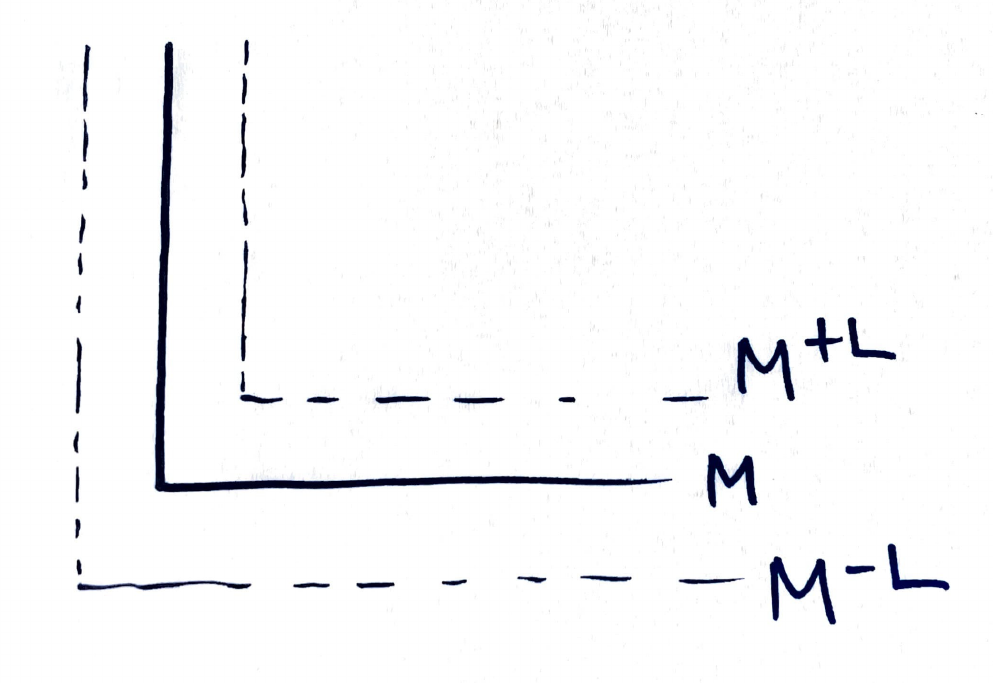}
 \caption{Collar extension and restriction.}
 \label{corner-extension-reduction}
 \end{figure}

\subsection{Definition of W-blocks}

Here we introduce the notion of W-blocks, which are closely related to Weinstein manifolds but differ in two key aspects: 
\begin{enumerate}
\item  they  allow boundaries and corners,
\item they are germs along their skeleta.
\end{enumerate}

   \begin{figure}[h!]
 \centering
  \includegraphics[scale=0.25]{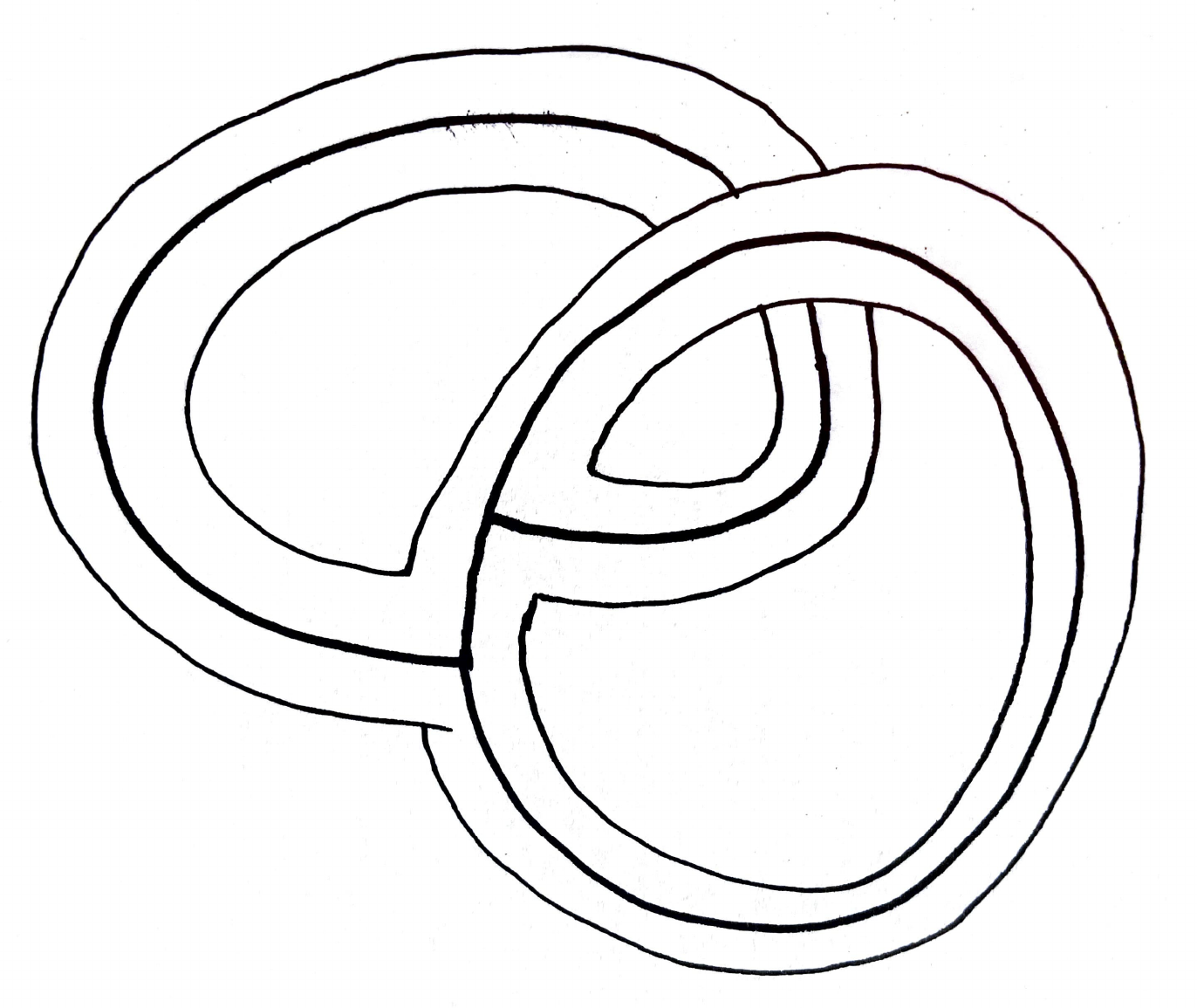}
 \caption{A Weinstein germ is the germ of a Weinstein manifold along its skeleton. }
 \label{germ-at-skeleton}
 \end{figure}

First, let us give a definition of a $2n$-dimensional Weinstein manifold with corners which is  inductive on the dimension $2n$ with the base case simply a  $0$-dimensional Weinstein manifold. As a guiding example, we recommend the reader keep in mind the
cotangent bundle of a compact  $n$-dimensional smooth manifold with corners (see Example~\ref{ex:co-bl} below).

For  $X$ a smooth manifold with corners and $W$ a compact manifold with corners, we will require any codimension zero embedding $W \subset X$ to have the following property: each facet (1-face) $F$ of $W$ is either entirely contained in $\partial X$ or is properly embedded in $X$, so that in the latter case $F$ intersects $\partial X$ transversely along its boundary and the interior of $F$ is contained in the interior of $X$. The union of the former facets is denoted $\partial_h W \subset \partial X$, the {\em horizontal boundary}, and the union of the latter facets is denoted $\partial_v W$, the {\em vertical boundary}. We will further assume compatibility of the collar structures in that the collar coordinates of facets in $\partial_h W$ are given by the collar coordinates of $\partial X$. 


\subsubsection{L-blocks}


 \begin{definition}
 
 A $2n$-dimensional  {\em Liouville manifold with corners} $(X,\lambda)$ is a manifold with corners $X$ with Liouville form $\lambda$, symplectic form $\omega = d\lambda$, and complete Liouville field $Z$ given by $\iota_Z \omega = \lambda$, satisfying  the following:
 
 \begin{itemize}
 
 \item  (Splittings) Each $k$-face $P \subset X$, with $k> 0$, admits a splitting $P = N \times \R^k$, so a further splitting  of its collar neighborhood $U_P =  N \times \cI^k \times \R^k$, so that we have an identification of Liouville forms
 $$(U_P, \lambda)  =  (N, \lambda_N) \times (T^*\cI^k, s dt)$$ 
where $(N, \lambda_N)$ is a $(2n-2k)$-dimensional Liouville manifold with corners  (inductively defined), $t$ is the collar coordinate on $\cI^k$, and $s$ is the splitting coordinate on $\R^k$.

Note such a splitting is unique if it exists (see Remark~\ref{rem:spl} immediately below).

 \end{itemize}

 We call $N$ the {\em nucleus} of the $k$-face $P$. Note   $P$ is coisotropic, $Z$-invariant, and under the splitting  $P = N \times \R^k$, we have $Z =Z_N \times s\partial_s$, so $N$ is $Z$-invariant as well. 

We also require the following to hold:

  \begin{itemize}

  \item (Liouville property)  There exists a manifold with corners $W \subset X$, which is a compact subset of $X$,  called a {\em defining domain}, such that  $Z$ is outwardly transverse to $\partial_v W$ and
    $
  X =  \bigcup_{t>0}Z^{t}(W).$


 \end{itemize}
 
  The {\em skeleton}  of $(X,\lambda)$ is the subset $\Skel(X, \lambda) = \bigcap_{t>0}Z^{-t}(W)$, which is independent of the choice of defining domain  $W \subset X$.

  \end{definition}

 \begin{remark}\label{rem:spl}
1.  When splittings as in the above definition exist, they are in fact unique. To see this,
 first note that $(N, \lambda_N)$ is the reduction of $(X, \lambda)$ along the $Z$-invariant coisotropic $P \subset X$. Then observe that any Liouville isomorphism of the product 
 $$(N, \lambda_N) \times (T^*\cI^k, s dt)$$ 
which is fixed on $N \times \cI^k$ must in fact be the identity. Indeed, if we fix any point $n\in N$, this reduces to the observation that  any Liouville isomorphism of a cotangent bundle   relative to the base must be the identity.
 \end{remark}

 Now we define the notion of a {\em Liouville block}, or L-block for short, by passing to the germ of a Weinstein manifold with corners near its skeleton.
 
  \begin{definition}

A      $2n$-dimensional {\em L-block} $(\sX,\lambda)$  is the germ  of  a $2n$-dimensional Liouville manifold with corners $(X,\lambda)$ along its  skeleton $\Skel(X, \lambda) \subset X$.  \end{definition}

   More precisely, as in the case of Liouville germs, given a fixed background smooth manifold an L-block is an equivalence class of defining domains for L-blocks under the equivalence relation resulting from identifying two defining domains if the L-block structures agree after restriction to a smaller defining domain. Note that any construction on Liouville manifolds with corners, or their defining domains, which are conical with respect to $Z$ outside of the skeleton are well-defined at the level of L-blocks.   
   
 Any two defining domains for the same L-block have the same skeleton and hence we may denote it by  $\Skel(\sX, \lambda)$, the {\em skeleton}  of the L-block $(\sX,\lambda)$. The {\em ideal boundary}  of  $(\sX,\lambda)$ is the cooriented contact manifold with corners $$\partial_\infty \sX \simeq \big(\sX \setminus \Skel(\sX, \lambda)\big)/\R,$$ where  the local $\R$-action  is by the flow of $Z$.  Again, this really means that canonical contactomorphism between the contact manifolds $\p_\infty W = ( W \setminus \Skel(W,\lambda) ) / \R$ and  $\p_\infty \sX$ is well defined at the level of germs.

  Given a $2n$-dimensional Liouville manifold with corners $(X,\lambda)$ and a $k$-face $P \subset X$ with nucleus $N\subset P$, we write  $\sP \subset \sX$ and $\sN\subset \sP$ for their 
   germs along the skeleton $\Skel(X, \lambda) \subset X$ and refer to them as a  $k$-face and nucleus of 
   $(\sX,\lambda)$ respectively. Given any choice of a defining domain $W$ for $X$ we get a choice of representatives for $\sP$ and $\sN$. Indeed $W$ has horizontal boundary $\partial_hW$, each facet consisting of a co-isotropic manifold with boundary whose symplectic reduction is a defining domain for the corresponding facet of $N$.

 Any  defining domain for $(X,\lambda)$ is also called a {\em defining domain} for 
  $(\sX,\lambda)$.
 
   \begin{figure}[h!]
 \centering
  \includegraphics[scale=0.3]{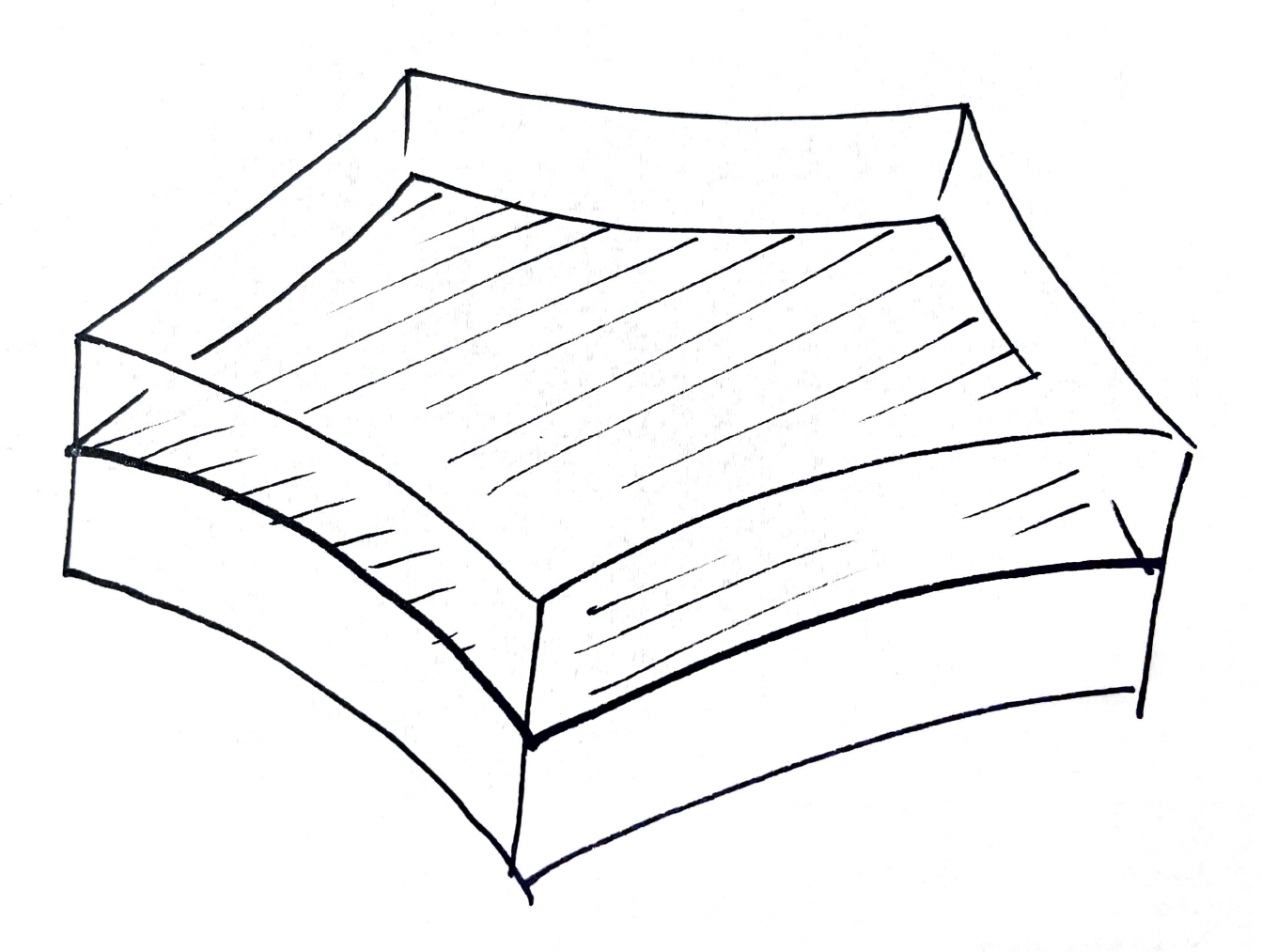}
 \caption{The cotangent bundle of a compact manifold with corners is a cotangent block.}
 \label{cotangent-block}
 \end{figure}

 Now we introduce {\em proper cotangent blocks}, a class of L-blocks of primary interest. To understand our use of the term proper, see Definition ~\ref{ex:gen-co-bl} where we introduce {\em cotangent blocks}.

 \begin{definition}[Proper cotangent blocks]\label{ex:co-bl}

 Let $M$ be a compact  $n$-dimensional smooth manifold with corners. Then its cotangent bundle 
 $(T^*M, pdq)$ is a Liouville manifold with corners.
 Note the $k$-faces $P \subset M$ correspond to the $k$-faces $T^*M|_P  \subset (T^*M, pdq)$, and the corner structure of $M$ induces splittings of $T^*M$. Namely, the collar neighborhood $U_P = P \times \cI^k$ provides an identification
 $T^* U_P = T^* P \times T^* \cI^k$, and hence a splitting $T^* M|_P = N \times \R^k$ with nucleus $N = T^*P$ and $\R^k = T_0^* \cI^k$.  We call  the L-block $(\sT^*M,pdq)$ given by
the
   germ of $(T^*M,pdq)$    along the $0$-section $M \subset T^*M$    a {\em proper cotangent block}. 
\end{definition}

%
%
%
%
%


\subsubsection{Split potentials}

Consider Liouville manifolds with corners $(X,\lambda)$ of Weinstein type, i.e. such that $Z$ admits a potential $\phi : X \to \RR$. For an L-block $(\sX,\lambda)$, which is the germ of a Liouville manifold with corners at the skeleton, we do not care or need to know what $\phi$ is outside of a neighborhood of $\Skel(\sX,\lambda)$, so a potential is the germ at $\Skel(\sX,\lambda)$ of a function  $\phi: \sX \to \RR$.

Recall by definition
each $k$-face $\sP \subset \sX$, with $k> 0$, admits a splitting $\sP = \sN \times \R^k$, hence inherits a further splitting  of its collar neighborhood $U_\sP =  \sN \times \cI^k \times \R^k$, so that we have an identification of Liouville forms
 $$(U_\sP, \lambda)  =  (\sN, \lambda_\sN) \times (T^*\cI^k, s dt)$$ 
where $(\sN, \lambda_\sN)$ is a $(2n-2k)$-dimensional L-block, $t$ is the collar coordinate on $\cI^k$, and $s$ is the splitting coordinate on $\R^k$. Note that if $\phi$ is a defining potential for a defining domain $W$, then $W_{\sN}=W\cap\sN$ is the defining domain for $\sN$ associated to the defining potential $\phi_\sN=\phi|_\sN$.

\begin{definition}\label{def:split-pot}
We say a potential $\phi:\sX\to \R$
is {\em split} 
 if 
for each $k$-face $\sP \subset \sX$, $k> 0$, with the nucleus $\sN$  
we have  on $\Op\sN$
\begin{equation}
  \phi = \phi_\sN + \sum_{i = 1}^k  s_i^2
\end{equation}
\end{definition}

\subsubsection{W-blocks}

We now give the definition of a {\em Weinstein block}, or W-block for short.

\begin{definition}
A  {\em Weinstein manifold with corners} is a Liouville manifold with corners such that $Z$ admits a split potential.\end{definition}

\begin{definition}
A {\em W-block} is an L-block $(\sX,\lambda)$ such that $Z$ admits a split potential, i.e. a W-block is the germ of a Weinstein manifold with corners $(X,\lambda)$ at the skeleton $\Skel(X,\lambda)$.  \end{definition}

 \begin{example}[Split Riemannian proper cotangent blocks]\label{ex:riem-co-bl}

 Let $M$ be a compact  $n$-dimensional smooth manifold with corners. 
Recall  a  proper cotangent block is  the L-block $(\sT^*M,pdq)$ given by
the   germ of the cotangent bundle  $(T^*M,pdq)$    along the $0$-section $M \subset T^*M$

  If we equip $M$ with a Riemannian metric $g$, then  when viewed as a fiberwise quadratic  function $Q_g:T^*M\to \RR$ it is a potential for the Liouville field $p\frac{\p}{\p p}$, see Example \ref{ex:cot-bundle-potential}. If the collar splitting is also a Riemannian splitting for the metric $g$ then $Q_g$ is a split potential. Hence a proper  cotangent block $(\sT^*M,pdq)$ is a W-block.

  \end{example}

%
%
%



\subsection{Homotopies of W-blocks}\label{sec:block-homotopy}
By definition, a W-block is a germ along its skeleton. On the other hand, during a homotopy of its Weinstein structure, the topology of the  skeleton may change. Therefore we will  consider two notions of equivalences of W-blocks under which the skeleton is either unchanged or not. One may make   similar definitions in the category of L-blocks.

\begin{definition}
We say that two W-blocks $(\sX_0,\lambda_0)$ and  $(\sX_1,\lambda_1)$  in the same background manifold  $X$ with corners, are {\em  homotopic} if there is 
\begin{enumerate}
\item  a smooth family of  hypersurfaces $\Sigma_t$  which are transverse to $\p X$ and bounding compact domains $W_t$
\item  a family of Liouville forms $\lambda_t$ on $W_t$
\item a smooth family of functions $\phi_t:W_t\to\RR$
\end{enumerate} such that $(W_t,\lambda_t)$ is a Weinstein domain (with corners) with a potential $\phi_t$  for all $t\in[0,1]$, and for 
   $i= 0,1$, $(W_i,\lambda_i)$ is a defining domain for the  W-block
$(\sX_i,\lambda_i)$. 
 We say that two W-blocks $(\sX_0,\lambda_0)$ and  $(\sX_1,\lambda_1)$   are {\em  strongly homotopic} if there is a    homotopy $(W_t,\lambda_t,\phi_t)$  as above, with constant  skeleton $L=\Skel(W_t,\lambda_t)$ throughout the homotopy.

\end{definition}

\begin{definition}
 Two W-blocks $(\sX_0,\lambda_0)$ and  $(\sX_1,\lambda_1)$  are {\em deformation equivalent} if $\sX_0$ is isomorphic to a W-block which is homotopic to $\sX_1$. If $\sX_0$ is isomorphic to a W-block which is strongly homotopic to $\sX_1$ we will say that $\sX_0$ and $\sX_1$ are {\em strongly deformation equivalent}. 
\end{definition}
Note that a  strong homotopy  of W-block structures is not required to be fixed on boundary faces, but merely requires the strong homotopy of their nuclei and homotopy of the boundary splittings. As in the case of Weinstein manifolds, strong homotopy of W-blocks can be accompanied
by symplectic isotopy of completion of defining domains.

\section{Operations on W-blocks}\label{sec:operations}
  \subsection{Horizontal gluing}\label{sec:hor-gluing}

  Suppose given two $2n$-dimensional W-blocks $(\sX_1,\lambda_1)$, $(\sX_2,\lambda_2)$, 
    boundary faces $\sP_1$ of $\sX_1$, $\sP_2$ of $\sX_2$, and a Liouville isomorphism of nuclei
  $\phi: \sN_1 \risom \sN_2$. Using the given collars and splittings,  we have $\sP_1 = \sN_1\times \RR$, 
   $\sP_2 = \sN_2\times \RR$, and can extend $\phi$ to a diffeomorphism $\Phi:  \sP_1\risom \sP_2$, $ \Phi = \phi \times \text{id} _\RR$.

   \begin{definition} The {\em horizontal gluing} of $(\sX_1,\lambda_1)$ and $(\sX_2,\lambda_2)$ along $\sP_1$, $\sP_2$, $\phi$ is the  W-block $(\sX, \lambda)$, where we set  $$\sX = \sX_1 \mathop{\cup }\limits_\Phi\sX_2= \sX_1 \mathop{\cup}\limits_{\sN_1\sim\sN_2 }\sX_2 , \qquad  \lambda|_{\sX_1} = \lambda_1, \quad  \lambda|_{\sX_2} = \lambda_2.$$
      \end{definition}
      
Note that we implicitly use the collar structures on $\sX_i$ to construct a smooth structure on $\sX$, and then the W-block structure on $\sX$ is determined by restriction of the W-block structures on $\sX_i$. Note also that 
$$
\Skel(\sX, \lambda) = \Skel(\sX_1, \lambda_1)  \mathop{\cup }\limits_{\Skel(\sN, \lambda_\sN)} \Skel(\sX_2, \lambda_2)
$$

  \begin{figure}[h!]
 \centering
  \includegraphics[scale=0.3]{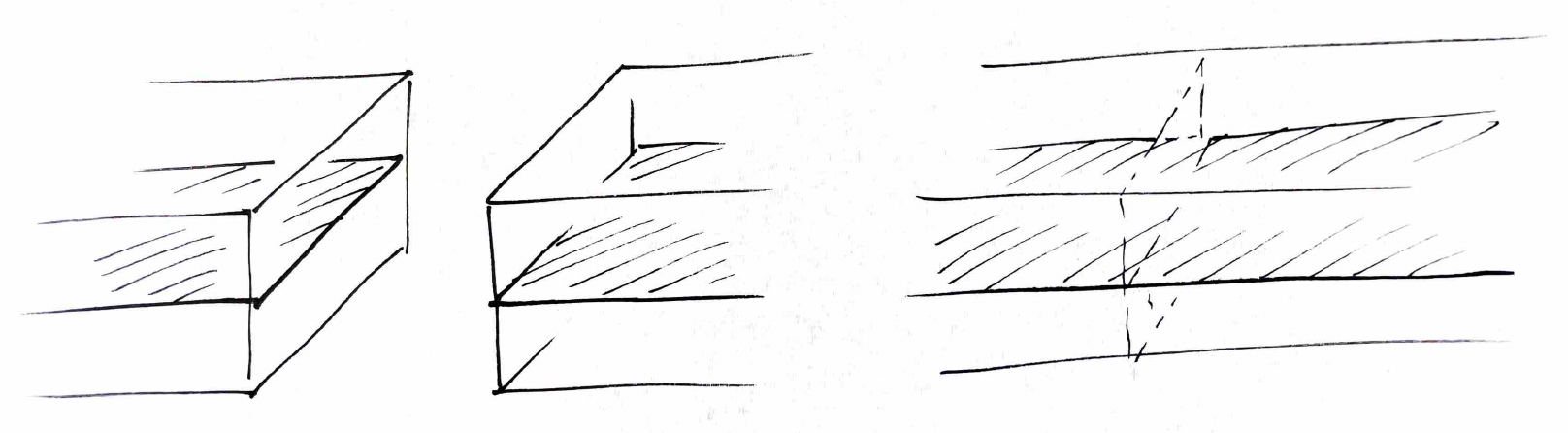}
 \caption{Horizontal gluing of W-blocks.}
 \label{horizontal-gluing}
 \end{figure}

   \begin{example}
   Suppose $M_1, M_2$ are two smooth manifolds with boundary and we have a diffeomorphism $\psi: \p M_1 \risom \p M_2$. Then $\psi$ lifts to a Liouville isomorphism $\phi:T^*(\p M_1) \risom T^*( \p M_2)$, and from there to a  diffeomorphism $\Phi:  T^*N|_{\partial N}\risom T^*M|_{\partial M}$.
Horizontal gluing gives $T^*M_1 \cup_\Phi T^*M_2 = T^*(M_1 \cup_\psi M_2)$, where $M_1 \cup_\psi M_2$ is the smooth gluing  given by $\psi$ (and the collar coordinates). More generally, suppose that we are given two diffeomorphic facets of smooth manifolds with corners $M_1$ and $M_2$ and a diffeomorphism between them. We may glue $M_1$ and $M_2$ along these facets to a new smooth manifold with corners whose cotangent bundle is the horizontal gluing of the cotangent bundles $T^*M_1$ and $T^*M_2$ along the corresponding facets.

   \end{example}
%
    
\subsection{Splitting}
 
The inverse operation to horizontal gluing is called splitting.   
   
\begin{definition}    A {\em splitting  hypersurface} $\sH\subset \sX$ in a 
W-block $(\sX,\lambda)$ consists of the germ of a hypersurface $\sH\subset \sX $  at $\Skel(\sX,\lambda)\cap \sH$ satisfying the following:
\begin{enumerate}
\item $\sH$ is regularly embedded; it has boundary and corners which lie in the boundary and corners of $\sX$, and we moreover assume that the collar coordinates for the corner structure of $\sH$ are given by restriction of that of $\sX$. 

\item $\sH$ is separating; dividing $\sX$ into two (possibly disconnected) halves $\sX_1$ and $\sX_2$.

\item A neighborhood of $\sU$ of  $\sH$ in $\sX$ is equipped with a splitting $\sU = \sN \times\sT^* (-\eps, \eps)$, such that   $
   \lambda|_{\sU}  = \lambda_\sN + sdt   
   $,   $(\sN, \lambda_\sN)$ is a W-block and
 $\sH = \{t=0\} \subset\sU$.  We call $\sN$ the nucleus of $\sH$.
     \end{enumerate}
     \end{definition}
  
Note in particular that a splitting hypersurface is invariant under the Liouville flow (hence conical). It is transverse to each boundary face $\sP\subset\sX$ and its nucleus $\sM$, and   $\sH\cap\sM$ is a splitting hypersurface for $\sM$.

\begin{definition}
The {\em splitting} of a  W-block $(\sX,\lambda)$ along a  splitting  hypersurface $\sH\subset \sX$  
is
the pair
of W-blocks $(\sX_1,\lambda_1)$, $(\sX_2,\lambda_2)$ obtained by removing $\sH$ from $\sX$ and then adding it back in to each component, so that
$\sX = \sX_1 \mathop{\cup }\limits_\sH\sX_2$, and $ \lambda|_{\sX_1} = \lambda_1$,
   $ \lambda|_{\sX_2} = \lambda_2$.

\end{definition} 
 
   \begin{figure}[h!]
 \centering
  \includegraphics[scale=0.3]{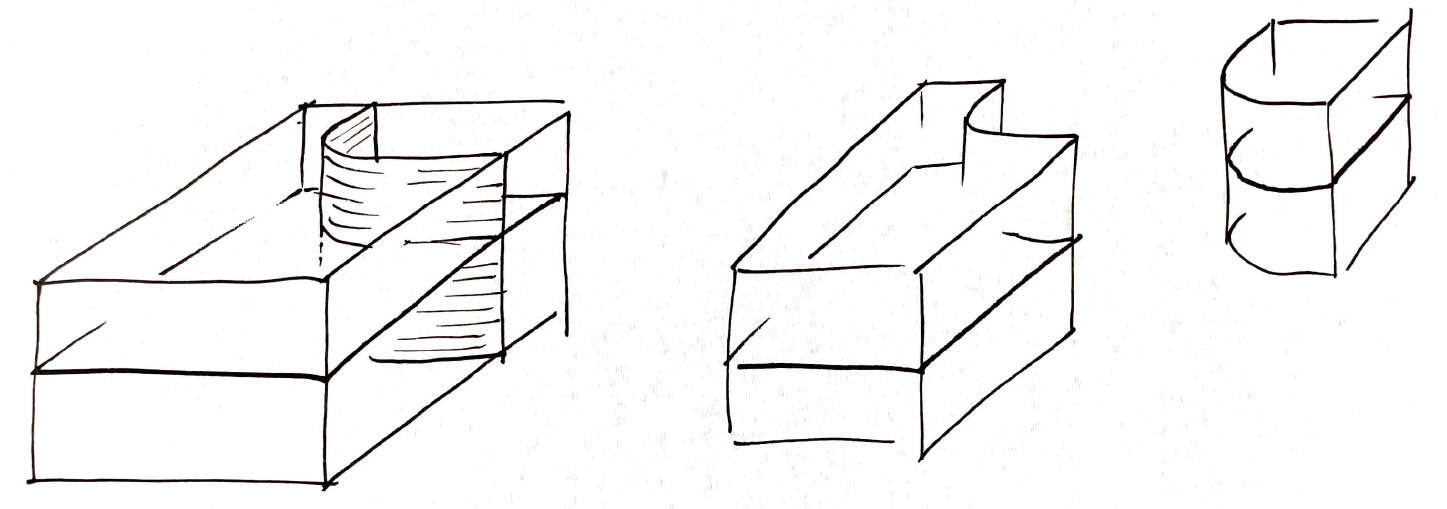}
 \caption{Splitting a W-block via a splitting hypersurface.}
 \label{splitting-hypersurface}
 \end{figure}

\begin{example}
Suppose $M$ is a smooth manifold with corners, and $N\subset M$ a regularly embedded smooth hypersurface with corners
dividing  $M$ into two parts $M_+$, $M_- $.
Assume the corner structure on $N$ is compatible with the corner structure on $M$, and that $N$ comes equipped with a product neighborhood $U = N \times (-\eps, \eps)$. 
Then   the splitting of the W-block $\sT^*M$ along the  splitting hypersurface $\sH = \sT^*M|_N$ results in  two W-blocks $\sT^* M_+$, $\sT^* M_-$. 
\end{example} 

More generally,   consider the following inductively constructed sequence of W-blocks  $$\sX=\sX_{\geq 1},\sX_1, \sX_{\geq 2} ,\sX_2  \dots, \sX_{\geq k}, \sX_k,$$
and dividing hypersurfaces $\sH_j\subset \sX_{\geq j}$, $j=1,\dots,k$.
\begin{itemize}
\item $\sH_1$ is  a splitting hypersurface in $ \sX$ which  divides  $\sX$ into two W-blocks $\sX_1 $ and $\sX_{\geq 2}$;
\item $\sH_2\subset \sX_{\geq 2}$ is a splitting hypersurface dividing $\sX_{\geq 2}$ into two W-blocks $\sX_{2} $ and $\sX_{\geq 3}$;
\item if for $0\leq j<k$  $\sH_j\subset \sX_{\geq j}$ is already defined  and divides  $\sX_{\geq j}$ into two W-blocks $\sX_{j}$ and $ \sX_{\geq j+1}$ then $\sH_{j+1}\subset \sX_{\geq j+1}$   is a splitting hypersurface  which divides $  \sX_{\geq j+1}$ into  two W-blocks $\sX_{j+1}$ and $ \sX_{\geq j+2}$.
\end{itemize}
We call the list  $[\sH_1, \ldots, \sH_k]$ a {\em dividing collection of hypersurfaces} for $\sX$ and refer to the presentation of $\sX$ as the result of  successive horizontal gluings,
$$\sX_{\geq k}=\sX_{\geq k+1}\mathop{\cup}\limits_{\sH_{k}}\sX_{k},\;\dots  \quad \sX=\sX_{\geq 1}=\sX_{ \geq 2 }\mathop{\cup}\limits_{\sH_{1}}\sX_{ 2},$$
an {\em admissible decomposition} of the W-block $\sX$.

\subsection{Smoothing corners of a W-block} \label{sec:smooth-corners-weinstein}
One can specialize  the procedure of smoothing corners  of a general manifold with corners described in Section \ref{sec:smooth-corners} to the case of W-blocks.
 
Let  $(\sX, \lambda)$ be  a W-block and 
$\bP=[\sP_1,\dots,\sP_k]$ an  admissible list of its facets and $\bN=[\sN_1,\dots,\sN_k]$   the corresponding list of their nuclei.
 Then there exists a W-block $(\sX^{  \frown,\bP}, \lambda^{ \frown,\bP})$ obtained by  successive smoothing corners between the facets in the list $\bP$.  To see this, recall along any $k$-face we have split collar neighborhoods
 $$(U_\sP, \lambda)  =  (\sN, \lambda_\sN) \times (\sT^*\cI^k, s dt) 
 $$ 
Then all smoothing respects these factorizations and is done with respect to the   collar coordinates $t$ on the factor $\cI^k$.
For example, for a proper cotangent block $(\sT^*M, pdq)$, whose faces correspond to faces of $M$, the smoothing gives the proper  cotangent block of the smoothing
$$
((\sT^*M)^{ \frown,\bP}, (pdq)^{\frown,\bP}) = \sT^*(M^{\frown,\bP}, u dv)
$$
The nucleus $\sN^{\frown. \bP}$ of the smoothed facet  $\sP^{\frown, \bP}$  is obtained by successive  horizontal gluing of $\sN_1$ to $\sN_2$ along the facet with the nucleus of the 2-face  $\sP_1\cap\sP_2$, then horizontally gluing the result to  (the smoothed) $\sP_3$, etc. 

  \begin{figure}[h!]
 \centering
  \includegraphics[scale=0.3]{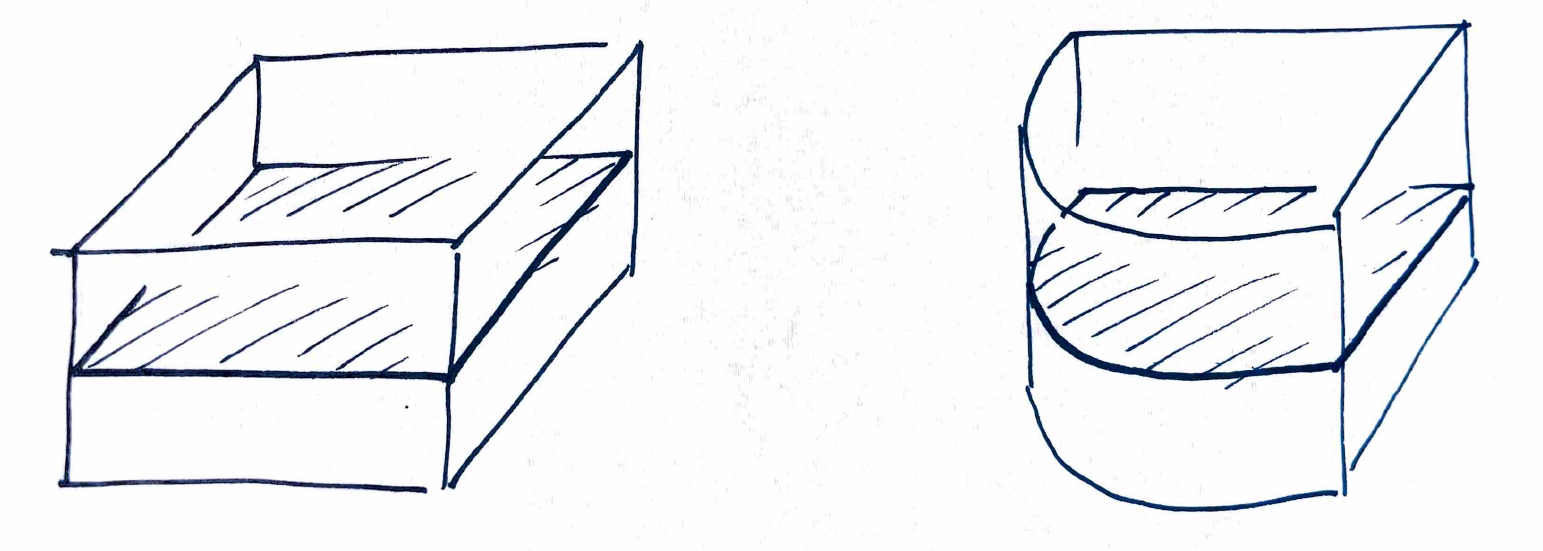}
 \caption{Smoothing corners of a W-block.}
 \label{smoothing-corner-block}
 \end{figure}

\begin{remark} It is straightforward to check  If $\bP'=[\sP_1,\dots,\sP_\ell],\; \ell<k$, is an initial sublist of $\bP =[\sP_1,\dots,\sP_k]$ 
then $\sX^{\frown,\bP}=(\sX^{\frown,\bP'})^{\frown,\bP/\bP'}$, where  we denote by $\bP/\bP'$ the list $[\sP^{\frown, \bP'},\sP_{\ell+1},\dots,\sP_k]$. 
\end{remark}

\subsection{Creating corners of a W-block}
Let us review the inverse operation  that introduces new corners, and which is also the specialization of an analogous operations for smooth manifolds with corners.

Let $(\sX,\lambda)$ be a W-block and $\sP\subset\sX$ its facet with a nucleus $\sN$. Consider the splitting $\Op_{\sX}\sN=\sN\times\sT^*\cI$, so that we have $\lambda|_{\Op\sN}=\lambda_\sN+s_1dt_1$, for $t_1$  the corresponding   collar coordinate near $\sP$, and so that $\sP=\{t_1=0\}$ and $\sX$ near $\sN$ is given by the inequality $t_1\geq 0$.
Let $\sH\subset\sN$ be a splitting hypersurface with its nucleus $\sM$ which splits $\sN$ into two blocks $\sN_1$ and $\sN_2$ horizontally glued along their common face $\sH$ with nucleus $\sM$.  
The form $\lambda|_{\Op_\sX\sM}$ can be written as $\lambda_\sM+s_1dt_1+s_2dt_2$, where  $\sH=\{t_1=s_1=t_2=0\}$, $\sN_1= \{t_1=s_1=0,t_2\geq 0\}$ and $\sN_2= \{t_1=s_1=0,t_2\leq 0\}$. For a small $\sigma>0$ consider a  $C^\infty$-function $\eta:\RR\to[ -\sigma,0]$ which is equal $|t|-\sigma$  for $|t|<\frac\sigma2$ and equal to $0$ for $|t|\geq\sigma$. Then by replacing $ \sX\cap \Op_\sX\sM$ by the domain
$\{t_1\geq \eta(t_2)\}$ we transform $\sM$ into a nucleus of a new  2-face $\sQ$ of a resulting W-block denoted by $\sX^{\wedge, \sH}$ adjacent to new facets $\sP_1$ and $\sP_2$ which replaced $\sP$. Gluing horizontally nuclei $\sN_1$ and $\sN_2$ along their common facet $\sH$ gives back the nucleus $\sN$ of the original facet $\sP$. 


More generally, given a splitting list $\bH=[\sH_1,\dots,\sH_k]$ for the nucleus  $\sN$ of a facet $\sP\subset\sX$, i.e. an admissible decomposition of $\sN$ which splits it as
$$\sN=( \sN_{\geq k+1}\mathop{\cup}\limits_{\sH_{k}}\sN_{ k})\mathop{\cup}\limits_{\sH_{k-1}}\sN_{k-1})\dots)\mathop{\cup}\limits_{\sH_{1}}\sN_{ 1},$$ we  can modify the W-block $\sX$ into a new block $\sX^{\wedge,\bH}$ which has facets $\sP_{\geq k+1}, \sP_{k}, \sP_{k-1},\dots,\sP_{2},\sP_1$ with nuclei   $\sN_{\geq k+1}, \sN_{k},\sN_{k-1},\dots,\sN_2,\sN_1$ which replace the facet $\sP\subset\sX$. Note that the list $\bP=[\sP_{\geq k+1}, \sP_k,\sP_{k-1},\dots,\sP_2,\sP_1]$ is admissible by smoothing corners in $\bP$, i.e. taking
$(\sX^{\wedge,\bH})^{\frown,\bP}$ we get back a block strongly deformation equivalent to $\sX$.
 
%

We say that the W-block  $\sX^{\wedge,\bH}$ is obtained from $\sX$ by {\em corner ramification}. Thus, {\em any W-block is the result of corner ramification of a W-block with a smooth boundary.}
 
  \begin{figure}[h!]
 \centering
  \includegraphics[scale=0.3]{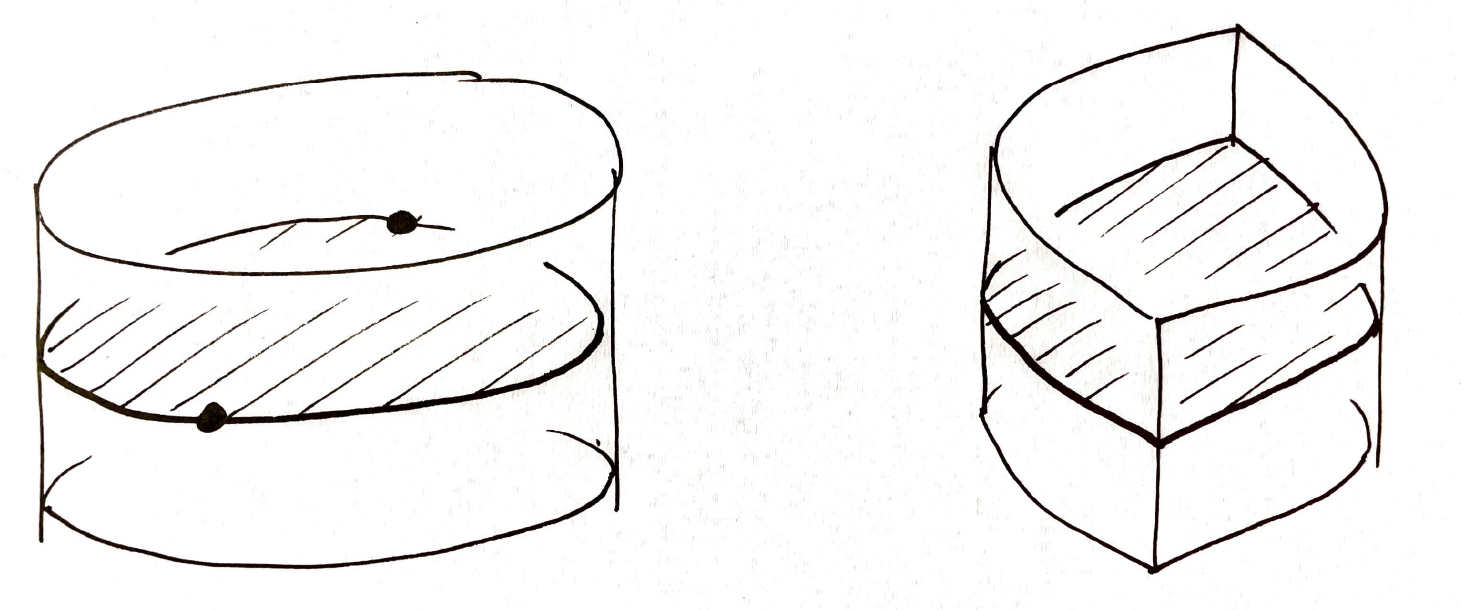}
 \caption{Creating corners.}
 \label{corner-creation}
 \end{figure}


\subsection{Legendrian and W-hypersurfaces in  W-blocks}\label{sec:Legendrian}
Recall that the  ideal boundary  of a W-block  $(\sX,\lambda)$ is the cooriented contact manifold with corners $( \partial_\infty \sX ,\xi_\infty) \simeq \big(\sX \setminus \Skel(\sX, \lambda) , \ker(\lambda) \big)/\RR$, where  the $\RR$-action  is by the flow of $Z$. 
For any defining domain $W\subset \sX$ we have a canonical contactomorphism  $( \partial_v W,\ker(\lambda|_{\p_v W} ) ) \to ( \partial_\infty \sX, \xi_\infty)$.  Note the contact form $\lambda|_{\partial_v W}$ depends on the choice of $W$ and hence is not well-defined at the level of germs. The ideal boundary $\p_\infty \sX$   inherits a natural corner structure from $\sX$. Moreover, the $k$-faces $\sP\subset  \sX$ correspond to (possibly disjoint unions of) $k$-faces $\partial_\infty \sP \subset \p_\infty \sX$.

\subsubsection{Preliminaries}

By a {\em conic} subset $S\subset \sX$, we will mean a subset  invariant under  the local $\RR$-action  by the flow of $Z$. 
Given a  conic subset $S\subset \sX$, its {\em ideal boundary} $\partial_\infty S \subset \partial_\infty \sX$ is defined to be 
$ \partial_\infty S = \big(S\cap  (\sX  \setminus   \Skel(\sX, \lambda)) \big)/\RR$. 

Conversely, given $S \subset  \partial_\infty \sX$, its {\em Liouville cone}
$C_\lambda(S) =  \Cone(S,\lambda) \subset \sX$ is defined to be
the pre-image of $S$ in $\sX \setminus \Skel(\sX, \lambda)$. By construction, $C_\lambda(S)$ is a $Z$-conic subset disjoint from $\Skel(\sX,\lambda)$ whose ideal boundary is $S$. Conversely, any $Z$-conic subset in $\sX$ disjoint from $\Skel(\sX,\lambda)$ is of the form $C_\lambda(S)$ for some subset $S \subset \Skel(\sX,\lambda)$. 


To speak about general (not necessarily conic) subsets of $\sX$, in particular of $\sX  \setminus   \Skel(\sX, \lambda)$, requires some care since   $\sX$ is only defined to be a germ of a defining domain $W$ along the skeleton    $ \Skel(\sX, \lambda)$. Therefore, by a  subset $S\subset \sX$, we will always mean a subset $S \subset W$  of some defining domain $W$. 
In practice, the only non-conic subsets we will deal with are subsets of $\sX \setminus \Skel(\sX,\lambda)$ which are lifts of subsets in $\p_\infty \sX$ (and are therefore equivalent to conic subsets disjoint from the skeleton) in the following sense:

We will say that a submanifold $T \subset \sX \setminus \Skel(\sX,\lambda)$ is a  {\em lift}, of a submanifold $S \subset \partial_\infty \sX$ if $T \subset \partial_v W$ for some defining domain $W$ and the map $\partial_v W \to \partial_\infty \sX$ restricts to a diffeomorphism $T \to S$. In particular each Liouville trajectory intersects $T$ at most once. 

\subsubsection{Legendrians in W-blocks}
  
     We introduce Legendrians in W-blocks and related notions. Roughly speaking, for a W-block $\sX$, Legendrians $\Lambda$ live in the ideal contact boundary $\p_\infty \sX$, and may be {\em represented} by Legendrian embeddings $\Lambda \hookrightarrow \sX \setminus \Skel(\sX)$, where the shorthand $\Lambda \subset \sX$ will always mean {\em in the complement of the skeleton}. We now give the precise definition, which we formulate for Legendrians but can be similarly formulated for isotropics. 
  
    \begin{figure}[h!]
 \centering
  \includegraphics[scale=0.3]{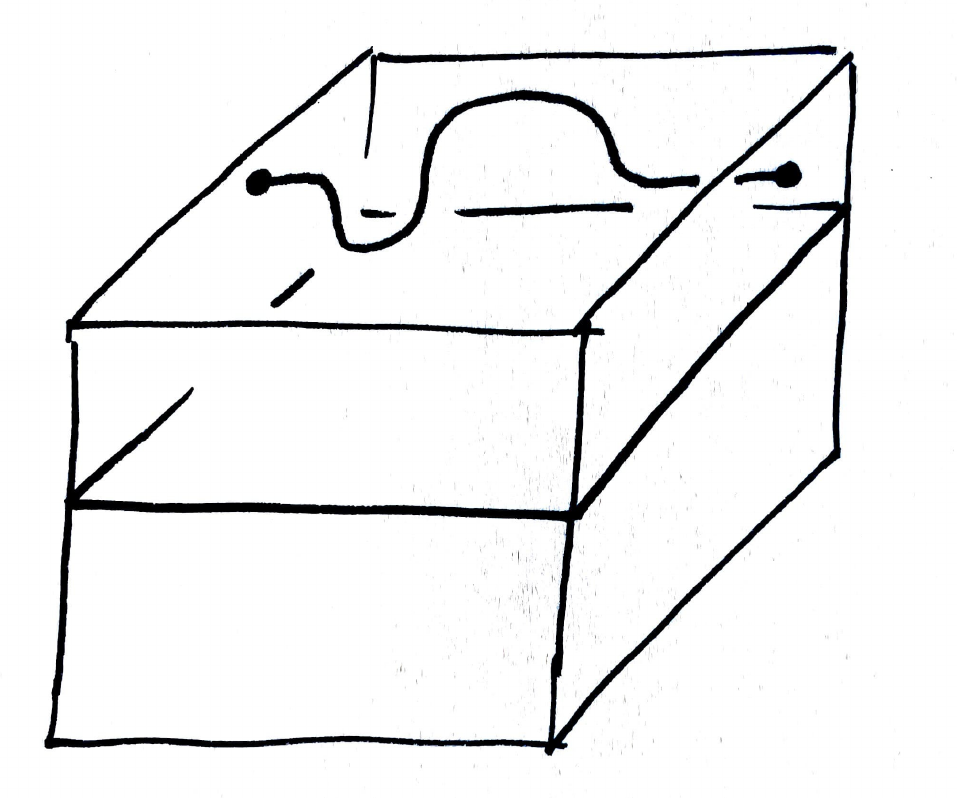}
 \caption{A Legendrian.}
 \label{legendrian}
 \end{figure}

 \begin{definition}\label{def:adLeg}
 Let $(\sX,\lambda)$ be a $2n$-dimensional W-block.

\begin{enumerate}
\item
 By a {\em Legendrian} in $(\sX,\lambda)$, we will mean an isotropic  embedding
$j_\infty:\Lambda \hookrightarrow  ( \p_\infty \sX, \ker(\lambda) )$
of an $(n-1)$-dimensional compact manifold in the ideal boundary.
 
\item By a {\em representative} of $j_\infty:\Lambda \hookrightarrow \p_\infty \sX$, we will mean  the data of a defining domain $W\subset \sX$ and 
 an isotropic  embedding
$j:\Lambda \hookrightarrow \p_v W \subset \sX$ such that $ j_\infty = j$ under $\partial_\infty \sX = \partial_v W$.  We will typically identify $\Lambda$ with its image $j_\infty(\Lambda)\subset  \p_\infty \sX$, or often its image $j(\Lambda)\subset 
 \p_v W \subset \sX$ when no ambiguity is possible. We will often write $\Lambda \subset \sX \setminus \Skel(\sX,\lambda)$, or even $\Lambda \subset \sX$, to denote a representative of a Legendrian.
 
 \item 
For a Legendrian $\Lambda \hookrightarrow \partial_\infty \sX$,  we will always assume that it is {\em adapted}, which means that for each  $k$-face $\sP$ of $\sX$, under  its
collar neighborhood
 $$(U_\sP, \lambda)  =  (\sN, \lambda_\sN) \times (\sT^*\cI^k, s dt) 
 $$ 
  we have $C(\Lambda) \cap U_\sP = C(\Lambda_\sN) \times \cI^k$, where $\Lambda_\sN$ is a Legendrian in  the nucleus $\sN$, and the coordinate $\cI^k$ gives the collar structure for $\Lambda$.
  \item 
For a representative $\Lambda \hookrightarrow \partial_v W$ of $\Lambda$, we also will always assume that it is {\em adapted} which means that
  for each  $k$-face $\sP\subset \sX$, under  
its  split collar neighborhood
 $$(U_\sP, \lambda)  =  (\sN, \lambda_\sN) \times (\sT^*\cI^k, s dt) 
 $$ 
 we have $\Lambda \cap U_\sP = \Lambda_\sN \times \cI^k$, where $\Lambda_\sN \subset \sN$ is a representative of a Legendrian in $\partial_\infty \sN$, and the coordinate $\cI^k$ gives the collar structure for $\Lambda$.
 
 \item  We will moreover impose the condition $j_\infty (\text{int} (\Lambda) )  \subset \text{int}(  \p_\infty \sX)$ and for each 1-face $P$ of $\Lambda$ demand that one of the following two conditions hold:
 \begin{enumerate}
 \item[ (A)] $j_\infty (P) \subset \p_1 \p_\infty \sX$, say inside a 1-face of $\p_\infty \sX$ which corresponds to a 1-face  $\sQ$ of $ \sX$ with nucleus, $\sN$, and $j_\infty|_P: P \to \p_\infty \sN$ is a Legendrian embedding such that on the collar $U_\sQ = \sN \times \sT^*\cI$ the restriction of $j_\infty$ to the collar $P \times \cI$ has the form 
 $$ P \times \cI \xrightarrow{j_\infty|_P \times \text{id} } \p_\infty \sN \times \cI \hookrightarrow  \p_\infty (\sN \times \sT^*\cI)$$ 
  \item[(B)] $j_\infty|_P$ is an embedding $(P,\p P) \to ( \p_\infty \sX,  \p_1 \p_\infty X)$ transverse to $\p_1 \p_\infty \sX$, and $j_\infty(\text{int}(P)) \subset \text{int}(\p_\infty \sX)$. 
\end{enumerate}
Faces $P$ of $\Lambda$ for which property (A) holds are called {\em pure}. If all 1-faces of $\Lambda$ are pure then $\Lambda$ is called pure.

 \item  For a representative $j:\Lambda \hookrightarrow \sX \setminus \Skel(\sX,\lambda)$ of a Legendrian in $\p_\infty \sX$, we analogously impose properties  (A) and (B) as follows. 
 \begin{enumerate}
 \item[ (A)] $j(P) \subset \p_1 \sX$, say inside a 1-face $\sQ$ of $\sX$ with nucleus $\sN$, and $j|_P: P \to \sN  \setminus \Skel(\sN,\lambda|_\sN)$ is a Legendrian embedding such that on the collar $U_\sQ = \sN \times \sT^*\cI$ the restriction of $j$ to the collar $P \times \cI$ has the form 
  $$ P \times \cI \xrightarrow{j|_P \times \iota }  \sN \times \sT^*\cI$$ where $\iota : \cI  \hookrightarrow \sT^*\cI$ is the inclusion of the zero section. 
  \item[(B)] $j|_P$ is an embedding $(P,\p P) \to ( \sX,  \p_1  \sX)$ transverse to $\p \sX$, and $j(\text{int}(P)) \subset \text{int}(\sX)$. 
\end{enumerate}
Faces $P$ of $\Lambda$ for which property (A) holds are called {\em pure}. If all 1-faces of $\Lambda$ are pure then the Legendrian $\Lambda$ itself is called pure. 
\end{enumerate}

  \begin{figure}[h!]
 \centering
  \includegraphics[scale=0.3]{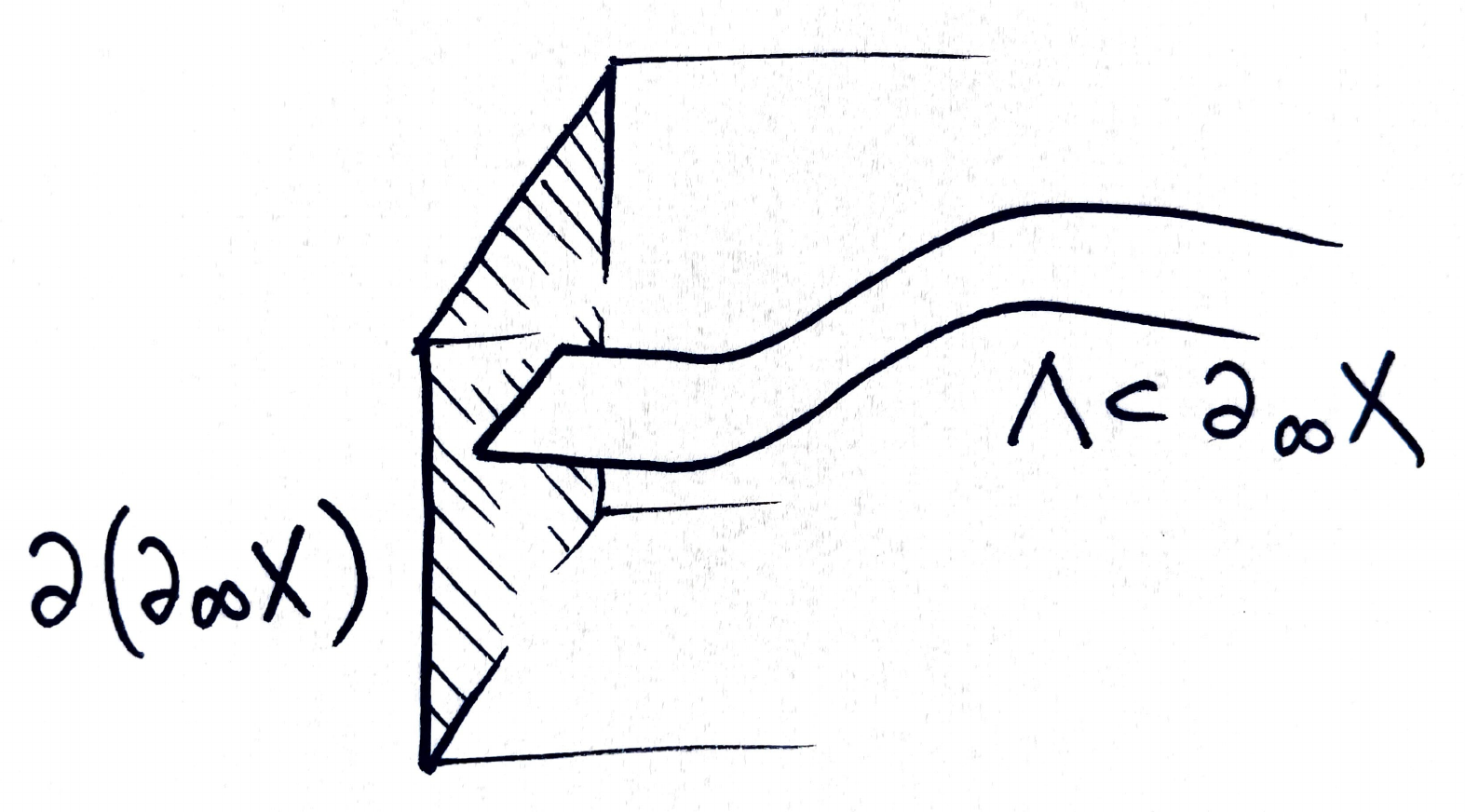}
 \caption{An example of a non-pure Legendrian. Some of the 1-faces of $\Lambda$ are contained in $\partial(\partial_\infty X)$ and some are not.}
 \label{purity}
 \end{figure}
%
%

\end{definition}

Note   an isotropic, or  specifically Legendrian, $\Lambda$  in $\p_\infty \sX$ is the same as a conic isotropic, or specifically Lagrangian, embedding $\Lambda \times \RR \subset  \sX\setminus\Skel(\sX, \lambda)$ where the $\RR$ direction parametrizes the Liouville flow.
Note also any two representatives of a Legendrian are equivalent.

\begin{example}
Suppose $M$ is a smooth compact manifold with boundary and corners, and $N\subset M$ a regularly embedded smooth submanifold.
Assume for each  $k$-face $P\subset M$,  with its   collar neighborhood
 $U_P = P \times \cI^k
 $ 
we have $N \cap U_P = (N \cap P) \times \cI^k$ and that this induces the corner structure on $N$ by restriction of collar coordinates. Then  the unit conormal bundle $S^*_N M \cap W$ is a Legendrian in $\sT^*M$.

\end{example}

%
%
%
%


 \subsubsection{W-hypersurfaces in W-blocks}\label{sec:W-hyp}
 Weinstein hypersurfaces in Weinstein manifolds
 were introduced  in \cite{E18}, and   are special cases of the Liouville hypersurfaces introduced by Avdek in \cite{A12}. The notion plays a similar role to the ``stops" of Sylvan, \cite{S16} and  the Liouville sectors of Ganatra-Pardon-Shende, \cite{GPS17}. Related constructions are also considered by    Ekholm-Lekili \cite{EL17}.  
 
 We will consider a natural generalization of this notion to the setting of W-blocks, which we call W-hypersurfaces.  As with Legendrians in W-blocks, we will think of these objects as lifts of W-blocks from the ideal boundary. To this end, 
observe that if $(\sA,\lambda_{\sA})$ is    a W-block, with Liouville field $Z_\sA$,  and $f:\sA\to \RR_{>0}$ is a   smooth function, then $(\sA,e^f\lambda_{\sA})$ is also a W-block if  and only if $  df(Z_{\sA})>-1$ (see  Lemma 12.1 in \cite{CE12}).  The space $C(\lambda_{\sA})$ of functions that satisfy this condition is a convex cone, and hence contractible.

  \begin{definition} \label{def:W-hyp}   Let $(\sX,\lambda)$ be a $2n$-dimensional W-block.
\begin{enumerate}

\item A {\em W-hypersurface} in $(\sX,\lambda_{\sX})$ is a $(2n-2)$-dimensional W-block $(\sA,\lambda_{\sA})$ with a smooth  embedding  $j_\infty:\sA\hookrightarrow \p_\infty \sX $   such that  $j^*\alpha_\infty=\lambda_{\sA}$, for some contact form $\alpha_\infty$ for the cooriented contact structure $\xi_\infty$. We assume the embedding is proper in that the boundary and corners of $\sA$ are mapped by $j_\infty$ into the boundary and corners of $\p_\infty \sX$ and the corner structure of $\p_\infty \sX$ induces that of $\sA$ via pullback.

\item  A {\em  representative} of $j_\infty:\sA\hookrightarrow \p_\infty \sX $ is   the data of a defining domain $W\subset \sX$ and 
 an  embedding
$j:\sA \hookrightarrow \p_v W \subset \sX$ such that $ j_\infty = j$ under the canonical identification $\partial_\infty \sX = \partial_v W$, again compatible with the corner structures,
and such that  $j^*\lambda=f\lambda_{\sA}$,  for some $f\in C(\lambda_{\sA})$. We will typically identify $\sA$ with its image $j_\infty(\sA)\subset  \p_\infty \sX$, or often its image $j(\sA)\subset 
 \p_v W \subset \sX$ when no ambiguity is possible.

\item A W-hypersurface $(\sA,\lambda_\sA)$ in $\partial_\infty \sX$ is always assumed {\em adapted}, which means that
   each  $k$-face $\sQ$  of  $\sA$, with nucleus $\sM$,  is the  intersection $\sQ = \sA \cap \p_\infty \sP$ with $\sP\subset \sX$ a $k$-face with nucleus $\sN$,
   such that under  
the  split collar neighborhood
 $$(U_\sP, \lambda)  =  (\sN, \lambda_\sN) \times (\sT^*\cI^k, s dt) 
 $$ 
 we have a splitting
 $$C( \sA ) \cap U_\sP = C(\sM ) \times \sT^*\cI^k. $$
\item A representative $\sA \subset \partial_v W$ of a W-hypersurface $(\sA,\lambda_\sA)$ in $\partial_\infty \sX$ is also  assumed {\em adapted}, which means that 
   each  $k$-face $\sQ\subset \sA$, with nucleus $\sM$,  is the  intersection $\sQ = \sA \cap \sP$ with a $k$-face $\sP\subset \sX$, with nucleus $\sN$,
   such that under  
the  split collar neighborhood
 $$(U_\sP, \lambda)  =  (\sN, \lambda_\sN) \times (\sT^*\cI^k, s dt) 
 $$ 
 we have the splitting
 $$ ( \sA \cap U_\sP , \lambda|_{\sA \cap U_\sP} )= ( \sM, \lambda_\sM) \times  (\sT^*\cI^k, sdt)$$ with $\sM = \sA \cap \sN$ and $\lambda_\sM = \lambda_\sN |_{\sM}$.

\end{enumerate}

 \end{definition}

\begin{remark}1.  Any embedding $j:\sA\hookrightarrow  \sX\setminus\Skel(\sX,\lambda)$
which lifts $ j_\infty $ over the projection $ \sX\setminus\Skel(\sX,\lambda) \to \partial_\infty \sX$,
and such that  $j^*\lambda=f\lambda_{\sA}$,  for some $f\in C(\lambda_{\sA})$, is contained in the vertical boundary $\p_vW$ of some defining domain $W\subset \sX$. 

2. Conditions (iii) and (iv) imply that coordinates $t_j$ in collar neighborhoods  (i.e. distances to a facet $\sP$ of $\sX$ and the corresponding facet $\sQ=\sA\cap\sP$ of $\sA$) coincide.
\end{remark}

\begin{example}
Suppose $M$ is a smooth compact manifold, and $N\subset M$ a regularly embedded  smooth hypersurface with coorientation
 $\sigma \subset S^*_N M$ in the cosphere bundle.
Assume for each  $k$-face $P\subset M$,   
with  collar neighborhood
 $U_P = P \times \cI^k
 $ 
we have that $N \cap U_P = (N \cap P) \times \cI^k$ and the collar coordinates of $N$ are obtained from those of $P$ by restriction. A choice of Riemannian metric embeds the cosphere bundle $S^*M$ in $T^*M$ as the unit cosphere bundle. Then  the coset $\sigma + \sT^*N \subset  \sT^*M$ is a representative of a W-hypersurface in $\sT^*M$ with skeleton equal to $\sigma$.

\end{example}

  \begin{example}[Ribbons of Legendrians]\label{ex:ribbon} An important class of W-hypersurfaces  consists of {\em ribbons} of Legendrians.  Let $j:\Lambda \hookrightarrow \partial_v W$  be   a representative of a Legendrian. 
  Suppose first  that $j:\Lambda \hookrightarrow \partial_v W$  is pure in the sense of Definition~\ref{def:adLeg}.
   Then  $j$ extends to a  W-hypersurface $\wh j :\sT^*\Lambda\hookrightarrow \partial_v W $ which we refer to as a {\em ribbon} of $\Lambda$. One may similarly define the ribbon of a non-pure Legendrian; we will do so in Remark \ref{rem: stripping not pure} after introducing the notion of stripping faces.  In either case, the space of all   ribbons of a  given Legendrian is contractible. 
   \end{example}
   
   The next result states that any strong homotopy of a W-hypersurface may be realized by an isotopy of any given representative. 
   
 \begin{proposition}\label{lm:strong-hom-hyper} 
    Let $(\sY,\mu_t)$, $t\in[0,1]$, be a strong homotopy of $W$-blocks. 
    Suppose the W-block $(\sY,\lambda_0) $ is realized as a  W-hypersurface  in the vertical boundary of a defining domain $W$ for a W-block $(\sX,\lambda)$.
   Then there exists an isotopy   $h_t:\sY\to\p_vW$, $t\in[0,1]$, starting from the  inclusion $h_0:\sY\hookrightarrow \p_vW$, such that for all $t\in[0,1]$
    \begin{itemize}
    \item $h_t|_{\Skel(\sY,\lambda_0)}=h_0|_{\Skel(\sY,\lambda_0)}$;
    \item  $h_t^*\lambda=\mu_t$;
    \item $(h_t(\sY),\mu_t)$ is a representative of a W-hypersurface in $\p_vW$.
    \end{itemize}
     \end{proposition}
    \begin{proof}  The contact form $\alpha|_{\Op(\Skel(Y,\mu_0))}$ has the form $dz+\mu_0$, where $z$ is the Reeb flow coordinate normalized to be $0$ on $\sY$. 
    As all blocks  $(\sY, \mu_t)$ have the same skeleta, and $d\mu_t=d\mu_0$ for all $t\in[0,1]$ we conclude that $\mu_t=\mu_0+dg_t$, where $g_t|_{\Skel(\sY,\mu)}=0$, $t\in[0,1]$. Then the required isotopy $h_t$ can be defined 
 by the formula $$(y,0) \mathop{\mapsto}\limits^{h_t}  (y,g_t(y)), \;y\in\sY.$$ 
 which was to be proved.   \end{proof}

 \subsection{Modifications of W-blocks near their boundary faces}
  \label{sec:boundaries-adjust}
  We discuss   in this section modifications of W-blocks near their boundaries necessary for making buildings. Let us begin with the following remark.
  \begin{remark}\label{rem:shrink-extend}  
   Let $(\sX,\lambda)$ be a W-block. Given a face $\sP$ of $\sX$ with the nucleus $\sN$ there is by definition a neighborhood  of $\sN$ in $\sX$ split as $(\sN\times\sT^*[0,\delta), \lambda_{\sN}\oplus sdt)$. Sometimes it is convenient to shrink $\sX$, say by replacing $\sT^*[0,\delta)$ by $\sT^*[\delta/2,\delta)$, or extend $\sX$ by replacing $\sT^*[0,\delta)$ by $\sT^*[-\delta,\delta)$, see the discussion in Section \ref{sec: collar ext/rest}.  As there exist obvious Liouville isomorphisms between the original, shrinked, or expanded versions of $\sX$,  we will  in most cases  not distinguish between these manifolds, and slightly abusing the terminology may address the homotopy of the extended W-block structure  as a  homotopy of the original one.
\end{remark}

  

\subsubsection{Basic models}

Consider 
the complex line $\C$ with coordinate $s+it$.

Given a  smooth strictly  subharmonic function $\phi:\C\to\RR$, i.e.~a function with $\Delta\phi=\phi_{ss}+\phi_{tt}>0$,
we have a Liouville form $\mu_\phi = d^\C\phi=\phi_sdt-\phi_tds$, the corresponding symplectic form $\om_\phi=\Delta_\phi ds\wedge dt,$ 
and the  corresponding Liouville  field $Y_\phi=\frac{1}{\Delta_\phi}( \phi_s\frac{\p}{\p s}+\phi_t\frac{\p}{\p t})$. Note $Y_\phi$  is the gradient of $\phi$ with respect to the metric $g_\phi=\Delta_\phi(ds^2+dt^2)$.

\begin{example}\label{ex:special-adj}

First, here are some simple examples of the above setup.

\begin{itemize}
\item[ (a)] $\phi=\frac{s^2}2;\; \mu_\phi=sdt,\;\om_\phi=ds\wedge dt,\;  Y_\phi=s\frac{\p}{\p s}.$
\item[(b)] $\phi=\frac{t^2}2;\; \mu_\phi=-tds,\;\om_\phi=ds\wedge dt,\;  Y_\phi=t\frac{\p}{\p t}.$
\item[(c)] $\phi=\frac{s^2}2-\frac{t^2}4;\; \mu_\phi= sdt+\frac{tds}2 ,\;\om_\phi=\frac12ds\wedge dt,\;  Y_\phi= s\frac{\p}{\p s}-\frac{t}2\frac{\p}{\p t}.$
\end{itemize}

Next, here are some  useful interpolations between the above basic examples.

\begin{itemize}
\item[(d)] Given $\eps\in(0,1)$,  consider a  smooth function $\theta:\RR\to\RR$ such that
  $$\theta(t)=\begin{cases} 0,&t\geq\frac{\eps}2\\
\frac{\eps^2- t^2}2,&t\leq \frac\eps4
\end{cases}
$$ 
with $1>\theta''(t) > -2$,  for all $t$, and $\theta'(t)<0$, for $0<t<\frac{\eps}2$.

For $\phi=\frac12 (s^2+\theta(t))$, we obtain the symplectic form $\omega_\phi = (1+\frac{1}{2}\theta''(t) ) ds \land dt$, the Liouville form $\mu_\phi=sdt-\frac{1}{2} \theta'(t)ds$ and  corresponding Liouville field
$Y_\phi= \frac1{1+\frac12\theta''(t)}\left(s\frac{\p}{\p s}+\frac{\theta'(t)}2\frac{\p}{\p t}\right)$. 

Note   the Liouville form $\mu_\phi$
 interpolates between   $sdt+\frac{tds}2$  from (c) for $t<\frac\eps5$, and  $sdt$ from  (a)  for $t>\frac\eps2$. Likewise,  the Liouville field  $Y_\phi$ interpolates between  $s\frac{\p}{\p s}-\frac{t}2\frac{\p}{\p t}$ from (c) 
  for $t<\frac\eps5$, and  $s\frac{\p}{\p s}$ from  (a)  for $t>\frac\eps2$. 
  
\item[(e)] Fix a smooth non-decreasing cut-off function $\zeta:\RR\to [0,1]$ with
$$\zeta(t)=\begin{cases} 1,&t \in [\frac{\eps}5,\infty)\\
0,&t\in (-\infty, \frac\eps6]  
\end{cases}
$$ 
and $\zeta(t)>0$ for $t>\sigma$. 
Set $\phi(s,t)= \frac{1}{2} \zeta(t) s^2-\theta(t) $, where $\theta$ is the function constructed in (d). Then  there exists $\rho>0$ such that $\phi(s,t)$ is strictly subharmonic  on   $\{|s|<\rho\}$ with symplectic form $\omega_\phi = \Delta_\phi ds \land dt$ for $\Delta_\phi =  ( \zeta(t) + \frac{1}{2}\zeta''(t)s^2-\theta''(t) )$. The  resulting Liouville form
$\mu_\phi = s\zeta(t) dt - (\frac{1}{2}\zeta'(t)s^2-\theta'(t) ) ds$  and Liouville field $Y_\phi = \frac{1}{\Delta_\phi} ( s \zeta(t) \frac{\p}{\p s} + (\frac{1}{2}\zeta'(t)s^2-\theta'(t) ) \frac{\p}{\p t}) $ interpolate between
the form $-tds$ and   field $t\frac{\p}{\p t}$ from (b)  for $t\leq\frac\eps6$, and the form $sdt$ and field $s\frac{\p}{\p s}$ from  (a) for $t\geq\frac{\eps}2$. Note that the $\frac{\p}{\p s}$ component of $Y_\phi$ is a positive multiple of $s \frac{\p}{\p s}$ for $t>\frac\eps6$.

 \item[(${\rm e}'$)] We will be also using a slight modification of the previous example.
 Namely,  take  $\phi(s,t)= \frac{1}{2} \zeta(\eps-t)s^2-\theta(\eps-t) $. Then $\mu=\mu_\phi$ on $\{|s|<\rho \}$ is equal 
 $ sdt$  for $t<\eps/2$ and $\mu=( t-\eps)ds$ for $t>\frac{4\eps}5  $. Respectively the
  Liouville field  $Y$ of $\mu$ is equal to  $ s\frac{\p}{\p s}$ for $t<\eps/2$ and  to $Y=( \eps-t)\frac{\p}{\p t}$  for $t>\frac{4\eps}5   $.

\item[(f)]  Choose  a  smooth function   $\eta:\RR\to\RR$ such that 
$\eta(t)=0 $ for $t\geq 0$, and $\eta'(t)<0$ and $\eta''(t)>0$ for $t<0$. 

Then $\phi(s,t)=\frac{s^2}2+\eta(t)$
is strictly subharmonic. 
 The resulting Liouville form $\mu_\phi=sdt-\eta'(t)ds$  and Liouville field $Y_\phi=\frac1{1+\eta''(t)}\left(s\frac{\p}{\p s}+\eta'(t)\frac{\p}{\p t}\right)$   coincide with the form  $sdt$ and  field $s\frac{\p}{\p s}$  from  (a) on the half-plane $t\geq 0$.  
 
 Note that the $\frac{\p}{\p t}$ component of $Y_\phi$ is negative on the open half-plane $t<0$, i.e. it is a positive multiple of  $t\frac{\p}{\p t}$.
 \item[(g)] Let us modify the function from example (e)
 by setting
  $\phi(s,t)=\frac{1}2\zeta(|t|) s^2-\theta(|t|)$.
  Then  the modified function $\phi(s,t)$ is  still strictly subharmonic on $\{|s|<\rho\}$ for a sufficiently small $\rho>0$.
  The Liouville form $\mu_\phi$ interpolates between $sdt$ for $|t|\geq\frac{\eps}2$ and $-tds$ for  $|t|\leq\frac\eps5$.
  
 A similar modification $\phi=\frac12(s^2+\theta(|t|))$ of the example (d)  is a strictly subharmonic function which generates
 $\mu_\phi$ which interpolates  between $sdt+\frac{tds}2$  for $|t|<\frac\eps5$ and $sdt$ for $|t|>\frac\eps2.$

\end{itemize}
\end{example}

 \begin{figure}[h!]
 \centering
  \includegraphics[scale=0.3]{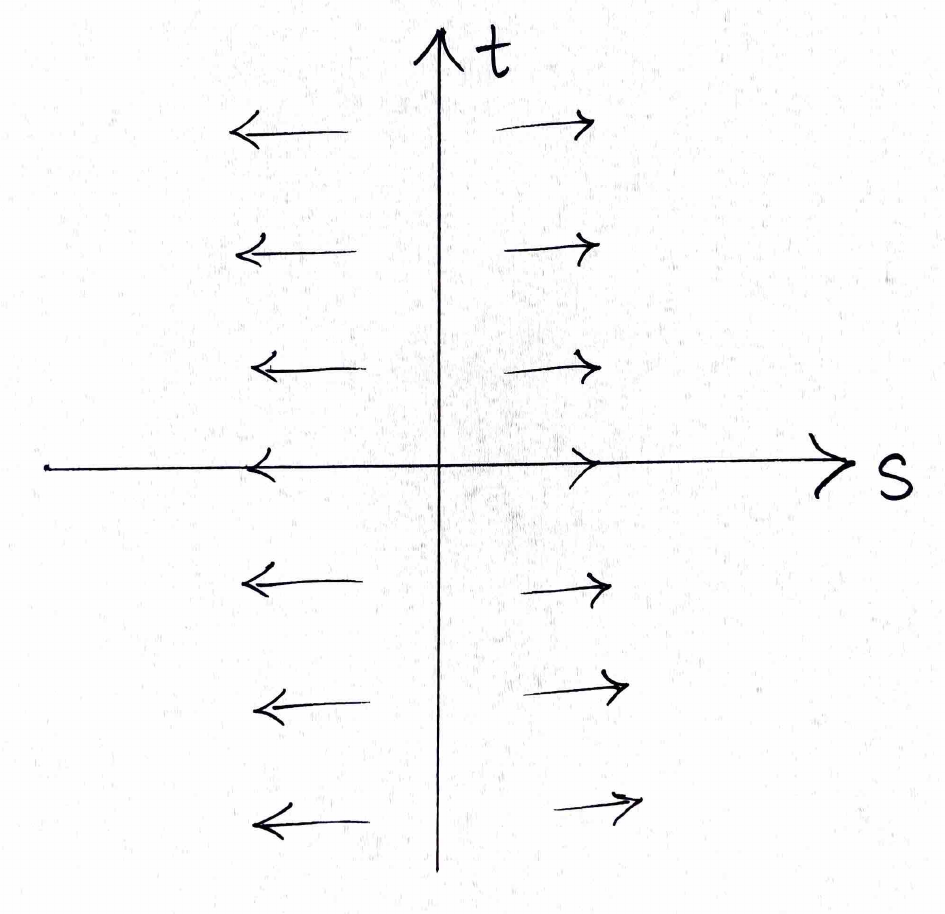}
 \caption{model (a)}
 \label{model-(a)}
 \end{figure}
 
  \begin{figure}[h!]
 \centering
  \includegraphics[scale=0.3]{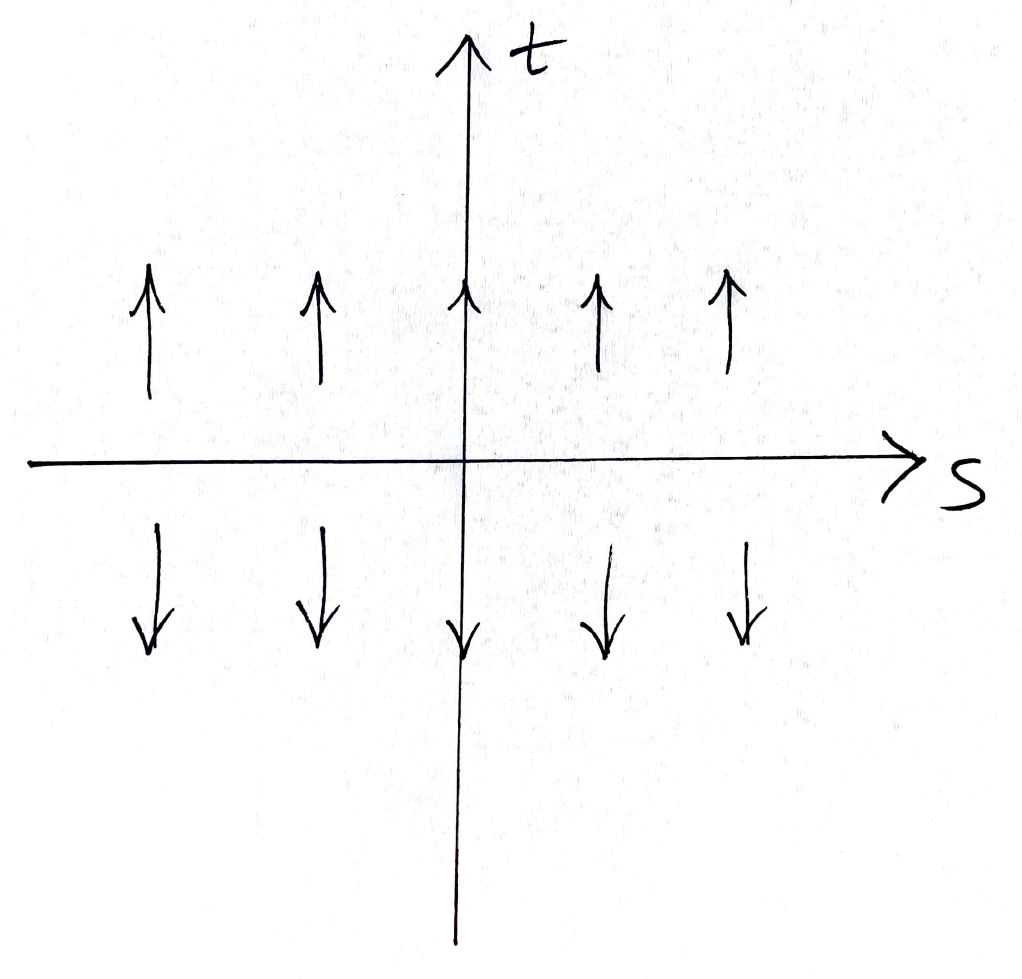}
 \caption{model (b)}
 \label{model-(b)}
 \end{figure}
 
  \begin{figure}[h!]
 \centering
  \includegraphics[scale=0.3]{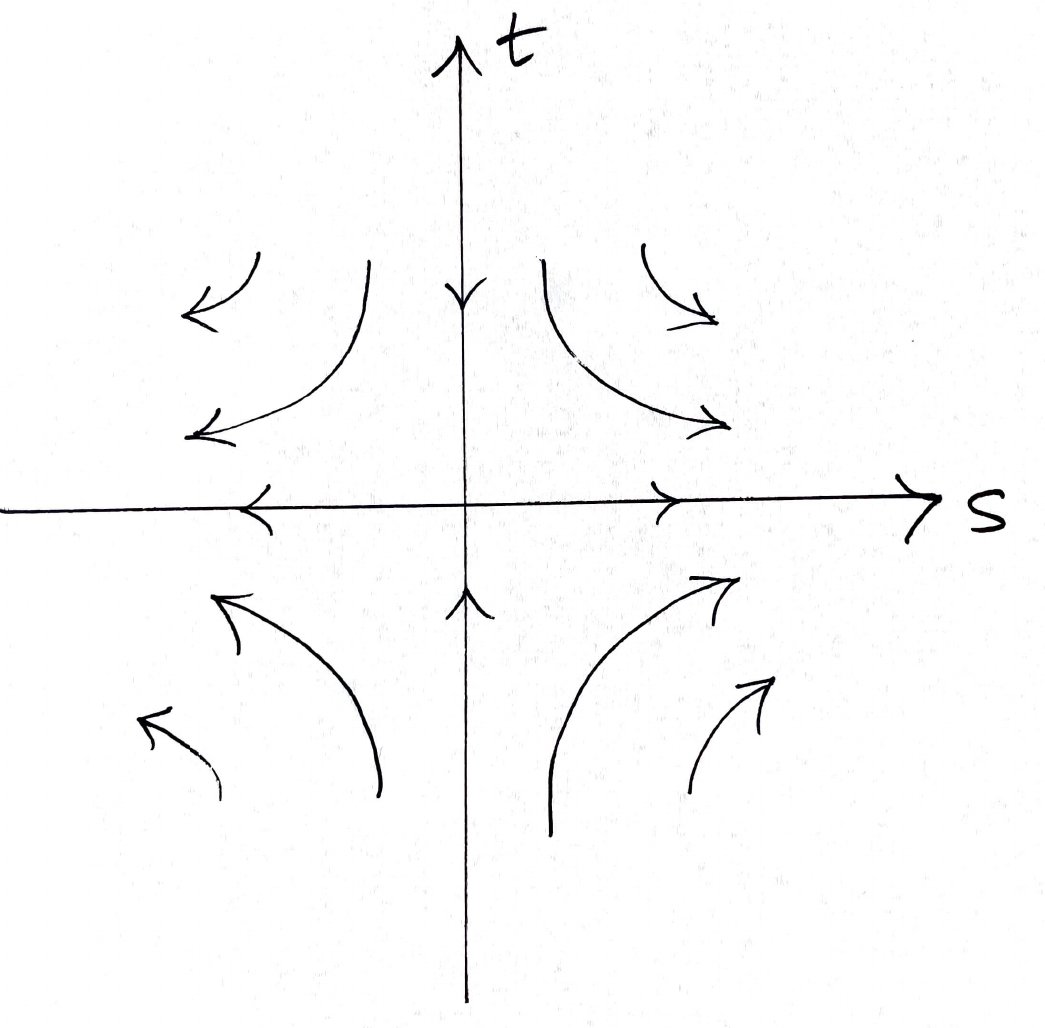}
 \caption{model (c)}
 \label{model-(c)}
 \end{figure}
 
  \begin{figure}[h!]
 \centering
  \includegraphics[scale=0.3]{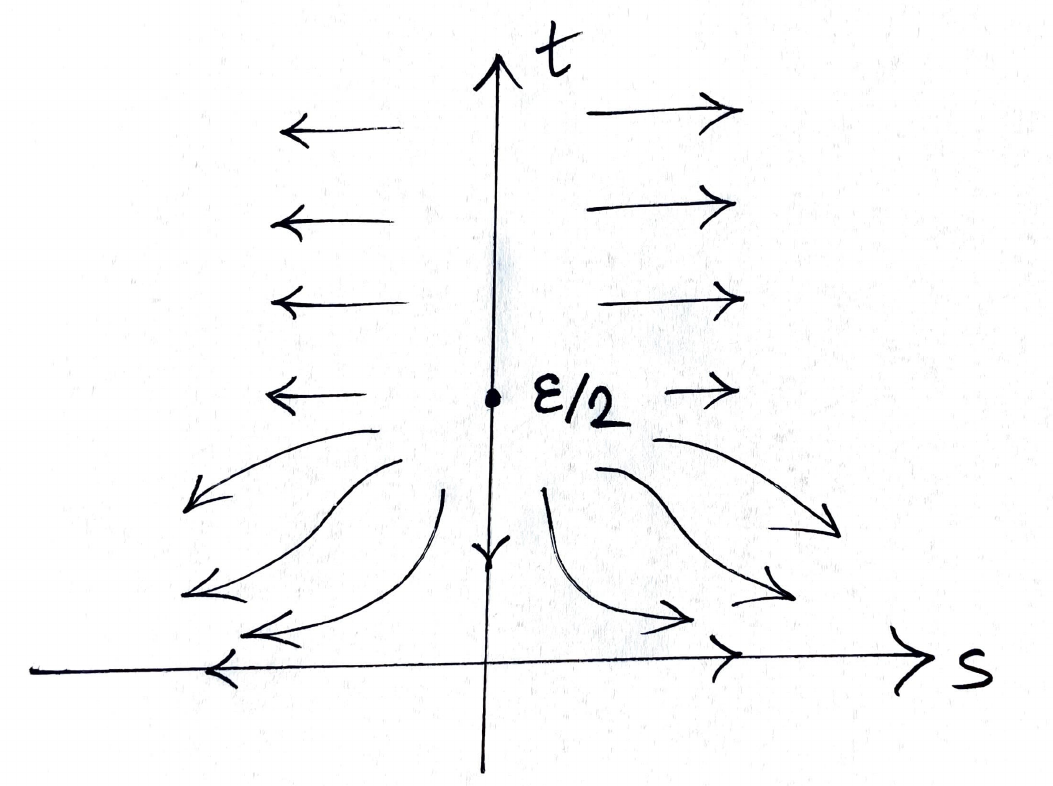}
 \caption{model (d)}
 \label{model-(d)}
 \end{figure}
 
  \begin{figure}[h!]
 \centering
  \includegraphics[scale=0.3]{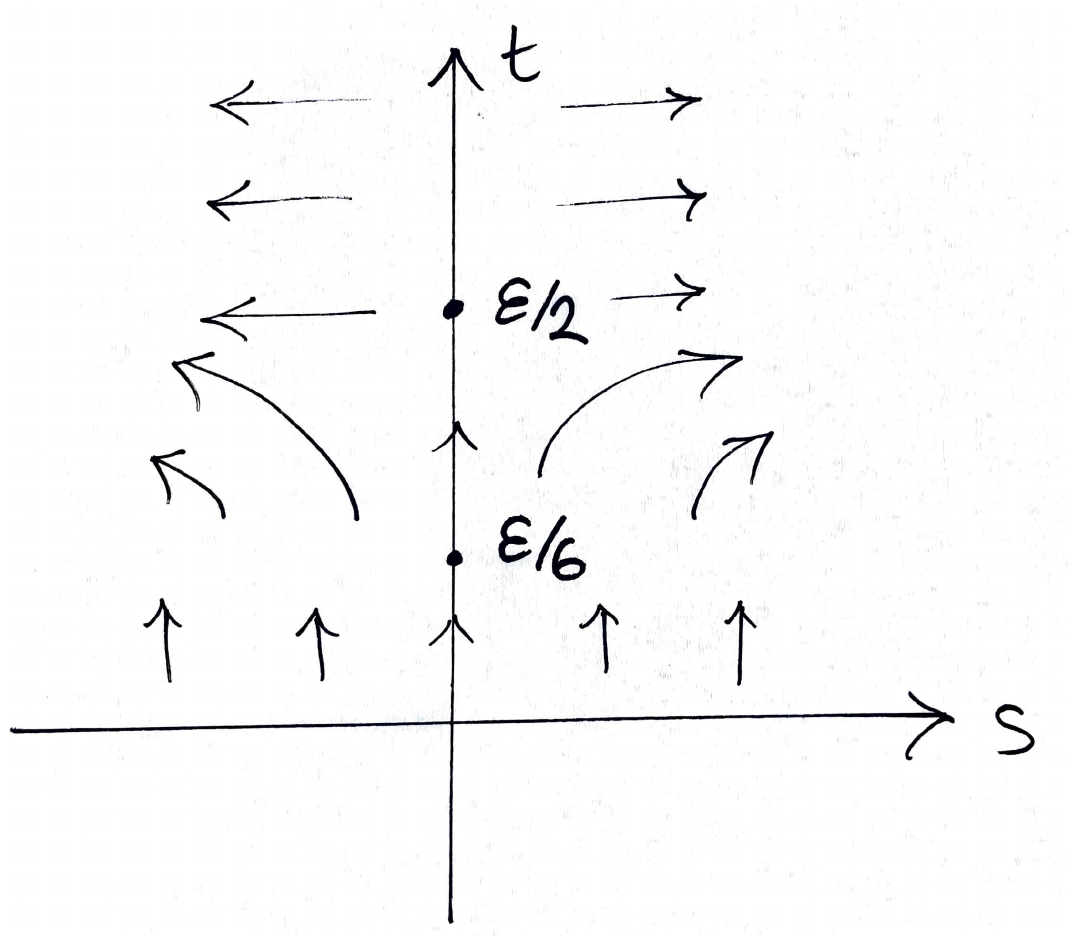}
 \caption{model (e)}
 \label{model-(e)}
 \end{figure}
 
  \begin{figure}[h!]
 \centering
  \includegraphics[scale=0.3]{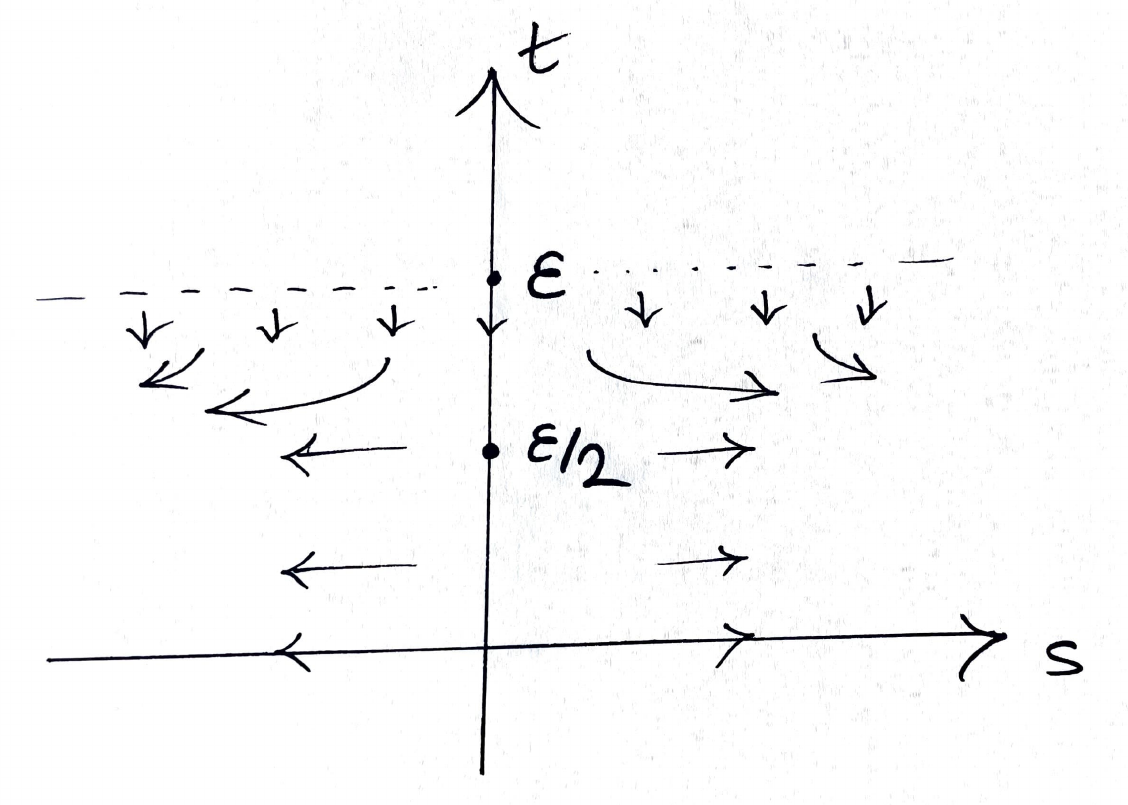}
 \caption{model (e')}
 \label{model-(e')}
 \end{figure}
 
  \begin{figure}[h!]
 \centering
  \includegraphics[scale=0.3]{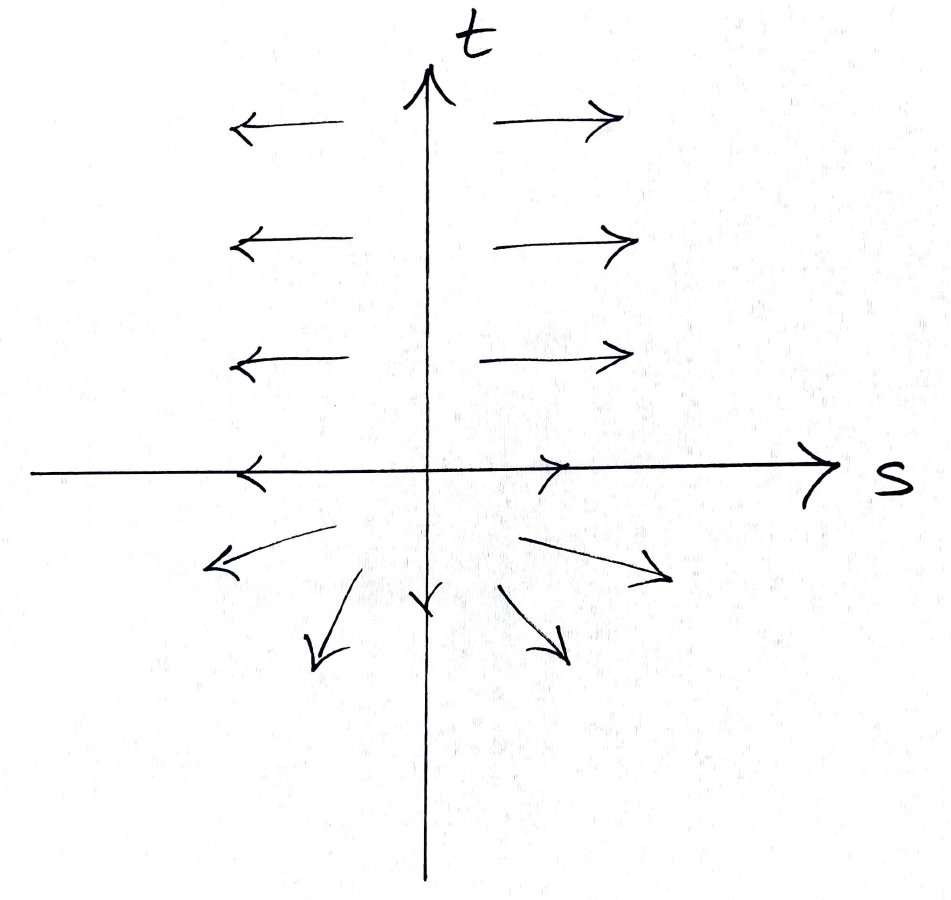}
 \caption{model (f)}
 \label{model-(f)}
 \end{figure}
 
  \begin{figure}[h!]
 \centering
  \includegraphics[scale=0.3]{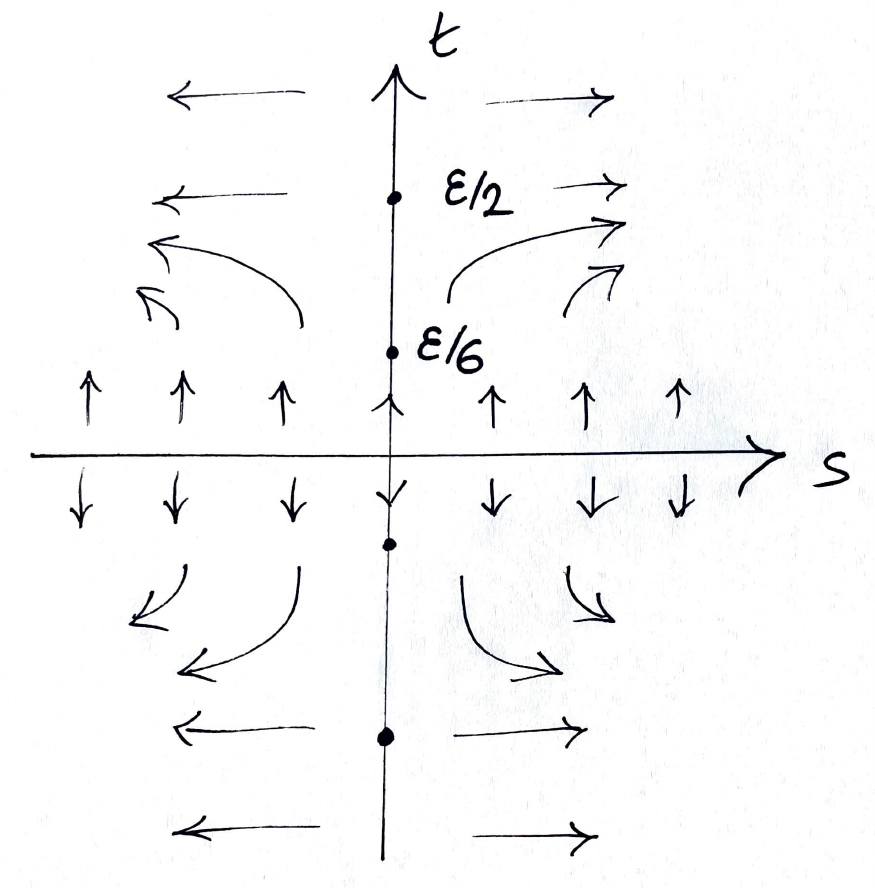}
 \caption{model (g)}
 \label{model-(g)}
 \end{figure}
 
   \begin{figure}[h!]
 \centering
  \includegraphics[scale=0.3]{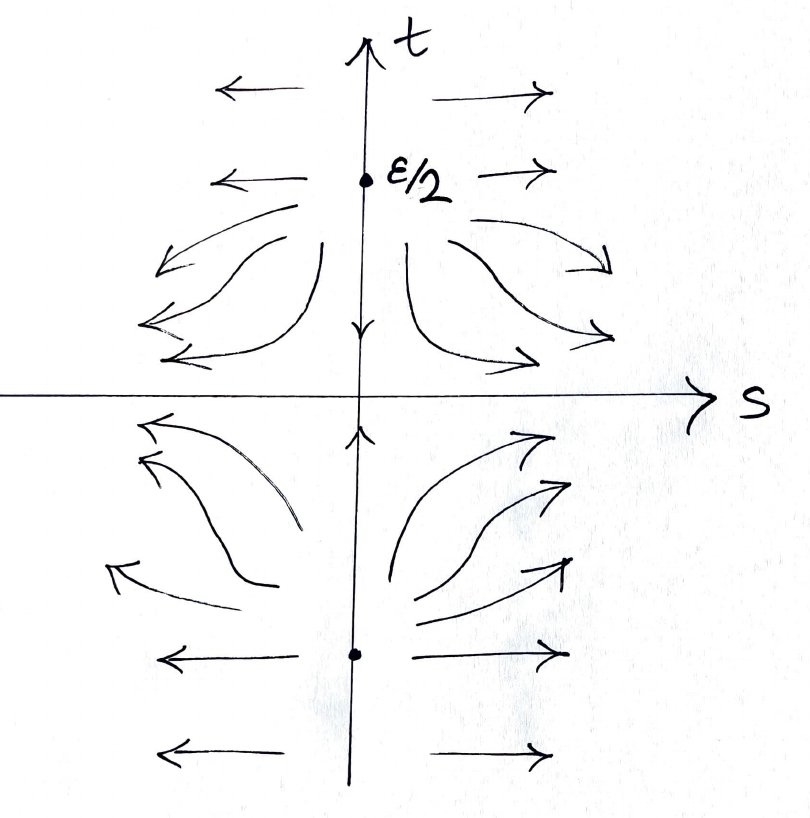}
 \caption{model (g')}
 \label{model-(g')}
 \end{figure}
  
 \begin{remark}\label{rem:modif-fixed-sympl} In the   above modifications of Liouville forms,  the underlying symplectic structure was not necessarily fixed.  While this is irrelevant for most applications, it is useful to note that using Remark \ref{rem:pot-inv} this   can be easily fixed by adjusting the models by appropriate diffeomorphisms. For instance, consider   model e). Take a  function $\delta:\Op\{s=0\}\to\R_+$ given by the formula $\delta(s,t)=(\zeta(t)-\theta'(t))s+\frac16\zeta''(s)t^3.$ Note that  $\frac{\p\delta}{\p s}=\Delta_\phi$. Define a diffeomorphism $h:\Op\{s=0\}\to\Op\{s=0\}$   by the formula
 $(t,s)\mapsto (t, \delta(s,t))$. 
  Then $h^*(ds\wedge dt)=\Delta_\phi ds\wedge dt$ and $h_*{\om_\phi}=ds\wedge dt.$
  Thus, according to Remark \ref{rem:pot-inv} the vector field $h^*Y_\phi$ is the gradient vector field of the function $\phi\circ h^{-1}$ with respect to the standard Euclidean metric $ds^2+dt^2$. It is a Liouville field for $ds\wedge dt$ dual to the Liouville form $h_*{\mu_\phi}$. 
   Note that $h$ is the identity for $t\notin \left(\frac\eps 6,\frac\eps2\right)$.
 \end{remark}

\subsubsection{Boundary attraction}\label{sec:b-attraction}

We continue with 
the complex line $\C$ with coordinate $s+it$.
Consider as well the real line $\R$ with coordinate $t$, and its cotangent bundle $T^*\R$ with Darboux coordinates $t, s$. Using these coordinates, let us identify $T^*\R \simeq \C$. In particular, we identify the germ $\sT^*[0,\eps)$ along its zero-section $[0, \eps)$ with the germ of $\RR \times i[0, \eps)$ along the imaginary interval $0 \times i[0, \eps)$.
  
    \begin{figure}[h]
\includegraphics[scale=0.3]{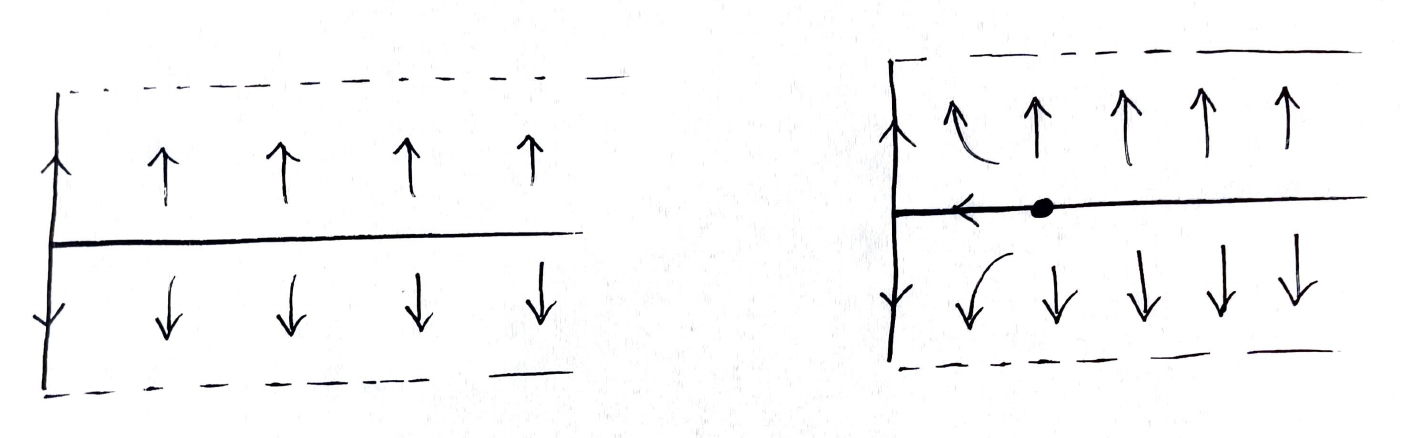}
\caption{Making the boundary attracting.}
\label{boundary-attracting}
\end{figure}

Let $\mu=\mu_\phi$ be the Liouville form on $\sT^*[0,\eps)$ from Example \ref{ex:special-adj}(d).
  \begin{definition}\label{def:attract}
  Let $\sP$ be a facet of $(\sX,\lambda)$ with the nucleus $(\sN,\lambda_\sN)$.
Let us   replace  $\lambda$ by a new Liouville form  
 $$
 \lambda^{\sP,\attract} =\lambda+(\mu-sdt) = \lambda_\sN + \mu 
 $$
 where we  identify $\mu-sdt$ with  its pullback along the splitting map $\sU_\sP\to \sT^*[0, \eps)$, extending it by 0 on the complement. 
We call this operation {\em making the facet $\sP$ attracting}.
Given a collection of faces $\bP=\{\sP_1,\dots,\sP_k\}$, repeating the construction   for  each facet $\sP\in\bP$ we get    a new Liouville form
$$
 \lambda^{\bP,\attract} =\lambda+\mathop{\sum}\limits_{i = 1}^k(\mu_i-s_idt_i)
 $$
which we say is {\em attracting} along the faces of $\bP$.

When $\bP$ is the collection of all faces of 
 $\sX$, 
we write $\lambda^{\attract}$ in place of   $\lambda^{\bP,\attract}$.
 
 \end{definition}

Note that  $(\sX,\lambda^{\bP,\attract})$  is not   a W-block due to the behavior of the  Liouville form near the boundary.  We call   $(\sX,\lambda^{\bP,\attract})$  a {\em W-block with attracting boundary}, and sometimes still call it W-block, if the boundary behavior is clear from the context.
The main advantage of   W-blocks with attracting boundary is that  unlike W-blocks which cannot have a Morse potential due to the postulated boundary behavior, {\em any W-block $(W,\lambda)$ with attracting boundary can be deformed via a $C^\infty$-small deformation to a W-block with attracting boundary with a Morse potential} (comp. Corollary \ref{cor:W-Morse} ).

 W-blocks with attracting boundary share a lot of properties of standard W-blocks. For instance, we have the following straightforward observation.

  \begin{lemma}\label{lm:Lyap-adjust}
There exists a defining domain $W$ for $(\sX,\lambda)$  
whose  vertical boundary $\partial_v W$  remains  transverse to the modified  Liouville field 
  $Z^{\bP,\attract}$, and the skeleton continues to satisfy  
  $$\Skel(X, \lambda) = \bigcap\limits_{t>0}\wt Z^{-t}(W)$$
  \end{lemma}
 
Note that by attaching external collars using modifications from (g) we get back  a W-block which is  homotopic to the original block $\sX$,
 see Remark   \ref{rem:shrink-extend}.



 \subsection{Stripping faces}\label{sec:erasing-face}
 Let $(\sX,\lambda)$ be a $2n$-dimensional W-block.
 
     \begin{figure}[h]
\includegraphics[scale=0.3]{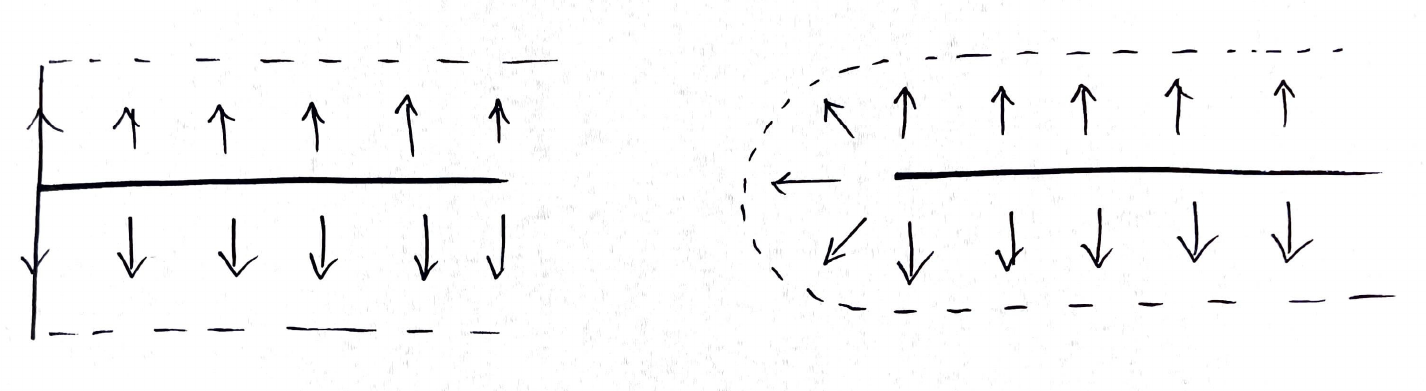}
\caption{Modification to strip boundary.}
\label{stripping-boundary}
\end{figure}

 As introduced for any manifold with corners in Section~\ref{ssec:Mbc},
 we have the collar extension $ \sX^{-\llcorner}$ whose collar restriction is the original manifold with corners $\sX =  (\sX^{-\llcorner})^{+\llcorner}$. 
 Using the collar splittings of $\sX$, by construction we have for some $\delta<0$
 $$
\sX^{- \llcorner} \setminus  \sX  = \bigcup \limits_{\mbox{$k$-faces $\sP$ with $k>0$}} \sN \times \sT^*{[\delta, 0)}^k
$$

We set  $\sX^\bigcirc = \Int \sX^{-\llcorner}$ to be the interior and regard it as a smooth manifold without boundary.

 We extend $\lambda$ to $\sX^{\bigcirc}$  in two ways.  First, we extend $\lambda$ as a 1-form, still denoted by $\lambda$, by taking its restriction to each $ \sN \times \sT^*{(\delta, 0)}^k$ to be  equal to $\lambda_\sN + \sum_{i=1}^k s_i dt_i$. Second, 
  we extend $\lambda$ as a Liouville form, denoted by $\lambda^{\bigcirc}$,  
  by taking its restriction to each $ \sN \times \sT^*{(\delta, 0)}^k$ to be  equal to $\lambda_\sN +\sum_{i=1}^k \mu_i$, where
$\mu_i$ is the  pullback  along the $i$-th collar coordinate of the form   $\mu=sdt-\eta'(t)ds$ from Example~\ref{ex:special-adj}(f) above.
    \begin{figure}[h]
\includegraphics[scale=0.2]{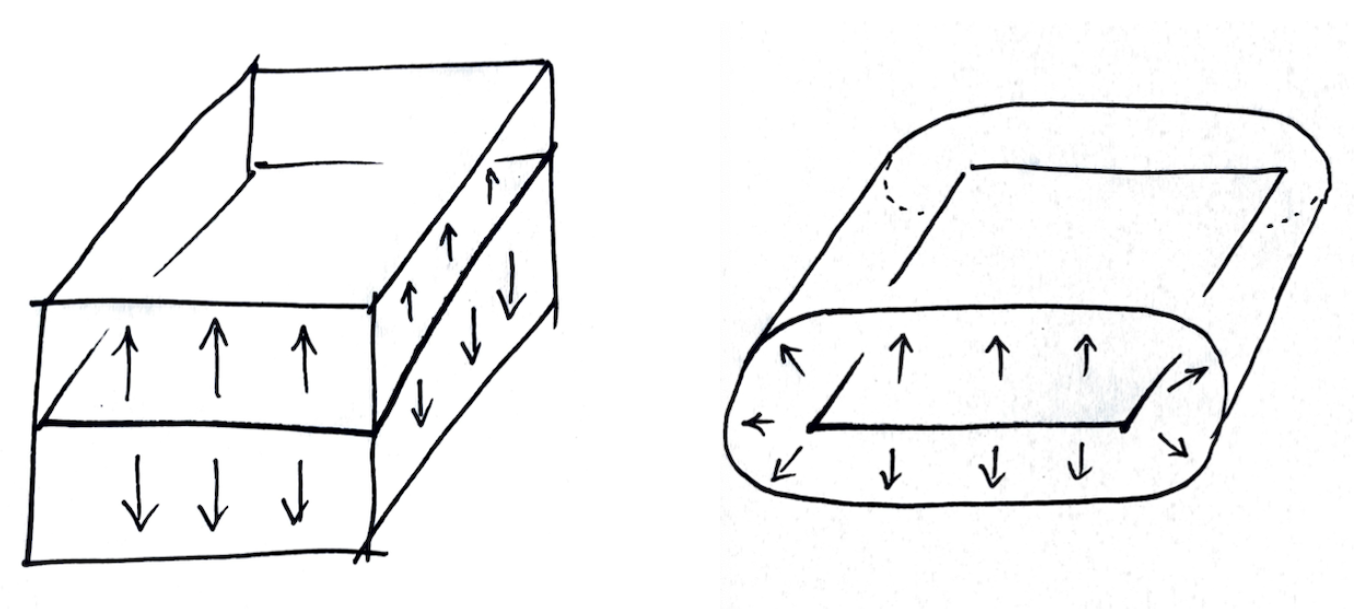}
\caption{Stripping of some boundary faces of a W-block.}
\label{partial-stripping boundary}
\end{figure}

 \begin{lemma}
 \label{lm:underlying-W-block}
$(\sX^{\bigcirc}, \lambda^{\bigcirc})$ is a W-block without boundary, i.e.~the germ of a Weinstein manifold along its skeleton
$\Skel(\sX^{\bigcirc}, \lambda^{\bigcirc}) = \Skel(\sX, \lambda)$. 
 
 \end{lemma}
\begin{proof}
It only remains to check the existence of a defining domain for
$(\sX^{\bigcirc}, \lambda^{\bigcirc})$.   Given a defining domain $W$ for $(\sX,\lambda)$, note the Liouville field $Z^\bigcirc$ associated to
$\lambda^{\bigcirc}$ is outwardly transverse along the faces of $W^{-\llcorner} \subset \sX^\bigcirc$. Thus smoothing the faces of $W^{-\llcorner}$
 to a smooth boundary, 
 preserving  transversality of   $Z^\bigcirc$,
  provides a  defining domain $W^\bigcirc$ for
$(\sX^{\bigcirc}, \lambda^{\bigcirc})$. 
\end{proof}

\begin{definition}
 We call  $ (\sX^\bigcirc,\lambda^\bigcirc)$
 the {\em underlying boundaryless W-block}   of $(\sX,\lambda)$.
 \end{definition}

 \begin{remark}
In Sylvan's terminology, a W-block may be thought of as a Weinstein manifold with ``stops" (that can have corners). From this viewpoint, taking the underlying boundaryless W-block is stop removal.
\end{remark}

More generally, consider a list $\bP=[\sP_1,\dots, \sP_k]$ of facets  of $\sX$ and let $\bP^c$ consist of all $k$-faces, $k>0$, which do not lie in the interior of $\bigcup_{i=1}^k \sP_i \subset \partial \sX$. Equivalently, $\bP^c$ consists of all $k$-faces which lie in the union of the facets not in $\bP$. In particular note that $k$-faces in the boundary of $\bigcup_{i=1}^k \sP_i \subset \partial \sX$ are in $\P^c$.  
 In this case, we similarly have the {\em $\bP$-stripped}  W-block 
$(\sX^{\bP,\bigcirc},\lambda^{\bP,\bigcirc})$ with underlying space
$$
\sX^{\bP,\bigcirc} = \sX^{\bigcirc} \setminus  \left( \bigcup \limits_{\mbox{$k$-faces $\sQ \in \bP^c$ with $k>0$}} \sN \times \sT^*{(\delta, 0)}^k\right)
$$
and restricted Liouville form $\lambda^{\bP,\bigcirc}= \lambda^{\bigcirc}|_{\sX^{\bP,\bigcirc}}$.
Note that $\sX^{\bP,\bigcirc}$ only depends on the underlying set $\bP$ not its ordering.
 
Note that   stripping away  a  face $\sP$ leaves its nucleus $\sN$ as a (representative of a) Weinstein hypersurface $\sN$ in the block $\sX$
Hence, we call sometimes the operation $\sX \mapsto  \sX^{\sP,\bigcirc} $ a { \em  nucleus   to   W-hypersurface conversion}.

 When $\bP$ is a list of all faces, then $(\sX^{\bP,\bigcirc},\lambda^{\bP,\bigcirc})=(\sX^\bigcirc,\lambda^\bigcirc)$ is the underlying boundaryless W-block of $(\sX, \lambda)$.

 Now we introduce {\em  cotangent blocks}, a generalization of the {\em proper cotangent blocks} of 
 Definition~\ref{ex:co-bl}.

 \begin{definition}[Cotangent blocks]\label{ex:gen-co-bl}

 Let $M$ be a compact  $n$-dimensional smooth manifold with corners, and $(\sT^*M, pdq)$ its $2n$-dimensional proper cotangent block. Recall the $k$-faces  $\sP \subset \sT^*M$ correspond  to $k$-faces $P \subset M$ by the formula $\sP = \sT^*M|_P$.
Given a  list $\bP=[\sP_1,\dots, \sP_k]$    of  facets of $\sT^*M$, we call  the $\sP$-stripped W-block  
$((\sT^*M)^{\bP,\bigcirc},(pdq)^{\bP,\bigcirc})$
 a {\em  cotangent block}.

\end{definition}
   \begin{remark}\label{rem: stripping not pure}
   
Suppose that   $j:\Lambda \hookrightarrow \partial_v W$ is a representative of a Legendrian which is not necessarily pure. Let $P_1,\dots, P_k$ be all of the facets of $\Lambda$ that fail the purity test of Definition~\ref{def:adLeg}.
Let $\bP=[\sT^*P_1,\dots\sT^*P_k]$ be the corresponding facets of the proper cotangent block $\sT^*\Lambda$,
and 
 $(\sT^*\Lambda)^{\bP,\bigcirc}$ the   cotangent block obtained by stripping the faces of $\bP$.
Then $j$ extends to  a W-hypersurface embedding $\wh j:(\sT^*\Lambda)^{\bP,\bigcirc}\hookrightarrow \partial_v W$
which we refer to as a {\em ribbon} of $\Lambda$. 
\end{remark}
 
 \subsection{Partial order on W-blocks}
 
 Now let us  define a partial order on W-blocks using the operations of stripping faces and smoothing corners.

 \begin{definition}
Given $2n$-dimensional W-blocks   $(\sX_1,\lambda_1)$, $(\sX_2,\lambda_2)$, we say that
$(\sX_1,\lambda_1)$  is an {\em ancestor  }  of    $(\sX_2,\lambda_2)$, 
and $(\sX_2,\lambda_2)$  is a {\em descendent }  of    $(\sX_1,\lambda_1)$, 
and write  
$$
 (\sX_1,\lambda_1) \prec (\sX_2,\lambda_2) 
$$
if   $(\sX_1,\lambda_1)$ is  obtained from $(\sX_2,\lambda_2)$  (up to  deformation equivalence) by  a sequence of operations  
 of stripping faces and smoothing corners. 
 
\end{definition}
    \begin{figure}[h]
\includegraphics[scale=0.3]{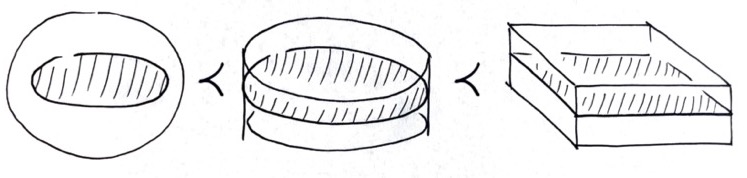}
\caption{An example of the partial order of ancestry.}
\label{fig:partial order}
\end{figure}

\begin{figure}[h]
\includegraphics[scale=0.2]{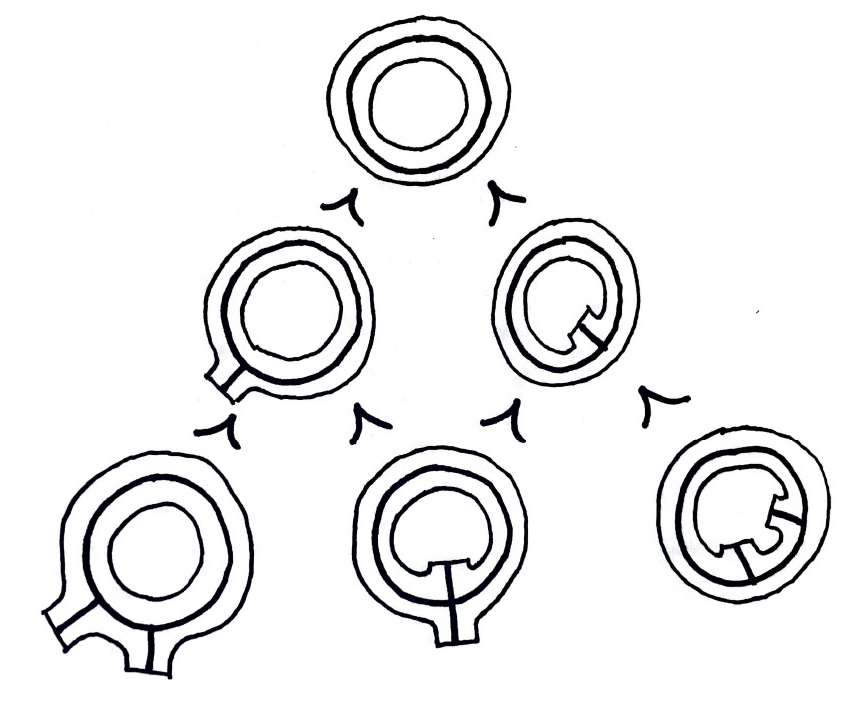}
\caption{Another example of the order of ancestry, here with $T^*S^1$ the common ancestor.}
\label{ancestrytree}
\end{figure}

 It is straightforward to check that $\prec$ gives a partial order on   $2n$-dimensional W-blocks up to  deformation equivalence.
Moreover, any W-block $(\sX, \lambda)$ has 
ancestor $(\sX^\bigcirc, \lambda^\bigcirc)   \prec  (\sX,\lambda) $ 
its underlying boundaryless W-block  $(\sX^\bigcirc, \lambda^\bigcirc)$
obtained by stripping all its boundary faces, which up to deformation equivalence can be equivalently obtained by any sequence of  stripping faces and smoothing corners. 
Furthermore, up to deformation equivalence, $(\sX^\bigcirc, \lambda^\bigcirc) $ has only itself as ancestor, in other words is initial for the order $\prec$.

\begin{example}  Here is a simple and instructive sequence of ancestors/descendants to keep in mind, where we denote $\sR^{2n}$ the germ of $\RR^{2n}$ at the origin.
$$
(\sR^{2n},\frac12 (pdq-qdp) \prec (\sT^*D^n,pdq) \prec  (\sT^*[0,1]^n,pdq) ) 
$$  
The second W-block $(\sT^*D^n,pdq) $ is obtained  from the third W-block $(\sT^*[0,1]^n,pdq)$ by smoothing the corners, and the first  W-block $(\sR^{2n},pdq-qdp)$ is obtained  (up to deformation equivalence) from the second W-block $(\sT^*D^n,pdq) $ by  stripping  the boundary face. See Figure \ref{fig:partial order}. There are also various other intermediate possibilities given by stripping faces and smoothing corners in various patterns.
\end{example}

 \subsection{Converting a W-hypersurface to a face nucleus}
            \begin{figure}[h]
\includegraphics[scale=0.3]{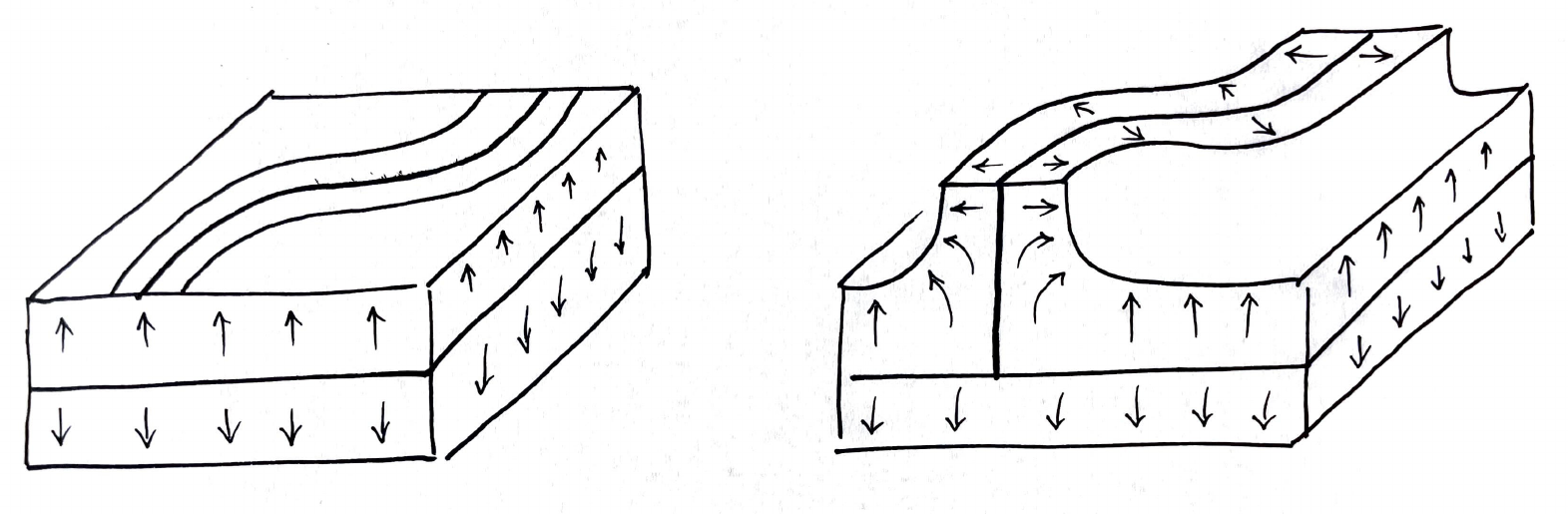}
\caption{Converting a W-hypersurface to a face nucleus}
\label{hyp-conversion}
\end{figure}

Let $(\sX,\lambda)$ be a W-block,  let $Z$ be the corresponding Liouville field and let $\sN\subset\p_\infty\sX$ be a W-hypersurface. Choose a defining domain $W$ for $X$ and choose a defining domain $V\subset\p_vW$ for a   representative of $ \sN$ in the  vertical boundary of  $W$. Denote $\lambda_V:=\lambda|_V$ and choose a function $\psi:V\to\RR_+$ such that $\{\psi=0\}=\Skel(\sN)$ and such that $V_\sigma:=\{\psi\leq \sigma\}$  are defining domains for $\sN$ for sufficiently small $\sigma>0$. Concretely, let $Z_{\sN}$ be the Liouville  field for $\sN$.  For any point $x\in V\setminus\Skel(\sN)$ denote $\wt\psi(x)=e^{-t}$  if $Z_{\sN}^t(x)\in\p V$. Note that $\wt\psi|_{\p V}=1$, and $\lim\wt\psi(x)=0$ when $x\to\Skel(\sN)$. Hence we can continuously extend  $\wt\psi$ to $\Skel(\sN)$ as equal to $0$ on $\Skel(\sN)$. Define now the required function $\psi:V\to\R$ by the formula $\psi:=\theta\circ \wh \psi$, where  $\theta: [0,\infty)\to[0,\infty)$ is a  smooth function with positive derivative on $(0,\infty)$, and equal  to $1$ at $1$, and such that $\theta^{(k)}(0)=0$ for all $k=0,1,2,\ldots$

 Consider  the domain $U:=\{|s|<\rho\}\subset (T^*[0,\eps), sdt)$ endowed with the Liouville form $\mu=\mu_\phi$ from Example \ref{ex:special-adj} (${\rm e}'$). Recall that $\mu=(t-\eps)ds$ for $t\geq\frac{4\eps}5  $ and  equal $sdt$ for $t\leq\frac\eps2$.   
For a sufficiently small $\sigma>0$ denote \begin{align*}
&Y_\sigma:=\{(x,(s,t))\in V\times U; \psi(x)+s^2\leq\sigma^2\}, \\
& S_\sigma=\{(x,(s,t))\in V\times U; \psi(x)+s^2<\sigma^2\},\; \\
&\Sigma_\sigma:=\{t=\frac{4\eps}5\}\cap Y_\sigma.\end{align*} 
Note that the negative Liouville flow  of the Liouville field of the form $\lambda_V+\mu_\phi$ is transverse to $\Sigma_\sigma$, and the negative   
flow of this Liouville field is complete in $Y_\sigma$. Denote by $Y'_\sigma$ the negative Liouville cone of $\Sigma_\sigma$ in $Y_\sigma$. Set $\wh Y_\sigma = Y_\sigma'  \cup(Y_\sigma\cap\{t<\frac {4\eps}5\}$).

 We claim that there exists a Liouville embedding $J:(Y'_\sigma, \lambda_{V}+(t-\eps)ds)\to (W,\lambda)$ such that $J(\{t= \frac{4\eps}5\}\cap Y'_\sigma) \subset \p_vW$, and $J(V\times \{ s=0, t= \frac{4\eps}5 \} )=V$. Indeed, 
  $V$  is transverse to the Reeb field $R$ of the contact form $\lambda|_{\partial_v W}$. Choosing  the flow coordinate of the vector field $\frac{5R}{\eps}$ as the  coordinate $s$ on $\Op_{\p_vW}V\subset\p_vW$  normalized to be equal to $0$ on $V$, we compute  that $\lambda_ {\Op_{\p_vW}V}=\lambda_{\sN}-\frac{\eps}5ds$.
  Note that the form $\lambda_{V}+\mu_\phi$ restricted to  $\Sigma_\sigma=\{t=\frac45\}\cap Y_\sigma$  is also equal to $\lambda_{V}-\frac{\eps}5ds$. Hence, there is a contact form preserving embedding $$J:\left(Y'_\sigma\cap\{t=\frac{4\eps}5\},\lambda_V-\frac{\eps}5ds\right)\to ( \p_vW,\lambda|_{\p W}).$$ We can then extend it to $Y'_\sigma$  as   the required Liouville embedding $J:Y'_\sigma\to W$ by matching the corresponding  negative Liouville trajectories.

Let us glue $\wh Y_\sigma$ to $W$ using $J$, i.e. define 
 $$(W_{+\sN},\lambda_{+\sN}):=(\wh Y_\sigma,\lambda_{\sN}+\mu_\phi)\mathop{\cup}\limits_{x\sim J(x),x \in Y_\sigma'}(W,\lambda).$$   
 Note that  the boundary of $W_{+\sN}$ is  the  union  $$(\p_vW\setminus \Int J(\Sigma_\sigma))\cup( S_\sigma\cap \{t\leq \frac {4\eps} 5\})\cup\p_h W\cup(Y_\sigma\cap \{t=0\}).$$ 
 The  Liouville vector  field $Z_{+\sN}$ of the form $\lambda_{+\sN}$ is transverse in the outward sense      $S_\sigma\cap \{t\leq \frac {4\eps}5\}$ and 
 $\p_vW\subset W \setminus \Int J(\Sigma_\sigma)$,  which intersect transversely along  $J(\p \Sigma_\sigma)$. It is   tangent to  $\p_h W$ and $ Y_\sigma\cap \{t=0\}$.  The corner of  $(\p_vW\setminus \Int J(\Sigma_\sigma))\cup( S_\sigma\cap \{t\leq \frac {4\eps }5\})$ along $J(\p \Sigma_\sigma)$ can be smoothed keeping  transversality to $Z_{+\sN}$. The resullting domain  $\wt W_{+\sN}$ is a defining domain for a W-block $\sX_{+\sN}$ which has an additional boundary face which has  $\sN$ as its  nucleus.
 
            \begin{figure}[h]
\includegraphics[scale=0.3]{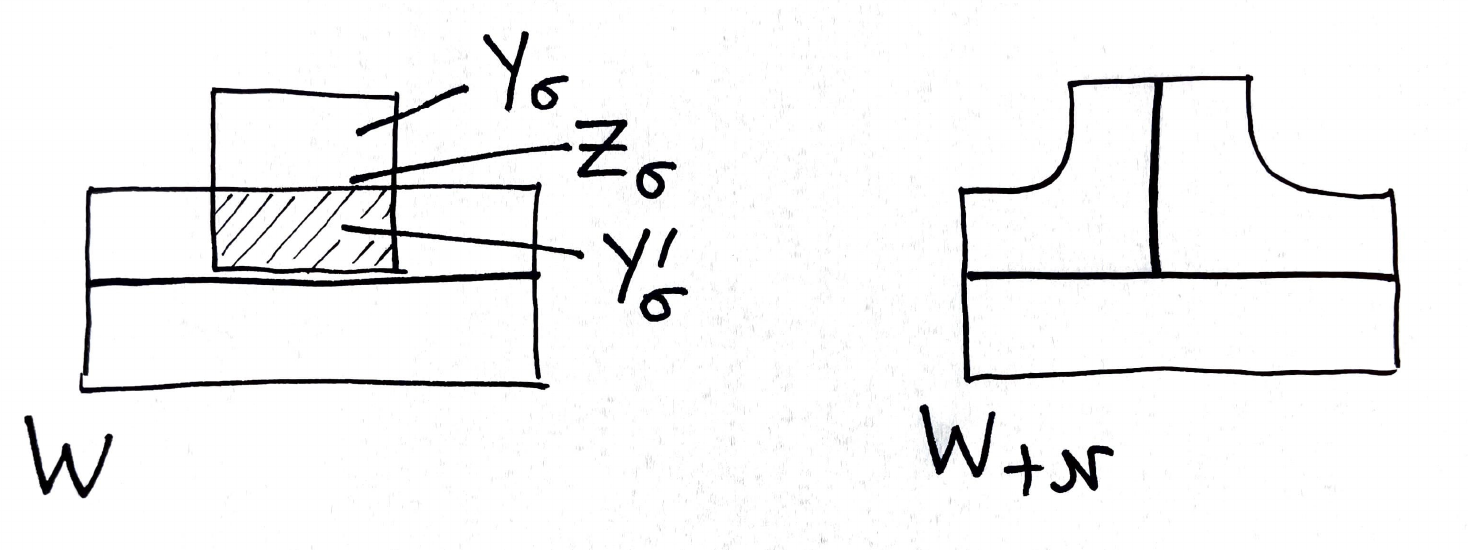}
\caption{Notation for the conversion construction.}
\label{hyp-conversion-scheme}
\end{figure}

\begin{definition}    We say that  the W-block $   (\sX_{+\sN}, \lambda_{+\sN} ) $ is the result of a {\em W-hypersurface to nucleus conversion} for the W-block $(\sX,\lambda)$ and the W-hypersurface $\sN$. \end{definition}

\begin{lemma}\label{lem:uniq conv}
Given a W-block $(\sX,\lambda)$ and a W-hypersurface $\sN \subset \p_\infty \sX$, the W-block $(\sX_{+\sN},\lambda_{+\sN})$ is well defined up to homotopically canonical strong deformation equivalence.
\end{lemma}
\begin{proof} Indeed, the choices made in the construction are, in this order, the constant  $\varepsilon>0$ and the function $\zeta$ from Example \ref{ex:special-adj}(${\rm e}'$), the choice of the defining domains $W$ for $\sX$ and $V$ for $\sN$, and the function $\psi$ and the constant the constant $\sigma>0$. Let us analyze the space of such choices as well as their effect on the W-block $(\sX_{+\sN},\lambda_{+\sN})$.

First, we can make $\varepsilon>0$ arbitrarily small from the onset and so for finite dimensional spaces of parameters we may assume that $\varepsilon$ is fixed. For fixed $\varepsilon$, the space of functions $\zeta$ satisfying the conditions of Example \ref{ex:special-adj}(${\rm e}'$) is convex, hence contractible. Varying the choice of $\zeta$ results in a strong deformation equivalence, hence the result is independent of this choice of to homotopically canonical strong deformation equivalence. 

Next, the space of defining domains $W$ for $\sX$ and $V$ for $\sN$ is also contractible. In fact, any two defining domains for the same W-block are canonically diffeomorphic (via the Liouville flow) and this diffeomorphism preserves the skeleton. Hence up to canonical strong deformation equivalence we may assume $W$ and $V$ are also fixed. 

Finally, the parameter $\sigma>0$ may also be made arbitrarily small, hence for finite parameter families we may assume it is fixed, and then the choice of $\psi$ depends on the choice of smoothing function $\theta$, the space of which is convex and hence contractible. Varying $\psi$ changes the representative $W_{+\sN}$ but leaves the W-block $(\sX_{+\sN},\lambda_{+\sN})$ unchanged. We conclude that $(\sX_{+\sN},\lambda_{+\sN})$ is well defined up to homotopically canonical strong deformation equivalence.\end{proof}

   In Section \ref{sec:erasing-face} we  discussed    the operation of  nucleus to hypersurface conversion, which was just stripping off a face of $\sX$.
    We note, that  if $\sP$ is a face of a block $\sX$ and $\sN$ its nucleus then $(\sX^{\sP,\bigcirc})_{+\sN}$ is strongly deformation equivalent to the original block $\sX$. However,  if $\sN$ is a W-hypersurface in $\sX$ then  while  $(\sX_{+\sN})^{\sP,\bigcirc}$, where $\sP$ is the created new face with the nucleus $\sN$,  is  deformation equivalent to $\sX$, it  is not strongly deformation equivalent. Indeed, the  W-hypersurface to nucleus conversion enlarges the skeleton, while  the opposite conversion operation of stripping leaves the skeleton intact.

\subsection{Vertical gluing}\label{sec:vert-gl}
Let $(\sX_1,\lambda_1)$, $(\sX_2,\lambda_2)$ be two $2n$-dimensional W-blocks. 
Suppose $\sP$ is a 1-face of $\sX_2$ with nucleus $(\sN, \lambda_2|_\sN)$ embedded as a  W-hypersurface  $j_\infty:\sN\hookrightarrow \p_\infty \sX_1$. Let us  convert it to a facet nucleus of $\sX_1$. i.e. form $((\sX_1)_{+\wt\sN},(\lambda_1)_{+\wt\sN})$, where $\wt\sN:=j(\sN)$. Let $(\sX,\lambda)$ be the W-block obtained by horizontal gluing of  $(\sX_2,\lambda_2)$ to  $((\sX_1)_{+\wt\sN},(\lambda_1)_{+\wt\sN})$ using the identification $j:\sN\to\wt\sN$ of the nuclei of the facet $\sP$ and the newly created face $\wt\sP$ of the block
  $((\sX_1)_{+\wt\sN},(\lambda_1)_{+\wt\sN})$. 

 \begin{definition}
  We say that $(\sX, \lambda)$ is the {\em vertical gluing} of the W-block $\sX_2$ along the nucleus $\sN$ of the face $\sP\subset \sX_2$ to the W-block  $\sX_1$
   along the representative of a W-hypersurface  $j:\sN\hookrightarrow  \sX_1 \setminus \Skel(\sX_1,\lambda_1)$, and denote the data of the gluing by $\sX_2 \to \sX_1$.
\end{definition}

           \begin{figure}[h]
\includegraphics[scale=0.3]{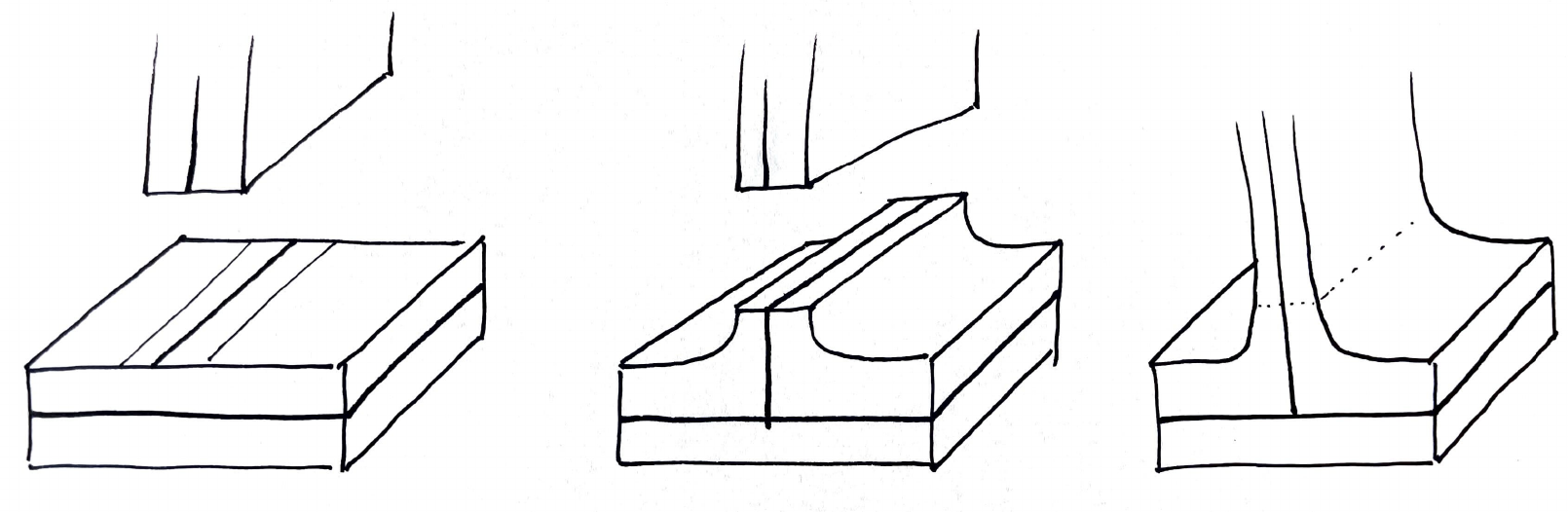}
\caption{Vertical gluing.}
\label{vertical-gluing}
\end{figure}

\begin{remark}\label{rem:uniq gluing}  The W-block $\sX$ obtained by the vertical gluing $\sX_2 \to \sX_1$ is independent of  the choice of a representative of the W-hypersurface $\sN\hookrightarrow \p_\infty \sX_1$ up to strong deformation equivalence of W-blocks, and we  usually do not specify this choice. It is also independent of all other choices involved in the construction up to homotopically canonical strong deformation equivalence, see Lemma \ref{lem:uniq conv}. 
\end{remark}

 Recall that  $\Skel((\sX_1)_{+\wt\sN})$ is  the union of the $\Skel(\sX_1)$ and the Liouville cone $C(\Skel(\wt \sN))$ over the skeleton of
$ \wt\sN$.  Hence,    $$\Skel (\sX_2\to\sX_1)=\Skel(\sX_1)_{+\wt\sN}\cup\Skel(\sX_2)=\Skel(\sX_1)\cup C(\Skel(\wt \sN))\cup\Skel(\sX_2).$$

Note also that the facets in the boundary $\p \sX$ of the vertical gluing $\sX=(\sX_2 \to \sX_1)$ 
are of three kinds:
\begin{enumerate}
\item A facet of $\sX_2$ disjoint from $\sP$.
\item A facet of $\sX_1$ disjoint from $j(\sN)$.
\item  A facet whose nucleus   is  a result of vertical gluing of the nucleus of a  facet of  $\sX_2$ to a facet of $\sX_1$; the gluing is done using an identification of the nucleus $\sM$ of a facet of $\sN$ with its image $\wt \sM = j(\sM)$ viewed as a W-hypersuface in the nucleus of the corresponding facet of $\sX_1$.
\end{enumerate}

  A similar description for the $k$-faces of $\sX=(\sX_2 \to \sX_1)$ may also be formulated.

 \begin{example}[Attaching along a smooth Legendrian]\label{lm:Leg-attach-ribbon}
 Let us focus on the case $ (\sX_1,\lambda_1)= (\sX,\lambda)$, $(\sX_2,\lambda_2)= (\sT^*M, pdq)$ where the second W-block is a cotangent block associated to  $M$ a smooth compact manifold with corners.

Any 1-face  $\sP$ of $\sT^*M$ is of the form  $\sT^*M|_Q$, where $Q\subset M$ is a 1-face, and its nucleus is  the cotangent block
 $(\sN, \lambda_\sN) = (\sT^*Q, pdq)$.

 Let $Q\hookrightarrow \partial_\infty \sX$ be a Legendrian embedding with representative  $Q\hookrightarrow \partial_v W$. According to Example \ref{ex:ribbon}, the embedding extends to a ribbon  $j_\infty: \sA \hookrightarrow \partial_\infty \sX$
 with  representative  $j: \sA \hookrightarrow \partial_v W$, where $ \sA $ is either the proper cotangent block $\sT^* Q$, when the embedding is pure, or more generally, the cotangent block  $(\sT^* Q)^{\bQ, \bigcirc}$, where $\bQ = \{ Q_1, \ldots, Q_k\}$   
   are the $1$-faces that fail the purity test of Definition~\ref{def:adLeg}.

 In both cases, in a  
 slight abuse of terminology, we will  call  $\sX \cup_J  \sT^*M$,   or more generally, 
 $\sX \cup_J  (\sT^*M)^{\bP,\bigcirc}$,  the vertical gluing of  the cotangent block $\sT^*M$ along the Legendrian embedding $Q\hookrightarrow \partial_\infty \sX$.
 
 When $\sX_1$ is a W-block without boundary (i.e. the germ of a Weinstein manifold without boundary) and $M=D^n$ with $Q=\partial D^n$, this recovers the notion of attaching a critical handle to a Weinstein manifold. Subcritical attachments correspond to $M=D^k \times D^{n-k}$ and $Q=D^k \times S^{n-k-1}$ for $k \geq1$. 
 \end{example}

 \section{W-buildings}\label{sec:W-bldg}
 
  \subsection{Definition of a W-building}\label{sec:def-W-bldg}
  
 
 
\begin{definition} A {\em $k$-story W-building}, denoted
  $$(\sX_k,\lambda_k)\to(\sX_{k-1},\lambda_{k-1})\to\dots \to(\sX_1,\lambda_1)$$
consists of the data of iterated vertical attaching from top to bottom of W-blocks:
 \begin{align*}
&\sX_{\geq k} =\sX_k, \, \, \sX_{\geq k-1}=(\sX_k\to\sX_{k-1}), \, \,
 \sX_{\geq k-2}=(\sX_{\geq k-1}\to\sX_{k-2}), \, \, \dots  \, \,
\sX_{\geq 1}=(\sX_{\geq 2}\to\sX_1).
 \end{align*}
 The structure of a $k$-story W-building on a W-block $(\sX,\lambda)$ consists of a $k$-story W-building $\sX_k  \to \cdots  \to \sX_1$ together with an isomorphism of W-blocks $\sX \simeq \sX_{\geq 1}$ which will often be implicit.
 \end{definition}
 
 One distinguished family of W-buildings is of particular importance:
 
 \begin{definition}
 A {\em proper cotangent building} $$\sT^*M_k \to \sT^*M_{k-1} \to \cdots \to \sT^*M_1$$ is a W-building in which the blocks are proper cotangent blocks $(\sT^*M_i, pdq)$ of compact manifolds with corners $M_i$ (see Example~\ref{ex:co-bl}). 
 
 A {\em cotangent building} 
 $$(\sT^*M_k)^{\bP_k, \bigcirc} \to (\sT^*M_{k-1})^{\bP_{k-1}, \bigcirc} \to \cdots \to (\sT^*M_1)^{\bP_{1}, \bigcirc}$$
  is a W-building in which the blocks are  cotangent blocks $((\sT^*M_i)^{\bP_i,\bigcirc},(pdq)^{\bP_i,\bigcirc})$, i.e. proper cotangent blocks $(\sT^*M_i, pdq)$ with some  faces  $\bP_i$ stripped (see Example~\ref{ex:gen-co-bl}). 
 \end{definition}
 
 \begin{example}
 Let $(W,\lambda)$ be the Weinstein domain obtained from a standard Lagrangian handle attachment to the standard Darboux ball $B$ along a Legendrian sphere $\Lambda \subset \p B$. We may choose a W-block deformation equivalence $$B \simeq (\sT^*D^n)^{\partial D^n, \bigcirc}.$$ Then the corresponding W-block $(\sX,\lambda)$ admits the structure of a 2-story  cotangent building $$\sT^*D^n \to (\sT^*D^n)^{\partial D^n, \bigcirc}$$
 which in fact arises by stripping the face of a 2-story cotangent building 
 $$ \sT^*D^n \to \sT^*D^n.$$
 If instead $(W,\lambda)$ is given by a subcritical attachment then up to deformation equivalence it also has the structure of a 2-story  cotangent building 
 $$\sT^*(D^{n-k} \times D^k)^{ D^{n-k} \times \partial D^k, \bigcirc} \to (\sT^*D^n)^{\partial D^n, \bigcirc}.$$
for some $k<0<n$, where the attachment is along a Legendrian $S^{n-k-1} \times D^k$.
 \end{example}
 
           \begin{figure}[h]
\includegraphics[scale=0.3]{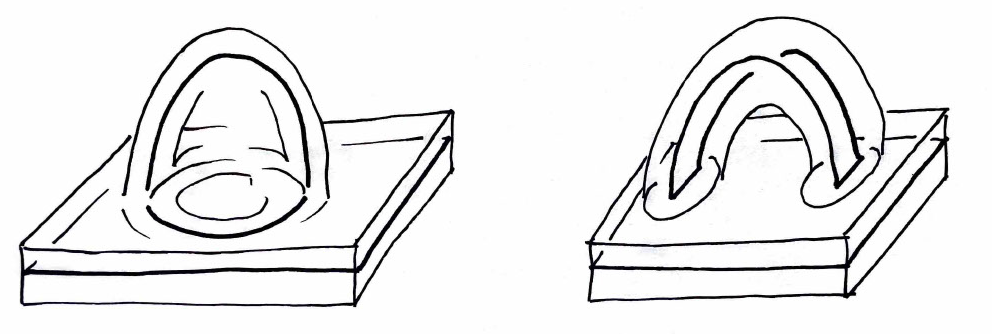}
\caption{A proper cotangent building and a non-proper cotangent building}
\label{proper-nonproper}
\end{figure}

 \subsection{Structure of a W-building}
 
 The structure of a W-building consists precisely of the collection of vertical attachments $\sX_{\geq j} \to \sX_{j-1}$, which inductively determine the W-blocks $\sX_{\geq j-1}$. This amounts to the following data. First note that $\sX_k$ is glued to $\sX_{k-1}$ along a facet $\sP_{k,k-1}$ of $\sX_k$ with nucleus $\sN_{\geq k-1}=\sN_{k,k-1}$ which is identified with a W-hypersurface $\wt\sN_{\geq k-1}=\wt\sN_{k,k-1}\subset\p_\infty \sX_{k-1}$. In turn $\sX_{\geq k-1}=(\sX_k\to\sX_{k-1})$ is glued to $\sX_{k-2}$ along its facet $\sP_{\geq k-1}$ with its nucleus $\sN_{\geq k-1}$ which itself is the result of vertical gluing $\sN_{k,k-2}\to \sN_{k-1,k-2}$  of nuclei (where either of these could be empty) of facets $\sP_{k, k-2}$ of $\sX_k$ and $\sP_{k-1,k-2}$ of $\sX_{k-1}$. The nucleus $\sN_{\geq k-1}$ is identified with a W-hypersurface
$\wt\sN_{\geq k-1}\subset\p_\infty\sX_{k-2}.$
 Continuing by induction, we see that $\sX_{\geq j}, j=k, k-1, \dots, 1$, is the result of vertical gluing  of $\sX_{\geq j+1}$ to $\sX_j$ along a facet $\sP_{\geq j+1}$ of  $\sX_{\geq j+1}$ with  the nucleus $\sN_{\geq j+1}$ of $ \sP_{\geq j+1}$ which has a  $(k-j)$-story building structure $$\sN_{k,j}\to\sN_{k-1,j}\to\dots\to \sN_{j+1,j}$$ of nuclei  $\sN_{i,j}$ of faces $\sP_{i,j}$ of blocks $\sX_i$, $i=k,\dots, j+1$. The nucleus $\sN_{\geq j+1,j}$ is identified with a W-hypersurface $\wt\sN_{j}\subset\p_\infty\sX_j.$
 
Note that the list of facets $\bP_i=[\sP_{i, i-1},\dots, \sP_{1,1}]$  of $\sX_i$ are admissible in the sense of Section \ref{sec:smooth-corners}.  We allow  facets $\sP_{i,j}$  to  be disconnected, but in that case the closed connected components are disjoint.  Facets $\sP_{i,j}$ and $\sP_{i,j'}$ for $j\neq j'$ could be either disjoint, or otherwise  adjacent to the same 2-face of $\sX_i$.

For each $\ell\leq j$ the W-block $\sX_{\geq j+1}$ has   also a facet $\sP_{\geq j+1,\ell}$ with nucleus  $\sN_{\geq j+1,\ell}$  which itself has a $(k-j)$-story  building structure resulted from successive  vertical gluing  $$(\dots(\sN_{k,\ell}\to \sN_{k-1,\ell})\to\dots)\to \sN_{j+1,\ell}.$$

Setting $\sN_\geq k=\varnothing$  we get     a  collection of    Liouville embeddings
 $$\Phi_j:=\Big((\sX_j)_{+\wt\sN_{  j}},(\lambda_j)_{+\wt \sN_{ j}} \Big) \to (\sX,\lambda)$$
 such that  $$\bigcup\limits_{j=1}^k\Phi_j\Big((\sX_j)_{+\wt\sN_j}\Big)=\sX.$$  We call the collection $\{\Phi_j\}_{j=1,\dots,k}$ a {\em W-atlas} of the building. Note that  $(\sX_j)_{+\wt\sN_j}\supset \sX_j$, and therefore, we also have  embeddings
 $\Phi_j|_{\sX_j}:(\sX_j,\lambda_j)\to (\sX,\lambda)$ (which do not cover $\sX$).

 Alternatively, the data of a W-building consists of the list of faces $\bP_j=[\sP_{j,j-1},\dots, \sP_{j,1}]$ of $\sX_j$ and their nuclei $(\sN_{j,j-1},\lambda_{j,j-1}:=\lambda|_{\sN_{j,j-1}}),\dots, (\sN_{j,1},\lambda_{j,1})$, W-blocks $(\sN_{\geq j},\lambda_{\geq j})$, presented as buildings $(\sN_{k,j-1},\lambda_{k,j-1})\to \dots\to(\sN_{j,j-1},\lambda_{j,j-1})$, and the realization of $\sN_{\geq j}$ as W-hypersurfaces in $\p_\infty \sX_{j-1}$. By Remark \ref{rem:uniq gluing} this sequence of W-blocks and W-hypersurfaces uniquely determines the resulting W-building up to homotopically canonical strong deformation equivalence.

          \begin{figure}[h]
\includegraphics[scale=0.3]{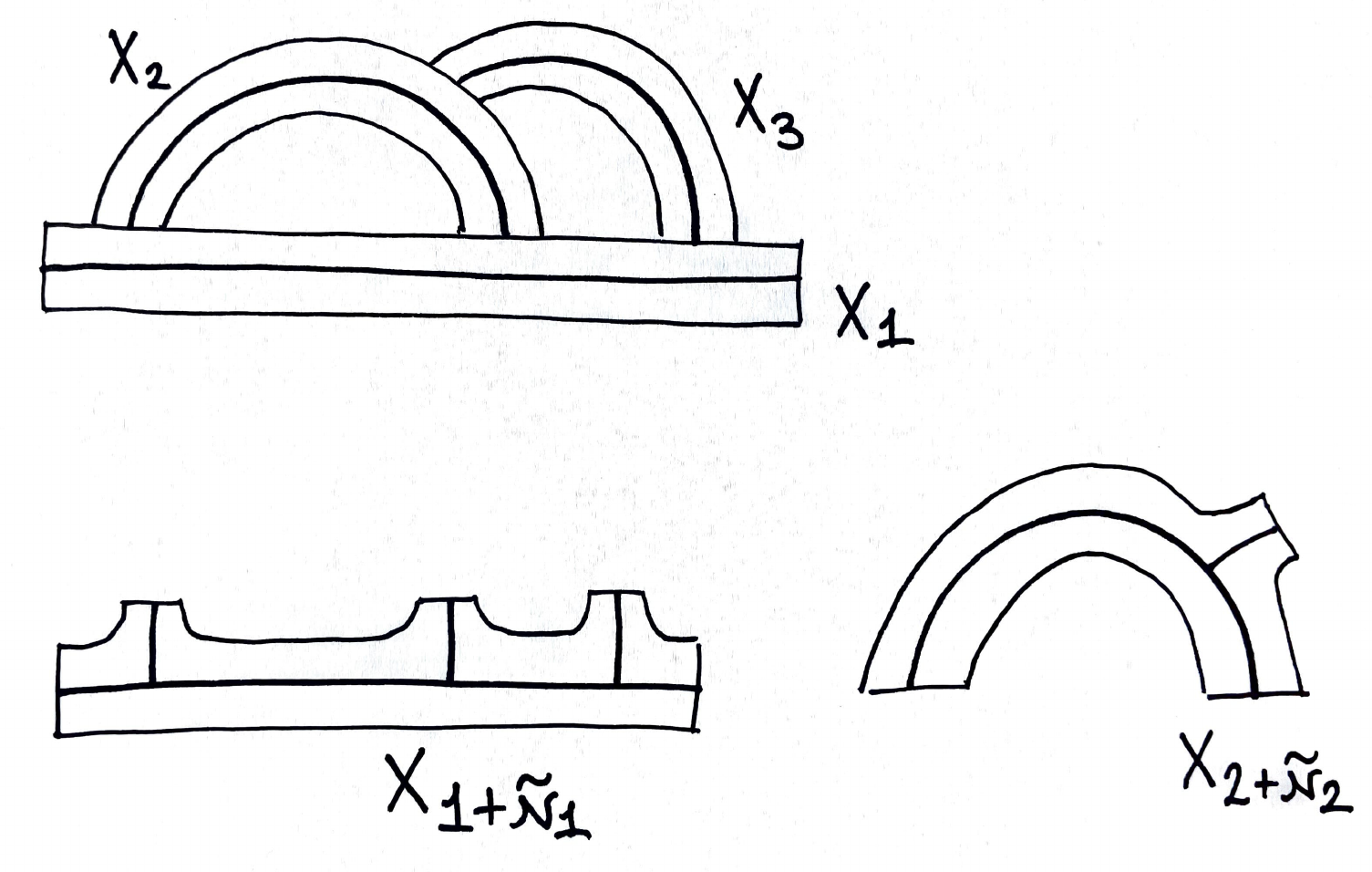}
\caption{An atlas for a W-building}
\label{atlas}
\end{figure}

\begin{remark}\label{lm:estimate-Liouville-field}  Suppose  a W-block $(\sX,\lambda)$ is presented as a $k$-story cotangent building $(\sT^*L_k\lambda_k)\to\dots\to (\sT^*_1,\lambda_1).$ Let $Z$ be the Liouville vector field determined by $\lambda$. Choose a background Riemannian metric on $\sX$. Let $||Z|| $ be the infimum of  the $C^0$-norm of   the vector field $Z $ over all defining domains of the corresponding W-block. Then   
  $$||Z||\leq   Ck\eps,$$ where $\eps>0$ is a constant entering the definition of vertical gluing, and  $C$ is a universal constant which depends only on the background metric.  In particular, there exists a strong homotopy  $(\sX,\lambda^\u)$, $u\in(0,\eps]$   of W-blocks  preserving the background  symplectic structure, together with the  cotangent  building structures, such that $\lambda^\eps=\lambda$ and $||Z^\eps|||\mathop{\to}\limits_{\eps\to 0} 0$. 
  
  To see this, first note that away from the attaching areas the Liouville vector field of   a cotangent building vanishes on the    skeleton. On the other hand in the attaching areas    the norm of the Liouville field  is bounded by $||C\eps||$, where $\eps$ is the parameter in Example \ref{ex:special-adj} (e) involved in the facet inversion. Up to a strong homotopy one can choose $\eps$ arbitrarily small without changing the background form (see Remark \ref{rem:modif-fixed-sympl}), and hence the claim follows.
 \end{remark}

  \subsection{Bottom-up construction of a W-building} While the building data  $$\sX=(\sX_k\to\dots\to\sX_1)$$ assumes a   ``top-down" construction of a W-block, the ``bottom-up" perspective is the more traditional viewpoint of Weinstein handlebody theory. Indeed, a W-building is given by successive vertical gluing
  $$(( \cdots (\sX_k\to \sX_{k-1}) \to \cdots ) \to \sX_2) \to\sX_1$$
  whereas a handlebody presentation can be viewed as successive vertical gluing $$\sX_k\to(\sX_{k-1}\to(\cdots\to(\sX_2\to\sX_1)\cdots))$$ where all $\sX_j$ are (generalized) cotangent blocks.  As we will see below, the building structure  has strictly more information. In particular, while the top-down construction of a building can always be  converted to a bottom-up construction, the opposite conversion requires additional constraints  on vertical attachments, as we now discuss.
 We begin with  the following lemma. 
    \begin{lemma}\label{lm:remains} Let $\sX $ be a W-block, and $\sA\subset\p_\infty\sX$ a W-hypersurface which has    a structure of a 2-story building $\sA_2\to\sA_1$, so that $\sA_1$ is also a W-hypersurface in $\p_\infty\sX$. Then  the block $\sX_{+\sA_1}\supset\sX_1$ contains a W-hypersurface $\wt \sA_2$  such that $C(\Skel(\wt\sA_2); Z_{+\sA_1})= C(\Skel (\sA)\setminus\Skel(\sA_1), Z)$.     \end{lemma}
 \begin{proof} We may take $\wt \sA_2$ to be the image of $\sA_2$ via the embedding $\sA_2 \subset (\sA_2 \to \sA_1)= \sA \subset \p_\infty \sX$ which is disjoint from $\sA_1 \subset \p_\infty \sX$ and therefore may be regarded as a W-hypersurface in $\sX_{+\sA_1}$.
 \end{proof}
     We will refer to the W-hypersurface $\wt\sA_2$ as the {\em remnant} of $\sA$ after converting $\sA_1$ into a face nucleus.

  \begin{prop}\label{prop:bldg-bottom-up}
 Consider a  W-block $(\sX,\lambda)$ presented as a $k$-story  W-building $$\sX_k\to\cdots\ \to \sX_1.$$   Let  $\sX_i^{\frown,\bP_i}$ be the W-block obtained from $\sX_i$ by smoothing corners in the list $\bP_i=[\sP_{i,i-1},\dots,\sP_{i,1}]$, let  $\sP_{\bP_i}^{\frown}$ be the new facet  resulted from the smoothing and  let  $\sN_{\bP_i}^{\frown}$ be  its nucleus.   
 Then for 
  each $i \leq k$  there exist
  \begin{itemize}
  \item  W-blocks $\sX^{\leq 1}=\sX_1 \subset\sX^{\leq 2}\subset \dots\subset  \sX^{\leq k-1}\subset \sX^{\leq k}=\sX$;
  \item for each $i=2,\dots, k$  embeddings   $\phi_i: \sN_{\bP_i}^{\frown}\to \wt\sN_{\bP_i}^{\frown}\hookrightarrow \p_\infty\sX^{\leq i-1}$ as Weinstein hypersurfaces.
 \end{itemize}  such that for each $i=1,\dots, k$ the block $\sX^{\leq i}$  has a building structure $\sX_i\to\dots\to\sX_1$, and for
  each $i=2,\dots, k$ the block $\sX^{\leq i}$ is the result of vertical gluing of $\sX_i^{\frown,\bP_i}$ to $\sX^{\leq i-1}$  using 
   $\phi_i: \sN_{\bP_i}^{\frown}\to \wt\sN_{\bP_i}^{\frown}\hookrightarrow \p_\infty\sX^{\leq i-1}$.
    \end{prop}
    
        \begin{figure}[h]
\includegraphics[scale=0.3]{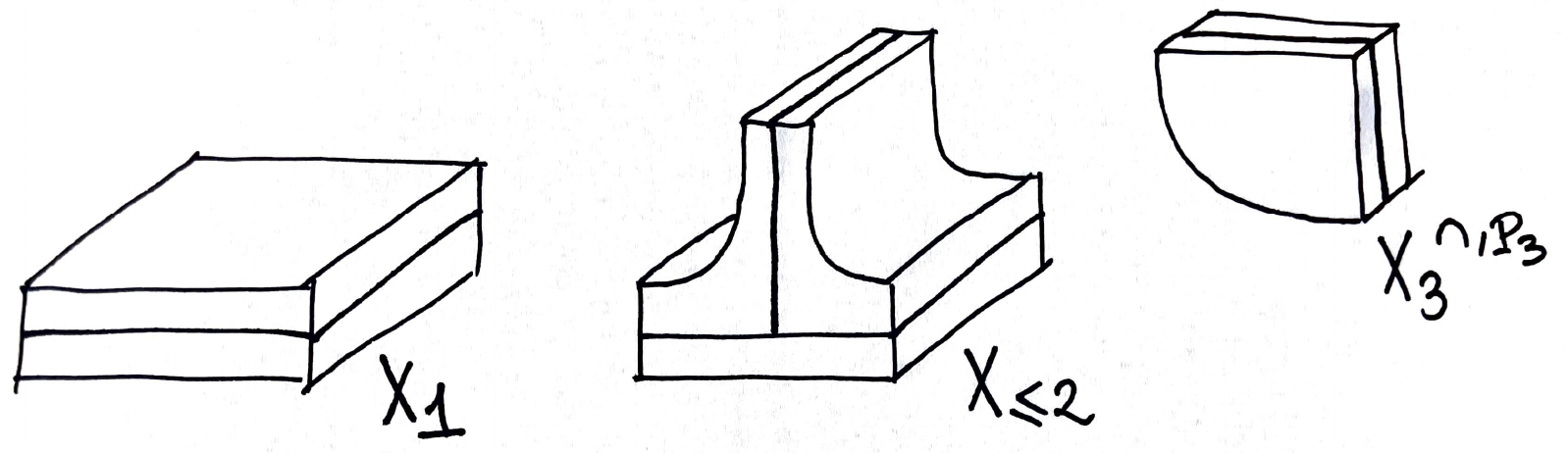}
\caption{Bottom-up perspective on a three-block building}
\label{down-up}
\end{figure}

       \begin{proof}[Proof of Proposition \ref{prop:bldg-bottom-up}] Consider first the case of a  3-story building, i.e. let $$\sX= (\sX_3\to\sX_2\to\sX_1).$$ Then $\sX$ is obtained by    vertically gluing  $\sX_{\geq2}$ to $\sX_1$ along its facet $\sP_{\geq 2}$ with an identification   $\phi_{\geq 2}:\sN_{\geq 2}\to \wt\sN_{\geq 2}$  for  a W-hypersurface $\wt\sN_{\geq 2}\subset\p_\infty\sX_1$. In turn, $\sX_{\geq 2}$ is the result of vertical guing of $\sX_3$ to $\sX_2$  via $\phi_3:\sN_{3,2}\to\wt\sN_{3,2}\hookrightarrow\p_\infty\sX_2$.
By construction $\sN_{\geq 2}$ has a structure of a 2-story building $\sN_{3,1}\to \sN_{2,1}$, so $\sN_{2,1}\subset \sN_{\geq2}$.  Denote $\wt \sN_{2,1}:=\phi_{\geq 2}(\sN_{2,1}) \subset \p_\infty \sX_1$. 

Let $\sX^{\leq2}$ be the 2-story building constructed by vertical gluing of $\sX_2$ to $\sX_1$ using $$\phi_{\geq2}|_{\sN_{2,1}}:\sN_{2,1}\to\wt\sN_{2,1}\hookrightarrow\p_\infty\sX_1.$$ Equivalently, 
\begin{equation}\label{eq: gluing}
\sX^{\leq2}=(\sX_1)_{+\wt\sN_{2,1}}\mathop{\cup}\limits_{\sP_{2,1}}\sX_2 \end{equation}
where $(\sX_1)_{+\wt\sN_{2,1}}$ is the result of conversion of the W-hypersurface  $\wt \sN_{2,1}$ into a nucleus of a facet. 

Let $\wt\sN_{3,1}\subset \p_\infty(\sX_1)_{+\wt\sN_{2,1}}$ be the remnant of $\wt\sN_{\geq 2}$ after the conversion.  The horizontal gluing \label{eq: gluing} also horizontally glues the W-hypersurfaces $\wt\sN_{3,1}\subset \p_\infty(\sX_1)_{+\wt\sN_{2,1}}$ and $\wt\sN_{3,2}\subset \p_\infty\sX_2$ into a W-hypersurface $\wt\sN$ which is  Liouville isomorphic to   $\sN^\frown_{\bP_3}$, the nucleus of the facet
  $\sP^\frown_{\bP_3}$ of $\sX_3^{\frown,\bP_3}$. Then the W-block resulted from the vertical gluing of $\sX_3^{\frown,\bP_3}$ to $\sX^{\leq 2}$ via the identification $\sN^\frown_{\bP_3}\to \wt\sN$ is strongly deformation equivalent to $\sX$.
  
 We continue by induction on 
 the number  $k$ of stories.
    Let us assume that the proposition  is proven for W-blocks of    height  $<k$.
    In particular, applying the induction hypothesis to bulldings $\sX_{\geq 2}=(\sX_k\to\dots \to \sX_2)$  we can  define   W-blocks
    $\sX^{\leq i}_{\geq2 }$ with the building structure $(\sX_i\to\cdots\to \sX_2)$, such that for $i>2$ the W-block  $\sX_{\leq i}^{\geq 2}$ can be constructed by vertical gluing of $\sX_i^{\frown,\bP_{i,\geq 2}}$ to   $\sX^{\leq i-1}_{\geq 2}$ using a Liouville embedding $\phi_{i,2}:\sN^{\frown}_{\bP_{i,\geq 2}}\to\p_\infty\sX^{\leq i-1}_{\geq 2}$. 
    
    Here $\bP_{i,\geq 2}:=[\sP_{i,i-1},\dots,\sP_{i,2}]$ is the sublist of the list $\bP_i$ of facets of $\sX_i$, 
    $\sX_i^{\frown,\bP_{i\geq 2}}$ is the result of smoothing corners of $\sX_i$ in the list $\bP_{i,\geq 2}$, and $\sP^\frown_{\bP_{i,\geq 2}}\subset \sX_i^{\frown,\bP_{i\geq 2}}$ is the smoothed facet with $\sN^\frown_{\bP_{i,\geq 2}}\subset \sP^\frown_{\bP_{i,\geq 2}}$ its nucleus. 
 
   By definition, $\sX$ is the result of vertical gluing of $\sX_{\geq 2}$ to $\sX_1$ along its  facet $\sP_{\geq 2,1}$ by identifying its nucleus $\sN_{\geq 2,1}$ with a W-hypersurface $\wt\sN_{\geq 2,1}$ of $\sX_1$. In turn, $\sN_{\geq 2,1}$ itself had a $(k-1)$-story building structure $\sN_{k,1}\to\dots\sN_{k,2}$, and hence, by the induction hypothesis there are already defined W-buildings   $\sN_{\geq 2,1}^{\leq i}\subset \sN_{\geq 2,1}$, $i=2,\dots, k$, which have W-building presentation $\sN_{i,1}\to\dots\to\sN_{2,1}, \; i=2,\dots, k$. Similarly, we have  W-hypersurfaces   $\wt \sN_{\geq 2,1}^{\leq i}\subset \wt \sN_{\geq 2,1}\subset \p_\infty\sX_1$. We define $\sX^{\leq i}$ by vertically gluing  $\sX^{\leq i}_{\geq 2}$ to $\sX_1$ by identifying the nucleus   $\sN_{\geq 2,1}^{\leq i}$ with the W-hypersurface   $\wt \sN_{\geq 2,1}^{\leq i}\subset\p_\infty\sX_1$. By induction,  $\sX^{\leq i}_{\geq 2}$ can be constructed by  vertically gluing $(\sX_i)^{\frown,\bP_{i,\geq 2}}$ to $\sX^{\leq i-1}_{\geq 2}$. Hence,  the W-block  $\sX^{\leq i}$ has a 3-story building structure $\sX_i^{\frown,\bP_{i,\geq 2}}\to\sX^{\leq i-1}_{\geq 2}\to\sX_1$. Let  $\bP:=[\sP^{\frown}_{\bP_{i,\geq 2}},\sP_{i,i}]$ be  the list  of 2-facets  of $\sX_i^{\frown,\bP_{i,\geq 2}}$.  We recall that $\sP^{\frown}_{\bP_{i,\geq 2}}$ is the facet resulted from smoothing the corners in the list $[\sP_{i,i-1},\dots,\sP_{i,2}]$, and thus $(\sX_i^{\frown,\bP_{i,\geq 2}})^{\frown,\bP}= \sX_i^{\frown,\bP_{i}}$. Hence, by the base of the induction case of 3-blocks we conclude that    $\sX^{\leq i}$  can be constructed by  vertically   attaching to $\sX^{\leq i-1}=(\sX^{\leq i-1}_{\geq2}\to\sX_1)$ 
the W-block  $\sX_i^{\frown,\bP_{i}}$ using the identification of the nucleus $\sN^\frown_{\bP_i}$ with a W-hypersurface $\wt\sN^\frown_{\bP_i}$
which is the result of the horizontal gluing of W-hypersurfaces $\wt \sN^\frown_{\bP_{i,\geq 2}}\subset \p_\infty\sX^{\leq i-1}_{\geq2}$ and
$\wt\sN_{i,1}\subset \p_\infty(\sX_1)_{+\wt\sN_{i-1,\geq 2}}$, which is the remnant of $\wt\sN_{i,\geq 2}\subset  \p_\infty\sX_1$ after converting
$\wt\sN_{i-1,\geq 2}$ into  a facet  nucleus.
         \end{proof}
  
   \subsection{W-hypersurfaces and Legendrians compatible with a building structure}
  For a 2-story building the top-down and bottom-up constructions are tautologically the same. However, already for the 3-story case any building, which by definition is given by successive top-down gluing, and hence according to Proposition  \ref{prop:bldg-bottom-up} can also  be constructed  by successive bottom-up gluing, the converse is  not  true unless one make some  additional assumptions on the W-hypersurfaces along which the blocks are vertically attached.
 \begin{definition}  
  Suppose a W-block $(\sX,\lambda)$ is presented as a $k$-story  W-building $$\sX=(\sX_k\to\cdots\to \sX_1)$$ and $\Phi_j:(\sX_j)_{+\sN_{\geq j+1,j}}\to\sX$, $j=1,\dots, k$ is the corresponding W-atlas.
 \begin{itemize}
 \item[1.] A W-hypersurface $\sA\subset\p_\infty\sX$  is called  {\em compatible} with the building structure, if $\Phi_j^{-1}(\sA)$ is a W-hypersurface in $(\sX_j)_{+\sN_j}$ for each $j=1,\ldots, k$.
 \item[2.] A Legendrian $\Lambda\subset\p_\infty\sX$ is called  {\em compatible} with the building structure, if $\Phi_j^{-1}(\Lambda)$ is a Legendrian in $(\sX_j)_{+\sN_{\geq j+1,j}}$ for each $j=1,\ldots, k$.
 \end{itemize}
 \end{definition}
 The following two lemmas are straightforward:
 \begin{lemma}\label{lm:comp-Leg-W}
 A compatible Legendrian admits a compatible ribbon. 
 \end{lemma}
 \begin{lemma} \label{lm:bldg-after} Suppose a  $2n$-dimensional W-block $\sX $ is presented as a $k$-story  W-building $\sX_k\to\cdots\to \sX_1$ and let $\Phi_j:(\sX_j)_{+\sN_{\geq j+1,j}}\to\sX$, $j=1,\dots, k$ be the corresponding W-atlas.   Let $\sY$ be a W-block,   $\sP$ a facet of $\sY$ and   $\sM$    the  nucleus of $\sP$. Let
 $\phi:\sM\to\sA\hookrightarrow \p_\infty\sX$ be an embedding of $\sM$ as a  W-hypersurface $\sA$ of $\sX$ compatible  with the building structure of $\sX$. Denote $\sM_j:=\phi^{-1}\Phi_j(\p_\infty(\sX_j)_{+\sN_{\geq j+1,j}}).$  Then   
 \begin{itemize}
 \item $\sM=(\dots((\sM_k\cup\sM_{k-1})\cup\sM_{k-2})\cup\dots)\cup\sM_1$ is an admissible decomposition of $\sM$, and   $\sH_{k-1}:=\sM_k\cap\sM_{k-1}, \sH_{k-2}:=(\sM_k\cup\sM_{k-1})\cap \sM_{k-2},\dots ,\sH_1=\bigcup\limits_{j=2}^k\sM_k\cap \sM_1$ is the corresponding dividing collection;
  \item Let $\wh \sX$ be  the result of the vertically attaching $\sY$ to $\sX$ via the identification $\phi:\sM\to\sA\hookrightarrow \p^\infty\sX$. Then
  $\wh\sX $  has a building structure  $\sY^{\wedge,\bH}\to \sX_k\to\dots\to \sX_1$ where $\bH=\{\sH_1,\ldots,\sH_k\}$.
\end{itemize} 
  \end{lemma}  
  Thus, if $\sX$ and $\sY$  are  $2n$-dimensional   W-blocks, $\sX$  has  a W-building structure $\sX=(\sX_k\to\dots\sX_1)$,   $\sN$ is a nucleus of a face $\sP$, then any embedding $\sN \to\p_\infty\sX$ as a W-hypersurface that is  compatible with the building structure of $\sX$ induces  an admissible decomposition $\sH$ of $\sN$, and hence, of $\sP$.  Creating additional corners of $\sY$ according to the decomposition and vertically attaching $\sY^{\wedge,\bH}\to\sX$  results in a $(k+1)$-story building $\sY\to \sX_k\to\dots\to\sX_1$.

 Conversely, suppose a W-block $(\sX,\lambda)$ is presented as a $k$-story  building $\sX_k\to\dots, \sX_1$.   Consider its bottom-up construction, as in Proposition \ref{prop:bldg-bottom-up}, i.e. take a sequence of buildings $\sX^{\leq i}=\sX_i\to\dots\sX_1$, $i=1,\dots, k$. According to
  Proposition \ref{prop:bldg-bottom-up} $\sX^{\leq i}$ is the result of vertical gluing of $\sX_{i}^{\frown,\bP_{i}}$ to $\sX^{\leq i-1}$ by identifying
  the nucleus $\sN_{\bP_i}^\frown$  of the smoothed facet   with the corresponding W-hypersurface    $\sN_{\bP_i}^\frown\subset\p_\infty\sX^{\leq i-1}$. Then by construction the hypersurface  $\sN_{\bP_i}^\frown$ is compatible with the building structure $\sX^{\leq i-1}=(\sX_{i-1}\to\dots\to\sX_1).$
 In other words, we have:
 \begin{lemma}\label{lm:up-vs-doen} In order that  the result of  successive ``bottom-up" vertical attachments $$\sX_k\to(\sX_{k-1}\to(\cdots\to (\sX_2 \to \sX_1)\cdots)$$ has a structure of a $k$-story building $\sX_k\to\dots\to\sX_1$  it is necessary and sufficient that the attaching hypersurfaces $\wt\sN_i\subset\p_\infty\sX^{\leq i-1}$, $i=1,\dots, k-1$, are compatible with the building structures $\sX^{\leq i}=(\sX_i\to\cdots \to\sX_1).$
 \end{lemma}
 
  In particular, suppose  $M$ is a compact $n$-dimensional manifold with boundary, and $f:\p M\to \sX\setminus\Skel(\sX)$ is a Legendrian embedding compatible  with the building structure. Then vertically attaching the (generalized) cotangent block $\sT^*M$ to $\sX$ along the ribbon of the Legendrian we add one story to the building $\sX$. 
  
  In the next section we will show that any Legendrian embedding in a W-building  is Legendrian isotopic to a Legendrian embedding compatible with the building structure. As any W-block has a  Weinstein handlebody presentation we conclude that   every W-block structure is homotopic to  the structure of a (generalized) cotangent building.

\subsection{Making a  Legendrian  compatible with a cotangent building structure}
The goal of this section is the following proposition.   \begin{prop}\label{prop:Leg-adjustment}
Suppose a W-block $(\sX,\lambda)$ is presented as   a proper cotangent building
  $\sX  = (\sX_k \to\dots\to\sX_1)$ with $\sX_j=\sT^*M_j$.  Let $\Lambda\subset \p_\infty \sX$ be  a Legendrian.  Then there exists a     Legendrian  isotopy $\Lambda_u$, $u\in[0,1]$,   starting from $\Lambda_0=\Lambda$ such that   $\Lambda_1$ is compatible with  the  building structure.   \end{prop}
  
      \begin{figure}[h]
\includegraphics[scale=0.3]{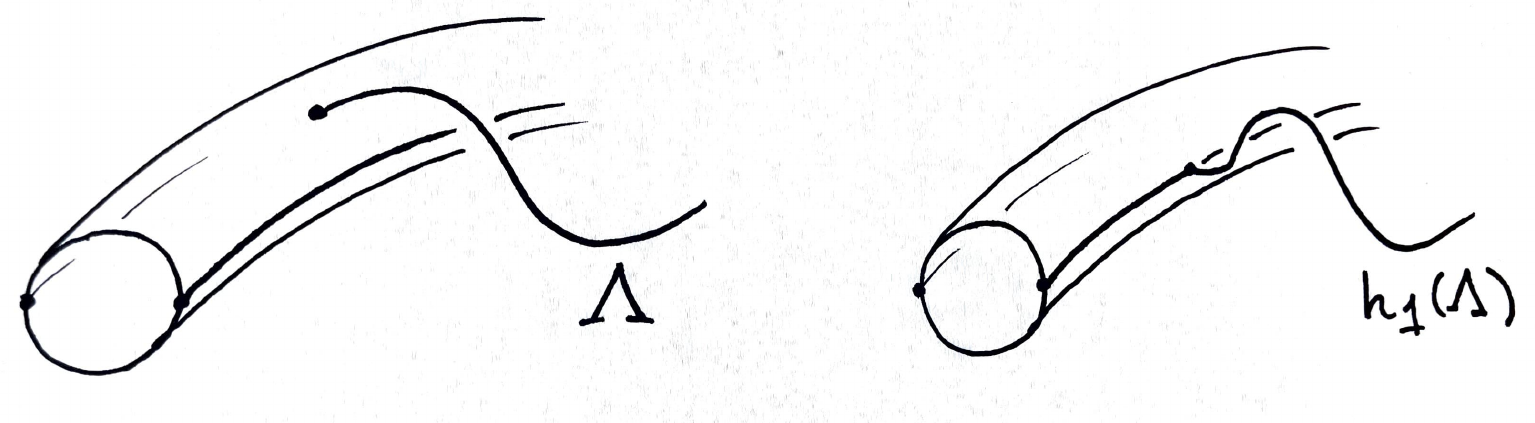}
\caption{Adjusting a Legendrian.}
\label{adjusting-leg}
\end{figure}
 
  \begin{proof} 
 We   prove the proposition by induction in $k$, beginning from $k=2$ (the case $k=1$ is a tautology).
  Suppose  $(\sX,\lambda)$ is presented as  a 2-story building $\sX_{ 2}\to \sX_1$, where $\sX_2$ is vertically glued to $\sX_1$ along a facet $\sP$ whose nucleus
  $\sN$ is realized as a W-hypersurface $\wt\sN\subset\p_\infty\sX_1$.   Since $\sX_2 = \sT^*M_2$ is a cotangent block, the nucleus $(\sN,\lambda_\sN)$ is a stratified subset of dimension $\leq n-1$. By definition $\sX$ is result of horizontal gluing $(\sX_1)_{+\wt\sN}\mathop{\cup}\limits_\sP\sX_2$, and  hence, $\sP$ is realized as a dividing hypersurface in $\sX$.  The Liouville form $\lambda|_{\Op\sP}$ can be written as $\lambda_{\sN}+sdt$ where $P=\{s=0\}$. We denote by $\pi$ the projection $\Op P\to\sN$. Consider the vector field $Y=Z_{\sN}-s\frac\p{\p s}+2t\frac \p{\p t}$ on $\Op \sP$.   The flow of $Y$ is given by $Y^u(x,s,t)=(Z_{\sN}^u(x), e^{-u}s, e^{2u}t)$, and hence
 $(Y^u)^*\lambda=e^u \lambda $.  Let $W$  be a defining domain for $\sX$. The flow $Y^u$ induces a (time-independent) contact isotopy on $\p W\cap\Op P$, and  hence, it may be cut-off for $|t|>\sigma$ for a small $\sigma>0$ to get a contact isotopy $h_u:\p W\to\p  W$ supported in $\p W\cap\{|t|\leq\sigma\}$. The isotopy $h_u$ is fixed on $Q:=\sY\cap\p W$ and leaves    $ \p W\cap\{t=0\}$ invariant.  We further reparametrize the isotopy to the interval $[0,1)$ by setting $\wh h_u=h_{\tan\frac{2u}\pi}$.
 For any  neighborhood $U\supset Q$, $U\subset \p W$, and any compact subset $C\subset  (\p W\cap\{|t|<\sigma\})\setminus\pi^{-1}(\Skel(\sN))$ there exists $u_0(C,U)\in(0,1)$ such that for $u\in[u_0,1)$ we have $\wh h_u(C)\subset U$.   On the other hand, for  a sufficiently small neighborhood $U\supset Q$ in $\p W$ we can use $s,t\in[-\sigma,\sigma]$ and $y\in Q$ as coordinates, and the isotopy $\wh h_u$ in these coordinates is given by $\wh h_u(y,s,t)=(y,e^{-\tan u}s,t).$  Thus, we can extend  $h_u|_C$ to a $C^\infty$-isotopy on   $u\in[0,1]$ by setting  $\wh h_1(y, s,t)=(y,0,t).$
  
   Let $\Lambda\subset \p W$ be  any Legendrian. We can assume that it intersects  transversely $P\cap \p W$ along an $(n-2)$-dimensional submanifold $M$. Generically, by dimensional reasons the image  $\pi(M)\subset\sN$ does not intersect $\Skel(\sN)$.
  For a sufficiently small $\sigma>0$ we can parametrize  $\Lambda$ by a Legendrian embedding $\Psi:M\times [\sigma,\sigma]\to \Lambda$ of the form $\Psi(y,t)=(\psi(y,t),t)\in (P\cap \p W)\times(-\eps,\eps)$, $y\in M, t\in(-\sigma,\sigma)$.   The compact set $\Psi(M\times [\sigma,\sigma])$  does not intersect $\pi^{-1}(\Skel(\sN))$ is $\sigma$ is small enough.  Let $p:\sY\setminus\Skel(\sY)\to Q$ be the projection along trajectories of the Liouville field $Z_{\sN}$.  Note that $p|_M:M\to Q$ is a Legendrian embedding. Then  there is defined a 
   Legendrian isotopy $\Psi_u:=\wh h_u\circ\Psi:M\times[-\sigma,\sigma]\to\p W\cap\{|t|\leq \sigma\} $, $u\in[0,1]$, such that    $\Psi_1(x,t)=(p\circ \psi(x,t), s=0, t) \in \Op Q $.   But this by definition means that the Legendrian $\wh\Lambda$ which coincides with $\Psi_1(M\times[-\sigma,\sigma])$ in $\p W\cap (P\times[-\sigma,\sigma])$ and with $\Lambda$ elsewhere in $\p W$ is adapted to the building structure $\sX_2\to\sX_1$.  
     
Suppose now that the proposition is already proven for buildings of height $<k$, $k\geq 3$. Let $\sX=(\sX_k\to\dots\to\sX_1)$ be a $k$-story building and $\Lambda\subset\sX$ a Legendrian.
By definition it is a result of vertical attachment of a $(k-1)$-story building $\sX_{\geq 2}$ to $\sX_1$, i.e. 
$\sX=
\sX_{\geq 2}\cup( \sX_1)_{+\wt \sN_{\geq 2}}$. Applying the $k=2$ case to the 2-story building $\sX=\sX_{\geq 2}\to\sX_1$ we deform $\Lambda$ to an adapted  Legendrian  $\Lambda'$ which intersects $\sX_{\geq 2}$ along a Legendrian $\Lambda'_{\geq 2}\subset\p_\infty\sX_{\geq 2}$ and intersects  $(\sX_1)_{+\wt \sN_{\geq 2}}$ along a Legendrian  $\Lambda'_1\subset\p_\infty( \sX_1)_{+\wt \sN_{\geq 2}}$. Thus to complete the proof it remains to apply the induction hypothesis to the Legendrian  $\Lambda'_{\geq 2}$ in the $(k-1)$-story building $\sX_{\geq 2}$. \end{proof}

\begin{remark}
The only point at which we used that the W-building was a proper cotangent building was when we asserted that the nucleus $(\sN,\lambda_\sN)$ is a stratified subset of dimension $\leq n-1$. The conclusion of Proposition \ref{prop:Leg-adjustment} therefore also applies to any W-building which satisfies that property. In particular the conclusion applies to any cotangent building. For arbitrary W-buildings the situation is less clear since we do not know that the potential is Morse or even Morse-Bott. However, one may certainly arrange that the nucleus $(\sN,\lambda_\sN)$ is a stratified subset of dimension $\leq n-1$ after a generic homotopy of the structure since under our definition it is always possible to Morsify potentials.
\end{remark}

  \subsection{Products and stabilizations of W-blocks}\label{sec:splitting-subcritical}
The product of two manifolds with corners (with corner structures) is naturally a manifold with corners (with the product corner structure), and likewise  the product of two W-blocks  is 
naturally a W-block
$$(\sX_1,\lambda_1)\times (\sX_2,\lambda_2)=(\sX_1\times \sX_2,\lambda_1\oplus\lambda_2)$$
with skeleton the product
$$\Skel (\sX_1,\lambda_1)\times \Skel (\sX_2,\lambda_2)=\Skel(\sX_1\times \sX_2,\lambda_1\oplus\lambda_2)$$
In particular, the product of proper cotangent blocks   is 
naturally a proper cotangent block
$$(\sT^*M_1,p_1 dq_1)\times (\sT^*M_2,p_2 dq_2)=(\sT^*(M_1 \times M_2),p_1 dq_1\oplus p_2 dq_2)
$$

In the context of W-blocks, there are  special cases of products to note and distinguish.
We will use the term  {\em stabilization} for  the product with  $(\sR^2,\frac12(pdq-qdp))$, and the term 
  {\em elongation} for the product
with $(\sT^*[0,1],udt)$.

    \begin{figure}[h]
\includegraphics[scale=0.3]{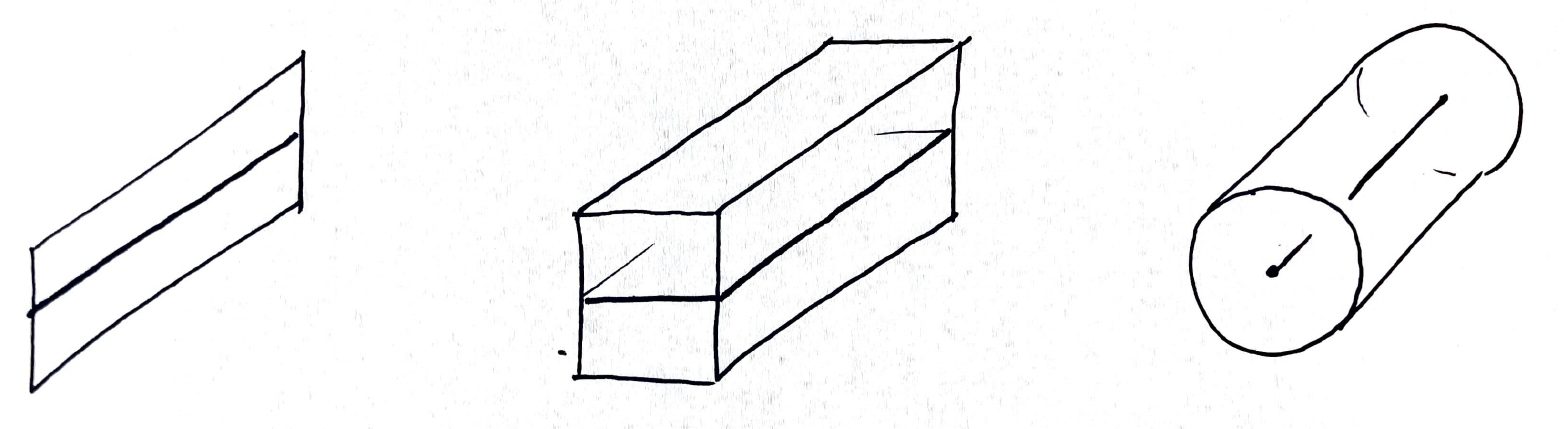}
\caption{The difference between elongation and stabilization.}
\label{elong-stab}
\end{figure}

 On the one hand, stabilization does not  change the collection of faces.  Each face  $\sP$ of $(\sX,\lambda)$ with nucleus $\sN$ yields a face $\sP\times \sR^2$ of $(\sX,\lambda)\times (\sR^2,\frac12(pdq-qdp))$ with nucleus $\sN\times \sR^2$.
 
 On the other hand, 
 elongation   associates to each   face $\sP$   of $(\sX,\lambda)$  with nucleus $\sN$   three  faces of
  $(\sX,\lambda)\times(\sT^*[0,1],udt)$ with nuclei of varying codimension: $\sP_{[0,1]}=\sP\times \sT^*I$ with  nucleus $\sN\times \sT^*I$, 
$\sP_0=\sP\times    \{t=0\}$ with  nucleus $\sN\times  \{t=0\}$, and  $\sP_1=\sP\times   \{t=1\}$ with  nucleus $\sN\times  \{t=1\}$.   
Note that (unlike elongation) the stabilization operation can also be applied to W-blocks with attracting boundary.

\subsection{Morse theory on W-blocks}
Let $(\sX,\lambda)$ be a W-block. Let $\wh\lambda:=\lambda^{\attract}$ be a modification of $\lambda$ which makes all faces attracting, see Definition \ref{def:attract}. We  will also assume that all zeros of $\lambda$ are non-degenerate.


Choose a potential   $\phi:\sX\to\RR$   for the Liouville field $Z$ of $\wh \lambda$ such that the level set $\{\phi=1\}$ is a defining domain for $(\sX,\wh\lambda)$, and all the critical values of $\phi$ are simple.

Note that  for every nucleus $\sN$ of a $j$-face $\sP\subset\p_j\sX$ the function $\phi|_{\sN}$ is Morse, and  thanks to the attracting condition hypothesis a critical point $p$ of $\phi|_{\sN}$ of index $i$ is a critical point of $\phi$ of index $i+j$.

Let $D^m$ be the unit disc in $\RR^m$. For each $k=0,\dots, m$ consider the  quadrant $$D_k^m:=D^m\cap\{x_1\geq 0,\dots,  x_{k}\geq 0\},$$ so that $D^m_0=D^m$. Then  $D_k^m$ is a manifold with corners, and   $\sT^*D_k^m=D_k^m\times\RR^m$ is a cotangent block. Denote $S_k^{m-1}:=\p D^{m}\cap\{x_1\geq 0,\dots,  x_{k}\geq 0\}$.
The next proposition is a ``cornered version" of the standard Weinstein handlebody decomposition statement for Weinstein manifolds.
\begin{prop}\label{prop:handle-attaching-n}  Let $\{\phi=c_1\}$ and $\{\phi=c_2\}$ be two regular levels of the potential $\phi$ such that the domain $\{c_1\leq\phi\leq c_2\}$ contains a unique critical point $p$ of index $m\leq n$. Suppose that
 $p\in \p_k\sX$, $k=0,\dots, n$,  (if $k=0$ this means that $p\notin \p \sX$). 
 Then  $\{\phi\leq c_2\}$ is the result of a vertical attachment of the cotangent block $\sT^*(D_k^m\times D^{n-m} )$ along  the    ribbon
of a    Legendrian embedding $S_k^{m-1}\times D^{n-m}\to \{\phi=c_1\}$.
\end{prop}
\begin{remark}  The Legendrian embedding $S_k^{m-1}\times D^{n-m}\to \{\phi=c_1\}$ is not pure when $m<n$.
Hence,  according to our convention from  Example \ref{lm:Leg-attach-ribbon}
 we vertically  attach in this case a generalized  cotangent block $\left(\sT^*(D_k^m\times D^{n-m})\right)^{P,\bigcirc}$  along a   ribbon
where  $P$  is the facet $D_k^m\times \p D^{n-m}$ of the Legendrian.
\end{remark}
 \begin{proof} This follows from the normal form for Weinstein handle attachment. 
 \end{proof}
 
 \subsection{Deforming a W-block into a cotangent building}\label{sec:block-to-bldg}
 \label{sec:building}
 \begin{thm}[Block to Building]
 \label{thm:WctoW} 
 Let $\sX$ be a $2n$-dimensional W-block. Then  up to homotopy
\begin{itemize}
\item [a)]   $\sX$ admits a  cotangent building  structure;
\item[b)] a descendant of $\sX$ admits a proper cotangent building structure.
\end{itemize}
 In the second case, we can take the  blocks  to be proper cotangent blocks of discs with corners, and in the first case, 
  we can take the  blocks to be 
   cotangent blocks of discs with corners, i.e.~proper cotangent blocks of discs with corners with some faces stripped.
\end{thm}

  \begin{proof} We first prove b)  by induction in dimension $2n$.
   The base of induction $n=0$ is trivial. Suppose that the statements holds for W-blocks of dimension $<2n$.
    Let $Z$ be the Liouville field of $\lambda$, $\wh Z$ its attracting boundary modification,  and $\phi:\sX\to\RR$ is a Morse potential. We can assume that there are regular values $c_1<c_2\in\RR$ such that  all critical values $>c_1$ have index $n$,   all critical values $<c_1$ have index $<n$ and there are no critical values $>c_2$.
    
    We will need the following proposition,  which is an analog of Cieliebak's theorem from  \cite{Ci02} (see  also \cite{CE12}):
 
\begin{prop}\label{thm:Cieliebak-subcrit}
 Suppose a  $2n$-dimensional  W-block with attracting boundary admits a Morse potential with critical points of index $<n$. Then  it is homotopic, after  possibly passing to a descendant block, to  a stabilization of another W-block with attracting boundary.
 \end{prop}  
  \begin{proof}    We argue by induction in the number of critical points.
  At  the induction step we need to show that if  
  $ \sX= \sY\times \sR^2$ is  the stabilization   a W-block  $\sY$, and  $\wh\sX$ is the result  of   vertical attaching to $\sX$ of a proper cotangent block   $\sT^*(D_k^m\times D^{n-m} )$ with $m<n$, $0\leq k\leq m$, then 
  $\wh \sX$ is deformation equivalent to  the stabilization  $\wh\sY\times\sR^2$ where the W-block $\wh \sY$ is the result of vertical attaching to $\sY$ of the proper cotangent block   $\sT^*(D^m_k\times D^{n-m-1} )$. We first sketch   Cieliebak's argument  for  the case $k=0$, and then explain minor adjustments  for the general case. 
   Let $W$ be a defining domain for $(\sY,\lambda_\sY)$, and $\xi|_W$   the contact structure on the vertical boundary $\p_vW$ induced by the Liouville form $\lambda_\sY$. Let $h:W\to\RR$ be  a potential normalized to have the constant maximal regular value $0$ on $\p W$. Then the function $H:=h+p^2+q^2$  on $W\times\RR^2$ is a potential for the stabilization $\sX=\sY\times \sR^2$ , and $V=\{H=0\}\subset W\times\R^2$ is a defining domain for $\sX$.  The vertical boundary $\p_vV$   carries the contact structure $\xi_V$ induced by the Liouville for $\lambda_{\sY}\oplus(pdq-qdp)$ and contains $(\p_vW,\xi_W)$ as its codimension 2 contact submanifold with a fixed trivialization $\mu^1,\mu^2\in\xi_V|_W$ of the  normal bundle. We have a canonical projection $\pi:\p_vV\to\p_vW$ such that $\pi|_{\p_v W}=\Id$ and for any $x\in \Int W$ the pre-image $\pi^{-1}(x)\subset\p_vV$ is diffeomorphic to a circle.
  The proper cotangent block
  $\sT^*(D^m\times D^{n-m} )$ is glued to $V$ along a ribbon of a Legendrian embedding  $\phi:S^{m-1}\times D^{n-m}\to\p_vV$. 
    Our goal is to find an isotropic  isotopy $ \phi_t:S^{m-1}\times D^{n-m}\to\p_vV$ $t\in[0,1]$ such that
  \begin{itemize} \item $  \phi_0= \phi$;
    \item $ \phi_1(S_k^{m-1})\times  D^{n-m-1}\subset  \p_vW$.   
  \end{itemize}
  Let $K:=\Skel(W)=\Skel(\sY)$ be the  skeleton of $\sY$. As $\dim W=2n-2$ the skeleton $K$  is a  $\leq(n-1)$-dimensional stratified subset of $W$. Then $\pi^{-1}(K)\subset \p_vV$ is a   stratified subset of dimension $\leq n $ of the $(2n-1)$-dimensional  contact manifold $\p_vV$.  On the other hand $\dim \phi(S^{m-1}\times 0)=m-1<n-1$. Hence, by a generic $C^\infty$-small isotopy  together with a shrinking   of the factor $D^{n-m}$ we can arrange that $\phi(S^{m-1}\times D^{n-m})\cap\pi^{-1}(K)=\varnothing$.  The Liouville projection $\ell:W\setminus K\to\p_vW$ lifts to a projection
  $\ol\ell:  \p_vV\setminus\pi^{-1} (K)\to\p_vW$, which yields a homotopy  $\wh\phi_t$ connecting  $\wh\phi_0= \phi$ with a  map into $\p_v W$. There exists a family of injective isotropic homomorphisms $\wh\Phi_t:T(S^{m-1}\times D^{n-m})\to \xi$  which covers the homotopy $\wh\phi_t$, such that $\wh\Phi_0=d\phi $. Moreover, we can arrange that  
  $ \wh\Phi_1(T(S^{m-1}\times D^{n-m-1}))\subset T\p_vW$. Indeed, the obstruction   to do that is    in $\pi_{m-1}(S^{2n-1})=0$.
  Now we use  Gromov's existence $h$-principle for Legendrian immersions to modify the embedding $\wh\phi_t$ to a Legendrian  regular homotopy   $\ol\phi_t: S^{m-1}\times D^{n-m}\to \p_vV$ such that $\ol\phi_0=\phi$ and $\ol\phi_1(S^{m-1}\times D^{n-m-1})\subset \p_vW$. By dimensional reasons  and shrinking if necessary the factor $D^{n-m}$, we can deform the Legendrian regular homotopy to a Legendrian isotopy with the same properties. 
  
  The case $k>0$ is similar except that one needs to argue inductively over the dimension of faces of the manifold with corners $S^{m-1}_k$, successively extending the Legendrian  isotopy to boundary faces of increasing dimension, and finally extending it the interior  of $S^{m-1}_k$.
   \end{proof} 
  We now continue our proof of Theorem  \ref{thm:WctoW}.
Inductively applying  Proposition ~\ref{thm:Cieliebak-subcrit} for critical points of $\phi$  with critical values $<c_1$ we conclude     that there is a  W-block $(\sY,\mu)$ of dimension $2n-2$ such that  the domain $\{\phi\leq c_1\}$ is a defining   domain for the stabilization 
$(\wt \sX,\wt \lambda ):=(\sY,\mu)\times(\sR^2,\frac12(udt-tdu))$.   Passing from $(\sR^2,\frac12(sdt-tds))$  to a descendant W-block   $(\sT^*I,sdt)$ we   get a  product W-block $ (\sY,\mu)\times (\sT^*I,sdt)$, descendant to  the block $\wt \sX$. 
 We will keep the notation $\wt\sX $ for the descendant block.
 By the induction  hypothesis,   there is a W-block descendant to $\sY$ which has a structure of  a cotangent building. This structure   survives   under the multiplication by the block $(\sT^*I,sdt)$. 
 Thus,
 $\wt\sX$ has a structure of proper cotangent building $\sX_k\to\cdots \to \sX_1$.

 According to Proposition \ref{prop:handle-attaching-n}  one gets the W-block $\sX$  by   vertically attaching to $\sX_{\leq n-1}$ proper  cotangent blocks $\sT^*D_k^n$, $k\in\{0,\dots, n\}$ along ribbons of  disjoint Legendrian embeddings $S_k^{n-1}\to\sX_{\leq n-1}\setminus\Skel(\sX_{\leq n-1})$.
  Using Proposition \ref{prop:Leg-adjustment} we can  make these embeddings compatible   with the  building structure of $\sX_{\leq n-1}=(\sX_{n-1}\to\cdots\to \sX_1)$. It remains to apply Lemma \ref{lm:bldg-after}   to conclude that a descendant block to $\sX$  has a proper cotangent building structure.
  
  \medskip
  The proof of case a) is more straightforward. Assuming that all critical values are distinct we induct over their numbers.
  Thus, the inductional step deals with a vertical attaching of a    cotangent block along a  Legendrian embedding $j:S_k^{n-1}\to\sX_{\leq m-1}\times D^{n-m}\setminus\Skel(\sX_{\leq n-1})$. If $m<n$ then the Legendrian embedding $j$ is not pure, and hence, we attach along it a non-proper cotangent block. Then we apply  Propositions \ref{prop:handle-attaching-n}, \ref{prop:Leg-adjustment}   and Lemma \ref{lm:bldg-after}, as in case b). On the other hand, we do not  need to use  in this case Theorem~\ref{thm:Cieliebak-subcrit}, and hence, do not have    to pass to a descendant block.
      \end{proof}
    \begin{remark}[Relative case]\label{rm:rel-case}
  Suppose that  $ \sX $ is a   W-block,  $W$ its defining domain, and $\phi:W\to\RR$ a potential  such that $\p W=\{\phi=c\}.$ Let $\wt c<c$ be a regular level of $\phi$ and $\wt\sX$  a $W$-block with the defining domain  $\wt W:=\{\phi\leq \wt c\}$.
If   $\wt\sX$  admits  a W-building structure
 $\sX_m\to\dots\to \sX_1$  then the proof of Theorem \ref{thm:WctoW}   yields a  cotangent building structure  $  \sX_{m+k}\to\dots  \sX_{m+1}\to\wt \sX_m\to\dots\to  \sX_1$  for $\sX$. If $\sX_m\to\dots\to  \sX_1$ is a proper cotangent building structure and all critical points  with critical values in $\{\wt c\leq\phi\leq c\}$ are of index $n$ then the construction yields the proper building structure   $  \sX_{m+k}\to\dots  \sX_{m+1}\to \sX_m\to\dots\to  \sX_1$  for $\sX$. However, if $\{\wt c\leq\phi\leq c\}$ contains critical  values of points of index $ 0<j<n$ then the proper cotangent building structure can be extended  to $\sX$ only at expense of 
 {\em rounding the handles}, i.e.  adding proper cotangent blocks of the form $\sT^*(S^{j-1}\times D^{n-j+1})\setminus \Int D^n$, attached along Legendrian embeddings of $S^{j-1}\times \p D^{n-j+1}$.
 \end{remark}

\end{spacing}




  
 




\end{document}